\documentclass[11pt]{amsart}
\usepackage[T1]{fontenc}
\usepackage{txfonts}
\usepackage{float}
\usepackage{amssymb,verbatim}
\usepackage[all]{xy}
\usepackage{subfig}
\usepackage{tikz}
\usepackage{color}
\usepackage{a4wide}
\usepackage{mathrsfs}
\usepackage{longtable}
\usepackage{hyperref}
\usepackage{wrapfig}
\usetikzlibrary{arrows}
\DeclareMathAlphabet{\mathpzc}{OT1}{pzc}{m}{it}
\newcommand{\R}{\mathbb{R}}
\newcommand{\C}{\mathbb{C}}
\newcommand\Z{\mathbb{Z}}
\newcommand{\N}{\mathbb{N}}
\newcommand{\Q}{\mathbb{Q}}
\renewcommand{\S}{\mathbb{S}}

\newcommand{\Gb}{\mathbf{G}}
\newcommand{\Hb}{\mathbf{H}}

\newcommand{\Mb}{\mathbf{M}}
\newcommand{\Pb}{\mathbb{P}}

\newcommand{\Tb}{\mathbf{T}}

\newcommand{\Ub}{\mathbf{U}}
\newcommand{\Vb}{\mathbf{V}}

\newcommand{\Zb}{\mathbf{Z}}
\newcommand{\Hbb}{\mathbb{H}}

\newcommand{\Lbb}{\mathbb{L}}

\newcommand{\Vbb}{\mathbb{V}}
\newcommand{\Wbb}{\mathbb{W}}

\newcommand{\bb}{\mathbf{b}}
\newcommand{\cc}{\mathbf{c}}

\renewcommand{\gg}{\mathbf{g}}

\newcommand{\hh}{\mathbf{h}}

\newcommand{\kk}{\mathbf{k}}

\renewcommand{\ss}{\mathbf{s}}
\newcommand{\uu}{\mathbf{u}}
\newcommand{\vv}{\mathbf{v}}
\newcommand{\xx}{\mathbf{x}}

\newcommand{\Acal}{\mathcal{A}}
\newcommand{\Bcal}{\mathcal{B}}

\newcommand{\Fcal}{\mathcal{F}}
\newcommand{\Gcal}{\mathcal{G}}
\newcommand{\Hcal}{\mathcal{H}}

\newcommand{\Rcal}{\mathcal{R}}
\newcommand{\Scal}{\mathcal{S}}

\newcommand{\Vcal}{\mathcal{V}}
\newcommand{\Wcal}{\mathcal{W}}

\newcommand{\Zcal}{\mathcal{Z}}
\newcommand{\Bc}{\mathcal{B}}

\newcommand{\Wc}{\mathcal{W}}
\newcommand{\Xc}{\mathcal{X}}
\newcommand{\CC}{\mathscr{C}}

\newcommand{\Aut}{{\rm Aut}}
\newcommand{\Ad}{{\rm Ad}}

\newcommand{\id}{\mathrm{id}}
\newcommand{\Id}{\mathrm{Id}}

\newcommand{\CP}{\mathbb{P}_{\mathbb{C}}}
\newcommand{\SL}{{\rm SL}}
\newcommand{\GL}{{\rm GL}}
\newcommand{\Mod}{\mathrm{Mod}}

\newcommand{\MCG}{\mathrm{Mod}}

\newcommand{\rk}{\mathrm{rk}}

\newcommand{\PB}{\mathrm{PB}}
\newcommand{\U}{\mathrm{U}}
\newcommand{\PU}{\mathrm{PU}}
\newcommand{\SU}{\mathrm{SU}}
\newcommand{\SO}{\mathrm{SO}}
\newcommand{\Sp}{\mathrm{Sp}}
\newcommand{\su}{\mathbf{su}}
\newcommand{\ssl}{\mathbf{sl}}

\newcommand{\real}{\mathrm{Re}}

\newcommand{\Tr}{\mathrm{Tr}}
\newcommand{\adj}{\mathrm{ad}}
\newcommand{\Span}{\mathrm{Span}}
\renewcommand{\Im}{\mathrm{Im}}
\renewcommand{\Re}{\mathrm{Re}}
\newcommand{\prim}{\mathrm{prim}}
\newcommand{\gf}{\mathfrak{g}}
\newcommand{\pf}{\mathfrak{p}}
\newcommand{\fh}{\mathfrak{h}}
\newcommand{\Hom}{\mathrm{Hom}}
\newcommand{\pdroots}{\mathbf{U}_{d,\mathrm{prim}}}
\newcommand{\hb}{\hat{b}}
\newcommand{\hg}{\hat{g}}

\newcommand{\hX}{\hat{X}}
\newcommand{\hY}{\hat{Y}}
\newcommand{\hZ}{\hat{Z}}

\newcommand{\hk}{\hat{k}}

\newcommand{\hS}{\hat{S}}
\newcommand{\vide}{\varnothing}

\newcommand{\eps}{\epsilon}

\newcommand{\ol}{\overline}

\newtheorem{Theorem}{Theorem}[section]
\newtheorem{Corollary}[Theorem]{Corollary}
\newtheorem{Lemma}[Theorem]{Lemma}
\newtheorem{Proposition}[Theorem]{Proposition}

\newtheorem{Definition}[Theorem]{Definition}

\newtheorem{Claim}[Theorem]{Claim}

\newtheorem{MainTheorem}{Theorem}

\theoremstyle{remark}
\newtheorem{Remark}[Theorem]{Remark}
\begin{document}
\title[Representations of braid groups via cyclic covers]{Representations of braid groups via cyclic covers of the sphere: Zariski closure and arithmeticity}

\author[G.~Menet and D.-M. Nguyen]{Gabrielle Menet and Duc-Manh Nguyen}

\address{Département de Mathématiques, CNRS UMR 7013\\
UFR Sciences et Techniques, Université de Tours\\
Parc de Grandmont\\
37200 Tours\\
France}
%\address{Institut de Math\'ematiques de Bordeaux IMB\\
%Universit\'e de Bordeaux, B\^at. A33\\
%351, cours de la Lib\'eration \\
%33405 Talence Cedex\\
%France}

%\email[V.~Koziarz]{vincent.koziarz@math.u-bordeaux.fr}
\email[G.~Menet]{menet.gabrielle@gmail.com}
\email[D.-M.~Nguyen]{duc-manh.nguyen@univ-tours.fr}

%\author{Vincent Koziarz}
%\address{Univ. Bordeaux, IMB, CNRS, UMR 5251, F-33400 Talence, France}
%\address{Institut de Math\'ematiques de Bordeaux IMB\\
%Universit\'e de Bordeaux, B\^at. A33\\
%351, cours de la Lib\'eration \\
%33405 Talence Cedex\\
%France}

\date{\today}
\begin{abstract}
Let $d \geq 2$ and $n\geq 3$ be two natural numbers. Given any sequence $\kappa=(k_1,\dots,k_n) \in \mathbb{Z}^n$ such that $\gcd(k_1,\dots,k_n,d)=1$,
we consider the representation of the pure braid group $\mathrm{PB}_n$ via the monodromy of the family of compact Riemann surfaces obtained from the plane curves $y^d=\prod_{i=1}^n(x-b_i)^{k_i}$ over the space $(\mathbb{C}^n)^*$, where $(\mathbb{C}^n)^*=\{(b_1,\dots,b_n)\in \mathbb{C}^n, \; b_i\neq b_j \text{ if } i\neq j\}$.
By restricting to a specific subspace of the cohomology of the fiber, we obtain a representation $\rho_d$ of $\mathrm{PB}_n$ into a linear algebraic group defined over $\mathbb{Q}$. The representation $\rho_d$ is of interest because of its primitivity with respect to  the data $d,n,\kappa$.
The first main result of this paper is a criterion for the Zariski closure of the image of $\rho_d$ to be maximal, and the second main result is a criterion for the image to be an arithmetic lattice in the target group. The latter generalizes previous results by Venkataramana~\cite{Venky:Annals, Venky:Invent}, and gives an answer to a question by McMullen~\cite{McM13}.
%%To prove the second main result, we use an arithmeticity criterion which was conjectured by Margulis and proved in full generality by Benoist and Miquel~\cite{BM:Duke}.
\end{abstract}

%\begin{abstract}
%Let $d \geq 2$ and $n\geq 3$ be two natural numbers. Given any sequence $\kappa=(k_1,\dots,k_n) \in \mathbb{Z}^n$ such that $\gcd(k_1,\dots,k_n,d)=1$, we consider the family of compact Riemann surfaces obtained from the plane curves defined by $y^d=\prod_{i=1}^n(x-b_i)^{k_i}$, where $\{b_1,\dots,b_n\}$ are $n$ distinct points in $\C$.
%The cohomology with integral coefficients of the fibers of this family gives rise to a $\Z$-local system over $(\C^n)^*$, where $(\C^n)^*=\{(b_1,\dots,b_n)\in \C^n, \; b_i\neq b_j \text{ if } i\neq j\}$.
%The monodromy of this local system  provides us with a representation  of the pure braid group $\mathrm{PB}_n$ into some symplectic group. By restricting to a specific subspace of the cohomology of the fiber, we obtain a representation $\rho_d$ of $\mathrm{PB}_n$ into a reductive linear algebraic group defined over $\Q$. The representation $\rho_d$ is of interest because of its primitivity with respect to  the data $d,n,\kappa$.
%The first main result of this paper is a criterion for the Zariski closure of the image of $\rho_d$ to be maximal, and the second main result is a criterion for the image to be an arithmetic lattice in the target group. The latter generalizes previous results by Venkataramana~\cite{Venky:Annals, Venky:Invent}, and gives an answer to a question by McMullen~\cite{McM13}.
%%To prove the second main result, we use an arithmeticity criterion which was conjectured by Margulis and proved in full generality by Benoist and Miquel~\cite{BM:Duke}.
%\end{abstract}

\maketitle

\section{Introduction}\label{sec:intro}
\subsection{Cyclic coverings of the sphere and representations of pure braid groups}\label{subsec:cyclic:covers:intro}
Let $d\geq 2$ and $n \geq 3$ be  natural numbers. Let us fix a tuple of $n$ natural numbers $\kappa:=(k_1,\dots,k_n),  \; k_i \in \Z_{\geq 0}$.
Given an $n$-tuple $\bb:=(b_1,\dots,b_n)$ of pairwise-distinct points in the complex plane $\C$, we are interested in the curve $\CC_{\bb,\kappa} \subset \C^2$ defined by the equation
\begin{equation}\label{eq:plane:curve}
(\CC_{\bb,\kappa}) \qquad y^d=\prod_{i=1}^n (x-b_i)^{k_i}
\end{equation}
Denote by $\ol{\CC}_{\bb,\kappa}$ the closure of $\CC_{\bb,\kappa}$ in $\CP^2$. Normalizing $\ol{\CC}_{\bb,\xx}$, we obtain a compact Riemann surface $\hX_{\bb,\kappa}$.
It is a well known fact that $\hX_{\bb,\kappa}$ is  connected  if and only if $\gcd(k_1,\dots,k_n,d)=1$. It is worth noticing that there is no loss in supposing  that $0 < k_i < d$, for all $i=1,\dots,n$ (c.f. Lemma~\ref{lm:equiv:exponents}).

Let $\zeta_d:=e^{\frac{2\pi\imath}{d}}$.
The map $(x,y) \mapsto (x,\zeta_d\cdot y)$ gives an  automorphism of $\CC_{\bb,\kappa}$, which extends naturally to an automorphism of $\hX_{\bb,\kappa}$.
We denote this automorphism  by $T$. Note that $T$ has order $d$.

The projection $\CC_{\bb,\kappa} \to \C, \; (x,y) \mapsto x$, extends to a ramified cover $\pi: \hX_{\bb,\kappa} \to \CP^1$ branched over the set $\{b_1,\dots,b_n\}\cup\{\infty\}$.
The covering $\pi$ induces an isomorphism $\hX_{\bb,\kappa}/\langle T \rangle \simeq \CP^1$. For this reason, $\hX_{\bb,\kappa}$ is called a {\em cyclic cover of $\CP^1$}.

%For these reasons, we will make the following assumptions throughout this paper
%\begin{equation}\label{eq:hypothese:ki}
%1 \leq k_i \leq d-1, \quad \forall i=1,\dots,n, \quad \text{ and  } \quad \gcd(k_1,\dots,k_n,d)=1.
%\end{equation}

Fix $\kappa=(k_1,\dots,k_n)$, and let $\bb$ vary in $(\C^n)^*:=\{(b_1,\dots,b_n) \in \C^n, \; b_i\neq b_j \text{ if } i\neq j\}$, we obtain a family  of compact Riemann surfaces parametrized by $(\C^n)^*$. This means that there is a complex manifold $\hat{\Xc}_\kappa$ and a surjective holomorphic map $p: \hat{\Xc}_\kappa \to (\C^n)^*$ such that the fiber of $p$ at every point $\bb\in (\C^n)^*$ is isomorphic to $\hX_{\bb,\kappa}$.
The cohomologies with integer coefficients of the fibers of $p$ form a $\Z$-local  system $\Hb(\kappa)$ over $(\C^n)^*$ whose fiber over a point $\bb\in (\C^n)^*$ is identified with $H^1(\hX_{\bb,\kappa},\Z)$. Define $\Hb_\C(\kappa):=\Hb(\kappa)\otimes_\Z \C$. Then $\Hb_\C(\kappa)$ is a flat complex vector bundle over $(\C^n)^*$ with fiber over $\bb\in (\C^n)^*$ being $H^1(\hX_{\bb, \kappa}):=H^1(\hX_{\bb,\kappa},\C)$.
The family $\hat{\Xc}_\kappa$ comes equipped with an automorphism $\Tb$ of  order $d$ whose restriction to each fiber of $p$ is identified with $T$.

\medskip

Let $\Ub_d:=\{e^{\frac{2\pi\imath k}{d}}, \; k=0,\dots,d-1\}$ be the set of $d$-th roots of unity.  For each $q\in \Ub_d$, let $H^1(\hX_{\bb,\kappa})_q$ denote the $q$-eigenspace of the action $T^*$ on $H^1(\hX_{\bb,\kappa})$.
We get a direct sum decomposition of vector bundles over $(\C^n)^*$
$$
\Hb_\C(\kappa):=\oplus_{q\in \Ub_d} \Hb_\C(\kappa,q)
$$
where $\Hb_\C(\kappa,q)$ is the flat subbundle of $\Hb_\C(\kappa)$ whose fiber over $\bb \in (\C^n)^*$ is identified with $H^1(\hX_{\bb,\kappa})_q$.

Recall that the fundamental group of the space $(\C^n)^*$ is naturally identified with the pure braid group $\PB_n$.
Let $\rho$ and $\rho_q$ be  the monodromy representations of $\pi_1((\C^n)^*)\simeq \PB_n$ associated  with the bundles $\Hb_\C(\kappa)$ and $\Hb_\C(\kappa, q)$ respectively.

The representation $\rho$ is a particular case of a general construction of monodromy representation associated with families of projective varieties. In \cite{GrSch75} Griffiths and Schmid asked whether the images of such representations are always arithmetic (recall that a subgroup $\Gamma \subset \GL(N,\Z)$ is {\em arithmetic} if $\Gamma$ has finite index in the subgroup of integral points of its Zariski closure).
It follows from the results of Deligne-Mostow~\cite{DM86} that the answer is negative in general. However, it is always challenging to determine the images of such representations and to know whether these images are as large as one may expect.

Let us fix a point $\bb \in (\C^n)^*$  and for simplicity we write $\hX$ instead of $\hX_{\bb,\kappa}$.
Let $(.,.)$ denote the standard intersection form on $H^1(\hX,\R)$ which is given by
$$
(\alpha,\beta)=\int_{\hX}\alpha\wedge\beta
$$
for all closed $1$-forms $\alpha,\beta$ on $\hX$. We extend $(.,.)$ to $H^1(\hX):=H^1(\hX,\C)$ by linearity. Let $\langle.,.\rangle$ be the Hermitian form on $H^1(\hX)$ given by
$$
\langle \mu,\eta\rangle :=\frac{\imath}{2}\int_{\hX}\mu\wedge\bar{\eta}=\frac{\imath}{2}(\mu,\bar{\eta})
$$
for all $\C$-valued closed $1$-forms $\mu,\eta$ on $\hX$. Note that  $\langle.,.\rangle$ is a Hermtian form  of signature $(g,g)$ where $g$ is the genus of $\hX$. We will call $\langle.,.\rangle$ the intersection form on $H^1(\hX)$.

\medskip

%Throughout this paper,  whenever we mention a subgroup of $H^1(\hX,\C)$ without indicating the ring to which the coefficients belong we will mean the associated complex vector subspace of  $H^1(\hX,\C)$.

Fixing a symplectic basis of $H^1(\hX,\Z)$, we can  view $T^*$ as a matrix of order $d$ in $\Sp(2g,\Z)$. Denote by $\Sp(\hX,\R)^T$ the group of endomorphisms of $H^1(\hX,\R)$ that preserve $(.,.)$ and commute with $T^*$. By construction, we have $\rho(\PB_n)\subset \Sp(\hX,\Z)^T$ (here $\Sp(\hX,\Z)^T$ is the subgroup of $\Sp(\hX,\R)^T$ that preserves the lattice $H^1(\hX,\Z)$).

\medskip
For each $q\in \Ub_d$, let $H^1(\hX)_q:=\ker(T^*-q\Id) \subset H^1(\hX)$.
The space $H^1(\hX)$ admits a natural  splitting as follows: for each $e \in \Z_{\geq 2}$,  let $\phi_e$ be the $e$-th cyclotomic polynomial.
The roots $\phi_e$ are the primitive $e$-th roots of unity.
For all  $e  \; | \; d$, we define $H^1(\hX)_e:=\ker \phi_e(T^*) \subset H^1(\hX)$.
Then $H^1(\hX)_e$ has a basis in $H^1(\hX,\Z)$, and we have
\begin{equation}\label{eq:decomp:H:1:X}
 H^1(\hX)=\bigoplus_{e \,  | \, d}H^1(\hX)_e, \quad \text{ with } \quad    H^1(\hX)_e:=\bigoplus_{q \in \Ub_{e,{\rm prim}}} H^1(\hX)_q,
\end{equation}
where $\Ub_{e,{\rm prim}}$ is the set of primitive $e$-th roots of unity.
Note that we have $H^1(\hX)_1\simeq H^1(\CP^1) \simeq \{0\}$.
For all $\tau \in \PB_n$, since $\rho(\tau)$ commutes with $T^*$, $\rho(\tau)$ preserves the decomposition \eqref{eq:decomp:H:1:X}.

Define $\rho_e(\tau):=\rho(\tau)_{|H^1(X)_e}$. Consider the   Riemann surface $\hX_e$  obtained from the equation $y^e=\prod_{i=1}^n(x-b_i)^{k_i}$.
Then $\hX$  is a cyclic cover of $\hX_e$ with the covering map given by $(x,y) \mapsto (x,y^{\frac{d}{e}})$.
It is not difficult to see that for all $q\in \Ub_e$, $H^1(\hX)_q \simeq H^1(\hX_e)_q$, and the representations $\rho_e$ associated to $\hX$ and to $\hX_e$ are the same (cf. Theorem~\ref{th:Menet:alphaij}).
For this reason, we will only focus on the case $e=d$.

\medskip

Let $\Sp(\hX,\R)_d^T$  be the group of endomorphisms of $H^1(\hX,\R)_d:=H^1(\hX)_d\cap H^1(\hX,\R)$ that preserve the intersection form $(.,.)$ and  commute with $T^*$.
Since the restriction of $(.,.)$ as well as the action of $T^*$ on $H^1(\hX,\R)_d$ are given by rational matrices, it follows that $\Sp(\hX,\R)_d^T$ is a linear algebraic group defined over $\Q$. Denote by $\Sp(\hX,\Z)^T_d$ the subgroup of integral points in $\Sp(\hX,\R)^T_d$.
A natural question one may ask about the representation $\rho_d$ is
\begin{flushleft}
(Q1) \hspace{1cm} How large is $\rho_d(\PB_n)$ in $\Sp(\hX,\R)^T_d$?
\end{flushleft}
By construction, for all $\tau\in \Pb_n$, $\rho_d(\tau)$ is given by a matrix with rational coefficients in a basis of $H^1(\hX)_d$ formed by elements of $H^1(\hX,\Z)$. In \cite{McM13}, McMullen also asked
\begin{flushleft}
(Q2) \hspace{1cm} When does  $\rho_d(\PB_n)$ contain a finite index subgroup of $\Sp(\hX,\Z)_d^T$?
\end{flushleft}
In the case $d=2$, $\hX$ is a hyperelliptic curve, and we have $\Sp(\hX,\R)_2^T \simeq \Sp(2g,\R)$. In \cite{Acampo} A'Campo showed that in this case $\rho_2(\PB_n)$ always contains the congruence subgroup modulo 2 in $\Sp(2g,\Z)$.
For $d\geq 3$, let $q=e^{-\frac{2\pi\imath}{d}}$ and $\U(H^1(\hX)_q)$ be the group of automorphisms of $H^1(\hX)_q$ that preserve the intersection form $\langle.,.\rangle$.
In their celebrated work \cite{DM86, Mos86} Deligne and Mostow gave some sufficient conditions on the set of numbers $\{\frac{k_1}{d}, \dots,\frac{k_n}{d}\}$ such that $\rho_q(\PB_n)$ is a lattice in $\U(H^1(\hX)_q)$.
Under these conditions, they also gave a criterion for $\rho_d(\PB_n)$ to be commensurable to $\Sp(\hX,\Z)_d^T$ (see~\cite[\textsection 12]{DM86}).
%giving an answer to (Q2).

In the case $\kappa=(1,\dots,1)$, McMullen~\cite{McM13} gave the answer to (Q2) for several pairs $(d,n)$. In \cite{Venky:Annals, Venky:Invent}  Venkataramana proved a criterion for $\rho_d(\PB_n)$ to be an arithmetic subgroup of $\Sp(\hX, \R)^T_d$ in the case  $\gcd(k_i,d)=1$ for all $i=1,\dots,n$. In particular, he showed that under this hypothesis, $\rho_d(\PB_n)$ is always an arithmetic lattice of $\Sp(\hX,\R)_d^T$ if $n\geq 2d$.

\subsection{Statement of the results}\label{subsec:main:results}
The goal of this paper is to address both questions (Q1) and (Q2) in full  generality,  that is without the assumption that $\gcd(k_i,d)=1$ for all $i=1,\dots,n$. Since the case $d=2$ has been completely solved by A'Campo, we will suppose that $d\geq 3$.

\begin{Definition}\label{def:good:weights}
Let $\mu:=(\mu_1,\dots,\mu_n)$ be a sequence of $n$ rational numbers where $\mu_i\in \Q\cap (0;1), \; i=1,\dots,n$. The sequence $\mu$ is said to be {\em good} if one of the following holds
\begin{itemize}
\item[(a)] $ 1 < \mu_1+\dots+\mu_n < n-1$,

\item[(b)] $ \mu_1+\dots+\mu_n \leq 1$ or $\mu_1+\dots+\mu_n \geq n-1$, and there is a triple of  indices $\{i, j, k\} $ in $\{1,\dots,n\}$ such that
\begin{itemize}
\item[$\bullet$] the order of $e^{2\imath\pi(\mu_i+\mu_j)}$ is greater than $5$,

\item[$\bullet$] at least one of the orders of $e^{2\imath\pi(\mu_i+\mu_k)}$ and $e^{2\imath\pi(\mu_j+\mu_k)}$ is greater than $2$.
\end{itemize}
\end{itemize}
\end{Definition}

\begin{Remark}\label{rk:good:symmetric}
Observe that $\mu$ is good if and only if $\bar{\mu}:=(1-\mu_1,\dots,1-\mu_n)$ is good.
\end{Remark}
%For all $q \in \Ub_d$, let $(r_q,s_q)$ be the signature of the restriction of the intersection form $\langle.,.\rangle$ to $H^1(\hX)_q$ (cf. Theorem~\ref{th:Menet:dim:signature}). Note that the signature $(r_{\bar{q}},s_{\bar{q}})$ of $\langle.,.\rangle$ on $H^1(\hX)_{\bar{q}}$ satisfies $r_{\bar{q}}=s_q$ and $s_{\bar{q}}=r_q$.

Let $\Ub^+_{d,\prim}$ be the set of primitive $d$-th roots $q$ of unity with $\Im(q) >0$. Denote by $H^1(\hX,\R)_{2\Re(q)}$  the $2\Re(q)$-eigenspace of  $(T^*+T^*{}^{-1})$ in $H^1(\hX,\R)$.
For $d\geq 3$, the space $H^1(\hX, \R)_d$ admits the following splitting
$$
H^1(\hX,\R)_d=\bigoplus_{q\in \Ub^+_{d,\prim}} H^1(\hX,\R)_{2\Re(q)},
$$
The group of endomorphisms of $H^1(\hX,\R)_{2\Re(q)}$ that preserve the bilinear form  $(.,.)$ and commute with $T$ is isomorphic to $\U(H^1(\hX)_q)$. As a consequence, we have (see \cite[\textsection 7]{McM13} for more details)
\begin{equation}\label{eq:inv:subgrp:factor}
\Sp(\hX,\R)^T_d \simeq \prod_{\substack{q \in \Ub^+_{d,{\rm prim}}}}
\U(H^1(\hX)_q) \simeq \prod_{\substack{q \in \Ub^+_{d,{\rm prim}}}} \U(r_q,s_q)
\end{equation}
where $(r_q,s_q)$ is the signature of the restriction of the intersection form $\langle.,.\rangle$ to $H^1(\hX)_q$.
Observe that for all $q\in \Ub^+_{d, \, {\rm  prim}}$, the composition of $\rho$ with the projection to the $\U(H^1(\hX)_q)$ factor in \eqref{eq:inv:subgrp:factor} gives the representation $\rho_q$.
%Note also that  $\U(H^1(\hX)_q)$ and $\U(H^1(\hX)_{\bar{q}})$ are isomorphic, and the representations $\rho_q$ and $\rho_{\bar{q}}$ are complex conjugate.

%We can now state the main results of this paper
Our first main result provides a simple arithmetic criterion for the image of $\rho_d$ to have maximal Zariski closure.
\begin{MainTheorem}\label{th:main:Zariski:dense}
Let $d \geq 3$, $n\geq 3$ be two natural numbers, and $\kappa=(k_1,\dots,k_n) \in (\Z_{\geq 1})^n$ be a sequence of integers such that $1\leq k_i \leq d-1$ for all $i=1,\dots,n$, and $\gcd(k_1,\dots,k_n,d)=1$.  Define
\begin{equation}\label{eq:epsilon}
\eps_0=\left\{
\begin{array}{cl}
1 & \text{ if } d \, | \,(k_1+\dots+k_n),\\
0 & \text{ otherwise}.
\end{array}
\right.
\end{equation}
Let $\hat{\Gb}$ denote the Zariski closure of $\rho_d(\PB_n)$ in $\Sp(\hX, \R)^T_d$, and $\hat{\Gb}^0$ its identity component. Assume that  for all $k \in \Z$ such that $\gcd(k,d)=1$, the sequence $\mu^{(k)}:=(\{ \frac{kk_1}{d}\},\dots,\{\frac{kk_n}{d}\})$, where $\{\frac{kk_i}{d}\}=\frac{kk_i}{d} -\lfloor \frac{kk_i}{d}\rfloor$, is good in the sense of Definition~\ref{def:good:weights}. Then we have
\begin{equation}\label{eq:id:comp:Z:closure}
\hat{\Gb}^0 \simeq \prod_{\substack{q\in \Ub^+_{d,{\rm prim}}}}\SU(r_q,s_q).
\end{equation}
if either
\begin{itemize}
\item[$\bullet$] $n-1-\eps_0 \geq 3$, or

\item[$\bullet$] $n-1-\eps_0=2$ and there exist $1 \leq i < j \leq n$ such that $\gcd(k_i+k_j,d)=1$.
\end{itemize}
\end{MainTheorem}

\begin{Remark}\label{rk:Zar:density:maximal}
In the course of the proof of Theorem~\ref{th:main:Zariski:dense}, we will see that $\hat{\Gb}^0$ cannot be larger than the right hand side of \eqref{eq:id:comp:Z:closure}.
\end{Remark}

In Theorem~\ref{th:Zar:density:on:factor}, we will show that if the vector $\mu^{(k)}:=(\{ \frac{kk_1}{d}\},\dots,\{\frac{kk_n}{d}\})$ satisfies one of the conditions in Definition~\ref{def:good:weights}, then the Zariski closure of $\rho_q(\PB_n)$ contains $\SU(H^1(\hX)_q)$ as a finite index subgroup, which is clearly a necessary condition for \eqref{eq:id:comp:Z:closure}.   Theorem~\ref{th:main:Zariski:dense} actually follows from the following

\begin{MainTheorem}\label{th:main:Zar:dense:bis}
If the Zariski closure of $\rho_q(\PB_n)$ contains $\SU(H^1(\hX)_q)$ for all $q\in \Ub_{d, \; {\rm prim}}$, then \eqref{eq:id:comp:Z:closure} holds provided for some $q\in \Ub_{d,\, {\rm prim}}$ (and hence for all $q \in \Ub_{d, \, {\rm prim}}$) we have  either $\dim H^1(\hX)_q \geq 3$, or $\dim H^1(\hX)_q=2$ and there exists a pair of indices $1 \leq i< j\leq n$ such that $\gcd(k_i+k_j,d)=1$.
\end{MainTheorem}

Our third main result addresses Question (Q2).

\begin{MainTheorem}\label{th:main:arithm}
Let $d \geq 3$ be a natural number. Let $\kappa=(k_1,\dots,k_n)$ a sequence of integers such that $1 \leq k_i \leq d-1$ for all $i=1,\dots,n$, and 
$n+1-\eps_0 \geq 5$, where $\eps_0$ is defined in \eqref{eq:epsilon}.
Suppose that there is a proper subset $I$ of $\{1,\dots,n\}$ which satisfies
\begin{itemize}
\item[(i)] $d \; | \; \sum_{i\in I} k_i$,

\item[(ii)] $\gcd(d, \{k_i, \; i \in I\})=1$ if $|I|\geq 3$,

\item[(iii)] $\gcd(d, \{k_i, \; i \in  I^c\})=1$ if $|I| \leq n-2-\eps_0$. 
\end{itemize}
Then $\rho_d(\PB_n)$ is commensurable to $\Sp(\hX,\Z)^T_d$ provided either
\begin{itemize}
\item[(a)] $d\not\in \{3,4,6\}$, or

\item[(b)] $d\in \{3,4,6\}$ and $2 < \frac{k_1}{d}+\dots+\frac{k_n}{d} < n-2$.
\end{itemize}
\end{MainTheorem}
\begin{Remark}\label{rk:arithmetic:n:cond:ii}\hfill
\begin{itemize}
\item[$\bullet$] Note that we have $2 \leq |I| \leq n-1$. The condition $n+1-\eps_0\geq 5$ implies that at least one of (ii) and (iii) is realized. In particular, we always have $\gcd(d,k_1,\dots,k_n)=1$. 

\item[$\bullet$] The conditions (ii) and (iii) are crucial to Theorem~\ref{th:main:arithm}, without these two conditions, its conclusion does not hold. Indeed, it follows from the results of Deligne and Mostow (see \cite[p. 86]{DM86}) that in the case $d=12, n=5,  \kappa=(7,5,4,4,4)$, $\rho_{12}(\PB_5)$ is not arithmetic, while the sequence $\kappa$ satisfies (i) with $I=\{3,4,5\}$, but not (ii). The same phenomenon occurs in the cases $ d= 12, n=5, \kappa=(7,6,5,3,3)$ and $d=12, n=6,  \kappa=(7,5,3,3,3,3)$.
\end{itemize}
\end{Remark}

In \cite{Venky:Invent} Venkataramana showed Theorem~\ref{th:main:arithm} in the case $\gcd(k_i,d)=1$ for all $i=1,\dots,n$ (in this case the conditions (ii) and (iii) are automatically satisfied).
Geometrically, this condition means that the preimage of $b_i$ in $\hX$ consists of a single point.
In~\cite{Venky:Invent}, it is shown that in this case the representation $\rho_q$ factors through the Gassner representation of the pure braid group $\PB_n$. The arithmeticity of $\rho_q(\PB_n)$ follows from a criterion (due to Ragunathan and Venkataramana) for a subgroup of a linear algebraic group $G$ defined over a number field $K$ to be commensurable to $G(O_K)$, where $O_K$ is the ring of integers in $K$.
The Zariski closure of $\rho_d(\PB_n)$ was not addressed in the works \cite{Venky:Annals, Venky:Invent}.

%This criterion asserts that if $G$ has real rank at least $2$ and  $\Gamma$ is a subgroup of $G(O_K)$, which intersects two opposites horospherical subgroups $U^+(O_K)$ and $U^-(O_K)$ in finite index subgroups of $U^+(O_K)$ and $U^-(O_K)$ respectively, then $\Gamma$ has finite index in $G(O_K)$ (c.f. \cite[Th. 4]{Venky:Invent} or \cite[Cor. 1]{Venky:Annals}).

In this paper, we use a more geometric approach to investigate the representations $\rho_q$ and $\rho_d$. Specifically, we will identify $\PB_n$ with the mapping class group of the $n$-punctured disc, and the relevant properties of $\rho_q$ (and hence of $\rho_d$) are obtained by using tools from  differential topology. This allows us to lift the condition on the cardinality of preimage of $b_i$.
Moreover, to prove the arithmeticity of the image of $\PB_n$, we employ a different strategy by using a criterion which is related to but different  from the one in \cite{Venky:Annals, Venky:Invent}. This was a conjecture of Margulis and  proved in full generality by Benoist and Miquel~\cite{BM:Duke} (see \textsection\ref{sec:horospherical} for more details).
An important feature of this criterion is that one first needs to show that the identity component of the Zariski closure of $\rho_d(\PB_n)$ is equal to $\prod_{q\in \Ub^+_{d,\prim}}\SU(H^1(\hX)_q)$. In particular Theorem~\ref{th:main:Zariski:dense} is required in the proof of Theorem~\ref{th:main:arithm}.

\medskip

Given $N\in \N$, a subgroup $\Gamma$ of $\GL(N,\Z)$ is said to be {\em thin} if it has infinite index in the subgroup $\mathcal{G}(\Z)$ of integral points in the Zariski closure $\mathcal{G}$ of $\Gamma$. For a more detailed account on this concept and related problems we refer to \cite{Sarnak:survey, FMS:JEMS2014}. Theorem~\ref{th:main:Zariski:dense} and Theorem~\ref{th:main:arithm} prompt us to ask the following

\begin{center}
(Q3) \hspace{2cm} For which $(\kappa, d)$ is $\rho_d(\PB_n)$ a thin group?
\end{center}

Some examples of $(\kappa,d)$ such that $\rho_d(\PB_n)$ is a thin group are known by the work of Deligne-Mostow~\cite{DM86}. In those examples, one has $ 1 < \frac{k_1}{d}+\dots+\frac{k_n}{d} \leq 2$ (recall that $1 \leq k_i \leq d-1$). To the authors' knowledge, little is known about this question in the general case.

\medskip

To close this section, we would like to mention some previous works which addressed  similar problems. In~\cite{GL:GAFA09}, Grunewald and Lubotzky consider a collection of linear representations of the automorphism group  of a free group, and show that the images of certain representations in this family are arithmetic. Note that $\PB_n$ is a subgroup of  the outer automorphism group of the free group generated by $n$ elements. Linear representations via covering construction of the mapping class groups  of closed surfaces (without punctures) in higher genus were investigated by Looijenga \cite{Looijenga97} and by  Grunewald, Larsen, Lubotzky, and Malestein \cite{GLLM:GAFA15}. The images of those representations are shown to be arithmetic in many cases.

\subsection{Outline}\label{subsec:outline}
We now give the  organization of the paper and a sketch of the proofs. Throughout this section $D$ will be the unit disc in $\C$, $\Bc:=\{b_1,\dots,b_n\}$ is a subset of $D$ where $b_i\neq b_j$ if $i\neq j$,  $\hX$ is the compact Riemann surface obtained from \eqref{eq:plane:curve}, and $q$ is a fixed primitive $d$-th root of unity.

In \textsection\ref{sec:results:Menet} we recall the main results of the first named author's thesis.
The first part of the thesis computes explicitly (a) the dimension of the space $\Vbb^{(q)}:=H^1(\hX)_q$,  (b) the signature $(r_q,s_q)$ of the restriction of the intersection form to $\Vbb^{(q)}$, and (c) the coefficients of the matrix of the intersection form in a specific basis of $\Vbb^{(q)}$. The second part of the thesis describes the action of $\rho_q(\tau)$ where $\tau\in \PB_n$ is a Dehn twist, on the basis constructed in the first part. It is shown in particular that  either $\rho_q(\tau)$  has  finite order, or is unipotent.

\medskip

In \textsection\ref{sec:top:prim}  we recall the topological properties of the (ramified) covering $\hX \to \CP^1$.
We then prove several technical lemmas on the cohomology (with compact support) of subsurfaces of $\hX$ that are preimages of subsets of $D$.
%One of the results that are frequently used in the sequel is Proposition~\ref{prop:coh:subsurf:embedding},  which gives the properties of the $q$-eigenspace of the action of $T^*$ on $H^1_c(Y)$, where $Y \subset \hX$ is  the preimage of a subsurface $E \subset D$ of the form $E=E_0\setminus \left(\cup_{j=1}^\ell E_j\right)$, where $E_0$ is a disc in $D$, and $E_1,\dots,E_\ell$ are $\ell$ pairwise disjoint discs contained in $E_0$.
We also recall a relation on Dehn twists on a four-holed sphere known as the ``lantern relation". This relation will be used to give the explicit matrix of the action of a particular Dehn twist on a 2 dimensional subspace of $H^1(\hX)_q$.

\medskip

In \textsection\ref{sec:preparation:unitary:grp}, we first prove some criteria for a subgroup generated by $3$ elements to be Zariski dense in $\SU(1,1)$ or in $\SU(2)$. These results serve to initiate the induction which  shows that the Zariski closure of $\rho_q(\PB_n)$  contains $\SU(\Vbb^{(q)})$ as a finite index subgroup. The key of the induction is the following statement (cf. Proposition~\ref{prop:density:induction}): let $V$ be a complex vector space endowed with a non-degenerate Hermitian form $H$. Assume that $V$ admits an orthogonal decomposition $V=V'\oplus V''$ such that the restrictions of $H$ to both $V'$ and $V''$ are non-degenerate. Let $\Gb$ be an algebraic subgroup   of $\SU(V)$. Assume that $\Gb$ contains $\SU(V')\times\SU(V'')$, and there are an element $\gamma \in \Gb$ and a vector  $v \in V\setminus(V'\cup V'')$
such that $\Im(\gamma-\Id_V)= \Span(v)$. Then we must have $\Gb=\SU(V)$.
We will apply this result in the situations where $V,V',V''$ are  $q$-eigenspaces of the action of $T^*$ on cohomology of subsurfaces of $\hX$ that are preimages of subsurfaces of $D$.

\medskip

In \textsection\ref{sec:Zar:dense:on:factor} we prove that for all $k\in \{1,\dots,d-1\}$, if the sequence $\mu^{(k)}:=(\left\{\frac{kk_1}{d}\right\},\dots,\left\{ \frac{kk_n}{d}\right\})$ is good in the sense of Definition~\ref{def:good:weights}, then the Zariski closure of $\rho_q(\PB_n)$ contains $\SU(\Vbb^{(q)})$ as a finite index subgroup (cf. Theorem~\ref{th:Zar:density:on:factor}).
In \textsection\ref{sec:Zar:density}, we give the proofs of Theorem~\ref{th:main:Zar:dense:bis} and of Theorem~\ref{th:main:Zariski:dense}. Theorem~\ref{th:main:Zar:dense:bis} follows from the fact that for any $q'\in \Ub_{d, \prim}, q'\neq q$, if $\dim \Vbb^{(q)}=\dim \Vbb^{{q'}} \geq 3$ or there is a pair of indices $\{i,j\}$ such that $\gcd(k_i+k_j,d)=1$, the representations $\rho_q$ and $\rho_{q'}$ are not conjugate. Theorem~\ref{th:main:Zariski:dense} then follows from Theorem~\ref{th:main:Zar:dense:bis} and Theorem~\ref{th:Zar:density:on:factor}.

\medskip

The sections \textsection\ref{sec:horospherical} and \textsection\ref{sec:prf:arithmetic} are devoted to the proof of Theorem~\ref{th:main:arithm}.
Let $G:=\prod_{q\in \Ub^+_{d, \prim}}\SU(\Vbb^{(q)})$ and $\Gamma:=\rho_d(\PB_n)\cap G$. To prove Theorem~\ref{th:main:arithm}, it is enough to show that $\Gamma$ is commensurable to $G(\Z)$. To this purpose, we will use the Margulis' arithmeticity criterion which is recalled in \textsection\ref{subsec:Margulis:crit}.
We first notice that it is enough to consider the case $d \, | \, (k_1+\dots+k_n)$ (cf. Lemma~\ref{lm:reduction:d:divides:sum:ki}).
Under the assumption that  $d \, | \, (k_1+\dots+k_m)$ for some $2\leq  m \leq n-1$, the condition in Definition~\ref{def:good:weights}(a) is satisfied for all $\mu^{(k)}=(\left\{\frac{kk_1}{d}\right\},\dots, \left\{\frac{kk_1}{d}\right\})$, where $k \in \Z$ such that $\gcd(k,d)=1$ (cf. Lemma~\ref{lm:0:mod:d:good}). Thus it follows from Theorem~\ref{th:main:Zariski:dense} that $\Gamma$ is Zariski dense in $G$.

To apply Margulis' criterion, we single out a horospherical subgroup $U=\prod_{q\in \Ub^+_{d,\prim}} U_q$ of $G$, where for each $q\in \Ub^+_{d,\prim}$, $U_q$ is a $2$-step nilpotent of $\SU(\Vbb^{(q)})$.
We have the following exact sequence
$$
0 \to N \to U \overset{\chi^+}{\to} U/N \to 0
$$
where $N\simeq \R^\ell$ is the center of $U$, and $U/N\simeq \C^{\ell(n-4)}$, with $\ell=|\Ub^+_{d,\prim}|$.
%
%
%This is the subgroup that acts trivially on the quotients of the following partial flag
%$$
%\{0\} \subset \Lbb^{(q)} \subset \Lbb^{(q)}\oplus\Wbb^{(q)} \subset \Vbb^{(q)},
%$$
%where $\Lbb^{(q)}$ is an isotropic line in $\Vbb^{(q)}$, and $\Wbb^{(q)}$ is a codimension 2 subspace orthogonal to $\Lbb^{(q)}$.
%The group $U_q$ is $2$-step nilpotent, and we have the following exact sequence
%$$
%\{0\} \to N_q\simeq \R \to U_q \overset{\chi_q}{\to} \C^{n-4} \to \{0\},
%$$
%where $N_q$ is the center of $U_q$.
%The subgroup of $\U(\Vbb^{(q)})$ that preserves the flag above is denoted by $\hat{P}_q$.  Then $\hat{P}_q$ normalizes $U_q$, and we have an induced action $\Rcal_q$ of $\hat{P}_q$ on $U_q/N_q$.
%Define
%$$
%U:=\prod_{q \in \Ub^+_{d, \prim}} U_q, \quad \text{ and } \quad N:=\prod_{q\in \Ub^+_{d, \prim}}N_q.
%$$
%Then $U$ is a horospherical subgroup of $G$, $N$ is the center of $U$, and we then have the following exact sequence
%$$
%0 \to N \to U \overset{\chi^+}{\to} U/N\simeq \prod_{q\in \Ub^+_{d, \prim}}(U_q/N_q) \to 0.
%$$
%Note that we have $N\simeq \R^{\ell}$ and $N/U\simeq \C^{\ell(n-4)}$, where $\ell=|\Ub^+_{d, \prim}|$.
%Let $\hat{P}^+:=\prod_{q\in \Ub^+_{d, \prim}}\hat{P}_q$. Then $\hat{P}^+$ normalizes $U$, and we also have an induced action  of $\hat{P}^+$ on $U/N$.

We will show that $\Gamma\cap U$ contains an irreducible lattices in $U$. Our strategy goes as follows: we first show that $\Gamma\cap U$ contains two elements $\tau$ and $\tau'$ such that $v:=\chi^+(\tau) \in U/N$ and $v':=\chi^+(\tau') \in U/N$ satisfy the following condition: for all $q \in \Ub^+_{d, \prim}$, the $q$-components of $v$ and $v'$ are non-trivial.
Let $\PB_{m-1}$ denote the subgroup of $\PB_n$ generated by $\{\alpha_{i,j}, \; 1 \leq i < j \leq m-1\}$, and $\PB_{m+1,n}$ the subgroup generated by $\{\alpha_{i,j}, \; m+1 \leq i < j \leq n\}$. It turns out that both $\rho_d(\PB_{m-1})$ and $\rho_d(\PB_{m+1,n})$ normalize $U$ and acts on $U/N$ by conjugacy. We then show that the union of the orbit of $v$ under the action of $\PB_{m-1}$ and the orbit of $v'$ under the action of $\PB_{m+1,n}$ contains an $\R$-basis of $U/N$. It follows that $\chi^+(\Gamma\cap U)$ contains a lattice in $U/N$.
It remains to show that $\Gamma\cap N$ contains a lattice in $N$. To achieve this, we make use of the law group in $U$ as well as the arithmetic properties of the intersection form.
It is not difficult to see that $\Gamma\cap U$ is irreducible in $G$. This allows us to conclude the proof of Theorem~\ref{th:main:arithm}.

\subsection*{Acknowledgement:} The authors warmly thank Jean-Fran\c{c}ois Quint for the very helpful discussions.
D.-M. N. is partly supported by the French ANR project ANR-19-CE40-0003.

\section{Construction and properties of the representations of pure braids via cyclic coverings}\label{sec:results:Menet}
We first recall the definition of the representations $\rho$ and $\rho_q$  introduced by McMullen in \cite{McM13}.
Without loss of generality, we can assume that the set $\Bc:=\{b_1,\dots,b_n\}$ is contained in the interior of the unit disc $D:=\{z \in \C, \; |z| \leq 1\}$.
It is well known that the braid group $\mathrm{B}_n$ is identified with the Mapping Class Group $\MCG(D,\Bc)$ that is the group of (homotopy classes of) orientation preserving homeomorphisms of $D$ which leave the set $\Bc$ invariant and restrict to the identity on $\partial D$. The pure braid group $\PB_n$ is the subgroup of $\mathrm{B}_n$ consisting of the homeomorphisms that are identity on $\Bc$.

Let   $h: D \to D$ be a homeomorphism that represents an element of $\PB_n$. Extending $h$ by identity outside of the unit disc, one can also view $h$ as a homeomorphism of the sphere $\CP^1$.
Let $\kappa$ and  $\hX$ be as in \textsection\ref{subsec:cyclic:covers:intro}.
Denote by $\tilde{\infty}$ the preimage of $\infty \in \CP^1$ in $\hX$. There exists a unique lift $ \tilde{h}: \hX \to \hX$ of $h$ which restricts to the identity  in a neighborhood of $\tilde{\infty}$.  By construction, $\tilde{h}$ commutes with $T$.
The correspondence $[h] \mapsto [\tilde{h}]$, where $[h]$ and $[\tilde{h}]$ are the homotopy classes of $h$ and $\tilde{h}$ respectively, provides us with a group morphism $\tilde{\rho}: \PB_n \to \MCG(\hX)^T$, where $\MCG(\hX)^T$ is the subgroup of the Mapping Class Group of $\hX$ consisting of (the homotopy classes of) homeomorphisms $\phi: \hX  \to  \hX$ such that $T^{-1}\circ\phi\circ T$ is homotopic to $\phi$.
%In what follows, given a homeomorphism $\phi \in \homeo^+(\hX)$, we denote by $[\phi]$ its image in $\MCG(\hX)$ ($[\phi]$.
Considering the action of $\tilde{h}$ on $H^1(\hX)$, we get a morphism $\rho: \PB_n \to \U(H^1(\hX))$, where $\U(H^1(\hX))$ is the group of automorphisms of $H^1(\hX)$ that preserve the intersection form $\langle.,.\rangle$.

%\medskip
%Consider now the action of $T$ on $H^1(X_{\bb,\kappa})$. Since we have $T^d=\id$, all the eigenvalues of $T^*$ are $d$-th roots of unity. For each $q \in \Ub_d:=\{\zeta^k,\; k=0,\dots,d-1\}$, denote by $H^1(X_{\bb,\kappa})_q$ the eigenspace of $T^*$ in $H^1(X_{\bb,\kappa})$ associated with eigenvalue $q$.  The group  of automorphisms of $H^1(X_{\bb,\kappa})_q$ that preserve the restriction of $\langle .,. \rangle$ will be denoted by $U(H^1(X_{\bb,\kappa})_q)$.

Since $\tilde{h}$ commutes with $T$,  $\tilde{h}^*$ preserves the eigenspaces $H^1(\hX)_q$ of $T^*$ for all $q\in \Ub_d$. As a consequence, for each $q\in \Ub_d$ we get a representation
$$
\begin{array}{cccc}
\rho_q: & \PB_n & \to & \U(H^1(\hX)_q)\\
        &  [h] & \mapsto & \tilde{h}^*_{|H^1(\hX)_q}.
\end{array}
$$
The morphisms $\rho$ and $\rho_q$  are precisely the monodromy actions of $\PB_n\simeq \pi_1((\C^n)^*)$ on the bundles $\Hb_\C(\kappa)$ and $\Hb_\C(\kappa,q)$ over  $(\C^n)^*$.

\medskip

In the case $\kappa=(1,\dots,1)$, the representations $\rho_q$'s have been studied by McMullen in \cite{McM13}. Actually, in this case $\rho$ and $\rho_q$ are the restrictions to $\PB_n$ of some  representations of the braid group $\mathrm{B}_n$ into $\U(H^1(\hX))$ and $\U(H^1(\hX)_q)$ respectively.
%In \cite{McM13}~McMullen studies the representations $\rho_q$ in the case $\kappa=(1,\dots,1)$.
In her thesis \cite{Menet:these}, the first named author generalises McMullen's results to the case where $k_i$ takes arbitrary values in $\Z_{\geq 1}$. The results of \cite{Menet:these} are summarized here below, they are the starting point of our investigation.
%It is also worth noticing that the proofs of the results of \cite{Menet:these} use in an essential way the results of \cite{McM13}.

\begin{Theorem}[Menet \cite{Menet:these}]\label{th:Menet:dim:signature}
Let $k\in \{1,\dots,d-1\}$ be such that $q=e^{-\frac{2\pi\imath k}{d}}$.
Define $F'_q:=\{i \in \{1,\dots,n\}, \; q^{k_i}=1\}$.
Then
\begin{equation}\label{eq:dim:H:1:X:q}
\dim H^1(\hX)_q=n-1-\#F'_q-\eps_0,
\end{equation}
where
\begin{equation*}
\eps_0=\left\{
\begin{array}{cl}
1 & \text{ if } q^{k_1+\dots+k_n}=1,\\
0 & \text{ otherwise}.
\end{array}
\right.
\end{equation*}
The restriction of the intersection form $\langle .,. \rangle$ to $H^1(\hX)_q$ is a non-degenerate Hermitian form with signature $(r_q,s_q)$ where
\begin{align}
\label{eq:signature:rq} r_q & =\left\lfloor \left\{\frac{k}{d}\cdot k_1\right\}+\dots+\left\{\frac{k}{d}\cdot k_n\right\}\right\rfloor -\eps_0,\\
\label{eq:signature:sq} s_q & = \left\lfloor \left\{(1-\frac{k}{d})\cdot k_1\right\}+\dots+\left\{(1-\frac{k}{d})\cdot k_n\right\}\right\rfloor -\eps_0
\end{align}
(here $\{x\}=x-\lfloor x\rfloor$ for all $x\in \R$).
\end{Theorem}

\begin{Theorem}[Menet \cite{Menet:these}]\label{th:Menet:generator:set}
For all $i=1,\dots,n$, let $D_i$ be a disc in $D$  such that $D_i\cap\Bc=\{b_i,b_{i+1}\}$ (by convention $b_{n+1}=b_1$). Let $X_i$ be the preimage of $D_i$ in $\hX$.
Assume that $q^{k_i}\neq 1$ for all $i\in \{1,\dots,n\}$.
Then there exists a generating set $\{g_1,\dots,g_n\}$ of $H^1(\hX)_q$ which satisfies the followings
\begin{itemize}
\item[$\bullet$] $g_i$ has compact support in $X_i$, \\

\item[$\bullet$] $\langle g_i,g_i \rangle=-\imath(1-q)(1-\bar{q})\frac{1-q^{k_i+k_{i+1}}}{(1-q^{k_i})(1-q^{k_{i+1}})}$,\\

\item[$\bullet$] $\langle g_i,g_{i+1}\rangle= \imath(1-q)(1-\bar{q})\frac{1}{1-q^{k_{i+1}}}$,\\

\item[$\bullet$] $\langle g_i,g_j\rangle =0$ if $|j-i|>1$.
\end{itemize}
Moreover, we have
\begin{equation}\label{eq:rel:generators}
\sum_{i=1}^ng_i=0 \in H^1(\hX).
\end{equation}
In the case $q^{k_1+\dots+k_n}=1$ (that is $\epsilon_0=1$), the generators $\{g_1,\dots,g_n\}$ satisfy an additional relation
\begin{equation}\label{eq:rel:gen:degen:case}
\sum_{i=1}^n\bar{q}^{k_1+\dots+k_i}g_i=0 \in H^1(\hX).
\end{equation}
\end{Theorem}
\begin{Remark}\label{rk:generators:up:to:const}
The generating family $\{g_1,\dots,g_n\}$ is uniquely determined up to a constant in $\S^1$.
\end{Remark}

\begin{Remark}\label{rk:compare:Menet:McM:gen:set}
In the case $k_1=\dots=k_n=1$, McMullen~\cite[Cor. 3.2, Cor 3.3]{McM13} showed that
$$
\dim H^1(\hX)_q=\left\{
\begin{array}{cl}
n-1, & \text{ if } q^n\neq 1, \\
n-2, & \text{ if } q^n=1,
\end{array}
\right.
$$
and the signature of the restriction of $\langle.,.\rangle$ to $H^1(\hX)_q$ is given by $(r_q,s_q)=(\lceil n(k/d)-1\rceil, \lceil n(1-k/d)-1\rceil)$. He also showed that (see~\cite[Th. 4.1]{McM13})
\begin{itemize}
\item[$\bullet$] $\langle e_i, e_i \rangle =2 \Im(q)$, and

\item[$\bullet$] $\langle e_i, e_{i+1} \rangle = \imath(1-\bar{q})$.
\end{itemize}
One can readily check that those formulas coincide with the ones in Theorem~\ref{th:Menet:dim:signature} and Theorem~\ref{th:Menet:generator:set} in this case.
\end{Remark}

Recall that $\PB_n$ is generated by  $\{\alpha_{i,j}, \; 1 \leq i < j \leq n\}$, where $\alpha_{i,j}$ is the Dehn twist about a circle bordering  a disc containing only $b_i,b_j$.

\begin{Theorem}[Menet~\cite{Menet:these}]\label{th:Menet:alphaij}
Assume that $q^{k_i}\neq 1$ for all $i=1,\dots,n$. Let $\{g_1,\dots,g_n\}$ be the generating set of $H^1(\hX)_q$ in Theorem~\ref{th:Menet:generator:set}. Then for all $i,j \in \{1,\dots,n\}, i < j$, the action of $\rho_q(\alpha_{i,j})$ on $H^1(\hX)_q$ is given by
\begin{equation}\label{eq:action:alphaij}
\rho_q(\alpha_{i,j})(x)=x-\imath\frac{(1-q^{k_i})(1-q^{k_j})}{(1-q)(1-\bar{q})}\langle x ,g^*_{i,j}\rangle g^*_{i,j},
\end{equation}
where
$$
g^*_{i,j}=g_i+\sum_{l=i+1}^{j-1}\bar{q}^{k_{i+1}+\dots+k_l}g_l.
$$
\end{Theorem}
\begin{Remark}\label{rk:compare:Menet:McM:Dehn:twists}
Note that we have $g^*_{i,i+1}=g_i$, and \eqref{eq:action:alphaij} generalizes the formulas (4.1) and (4.2) of \cite{McM13}.
\end{Remark}
Using the formulas in Theorem~\ref{th:Menet:generator:set}, one can check that
$$
\langle g^*_{i,j}, g^*_{i,j}\rangle = -\imath(1-q)(1-\bar{q})\frac{1-q^{k_i+k_j}}{(1-q^{k_i})(1-q^{k_j})}.
$$
In particular, $\langle g^*_{i,j},g^*_{i,j}\rangle =0$ if and only  if $q^{k_i+k_j}=1$.
In the case $q^{k_i+k_j}\neq 1$, we can write
$$
\rho_q(\alpha_{i,j})(x)=x-(1-q^{k_i+k_j})\cdot\frac{\langle x, g^*_{i,j}\rangle}{\langle g^*_{i,j}, g^*_{i,j}\rangle}\cdot g^*_{i,j}
$$
which means that $\rho_q(\alpha_{i,j})$ is a complex reflection which fixes the hyperplan $\langle g^*_{i,j} \rangle^\perp$ and maps $g^*_{i,j}$ to $q^{k_i+k_j}\cdot g^*_{i,j}$.
As a consequence, we get
\begin{Corollary}\label{cor:det:alpha:ij}\hfill
\begin{itemize}
\item[$\bullet$] If $q^{k_i+k_j} \neq 1$ then $\rho_q(\alpha_{i,j})$ has  the same  order as  $q^{k_i+k_j} \in \Ub_d$. If $q^{k_i+k_j}=1$, then $\rho_q(\alpha_{i,j})$ is unipotent.

\item[$\bullet$] For all $1\leq i < j \leq n$, we have
$$
\det(\rho_q(\alpha_{i,j}))=q^{k_i+k_j} \in \Ub_d.
$$
\end{itemize}
\end{Corollary}

For more general Dehn twists, we have

\begin{Theorem}[Menet~\cite{Menet:these}]\label{th:Menet:Dehn:twist}
Let $\gamma$ be a simple closed curve in $D$ bordering a disc $E$  such that $E\cap \Bc=\{b_1,\dots,b_r\}$, with $r\geq 2$. Let  $\tau_\gamma$ be the Dehn twist about  $\gamma$. Let $\{g_1,\dots,g_n\}$ be the generating set of $H^1(\hX)_q$ given in Theorem~\ref{th:Menet:generator:set}.  Then the action of $\rho_q(\tau_\gamma)$ on $H^1(\hX)_q$ satisfies
$$
\rho_q(\tau_\gamma)(g_i)= \left\{
\begin{array}{ll}
q^{k_1+\dots+k_r}g_i, & \text{ if } i=1,\dots,r-1,\\
g_r+q^{k_1+\dots+k_r}\sum_{l=1}^{r-1}(\bar{q}^{k_1+\dots+k_l}-1)g_l, & \text{ if } i=r, \\
g_i, & \text{ if } i=r+1,\dots,n-1,\\
g_n+\sum_{l=1}^{r-1}(1-q^{k_{l+1}+\dots+k_r})q_l, & \text{ if } i=n.
\end{array}
\right.
$$
If $q^{k_1+\dots+k_r}=1$, then $\rho_q(\tau_\gamma)$ is unipotent. Otherwise $\rho_q(\tau_\gamma)$ has  the same order as $q^{k_1+\dots+k_r} \in \Ub_d$.
\end{Theorem}

\section{Topological preliminaries}\label{sec:top:prim}
Let $\kappa=(k_1,\dots,k_n), \bb=(b_1,\dots,b_n)$ and $\hX$ be as in \textsection~\ref{subsec:cyclic:covers:intro}.
In this section, we prove some  technical results on the cohomology with compact support on subsurfaces of $\hX$ invariant by the action of $T$.
Those results are frequently used in the proof of the main theorems. Throughout this section, we do not assume that $\gcd(k_1,\dots,k_n,d)=1$.
Let $e:=\gcd(k_1,\dots,k_n,d)$. This means that $\hX$ has $e$ connected components. The automorphism $T$ permutes cyclicly the components of $\hX$, and $T^{e}$ preserves each of these components.

\subsection{Topological cyclic coverings of punctured spheres}\label{subsec:topo:cover}
Recall that $\tilde{\Bc}$ and $\tilde{\infty}$ are the preimages of $\Bc=\{b_1,\dots,b_n\}$ and of $\infty$ in $\hX$ respectively. By a slight abuse of notation, we denote again by $\pi$  the restriction of the projection $\pi: \hX \to \CP^1$ to $\hX^*:=\hX\setminus(\tilde{\Bc}\cup\tilde{\infty})$. The map $\pi: \hX^* \to \C\setminus\Bc$ is a topological covering of degree $d$.

Topologically, the punctured surface $\hX^*$ can be constructed as follows: consider the group morphism $\chi: \pi_1(\C\setminus\Bc) \to \Z/d\Z$ which maps  a loop freely homotopic to the  border of a small disc about $x_i$ to $k_i \mod d$. Then $\chi(\pi_1(\C\setminus\Bc))=\{0,e,\dots,d-e\}$. Let $\Pi_0:=\ker(\chi) \subset \pi_1(\C\setminus\Bc)$. Let $\hX^*_0$ be a component of $\hX^*$. Then  $\hX^*_0$ is homeomorphic to the quotient $\Hbb/\Pi_0$, where $\Hbb$ is the universal cover of $\C\setminus\Bc$.
In this setting, the action of $T^{e}$ on $\hX^*_0$ is induced by  the action of the lift of a loop $\gamma$ in   $\C\setminus\Bc$ such that $\chi(\gamma)=e$.
Let $\hX_0$ be the closure of $\hX^*_0$ in $\hX$. Then $\hX_0$ is a connected component of  $\hX$, and we have $\hX= \sqcup_{k=0}^{d/e-1}T^k(\hX_0)$.

\begin{Lemma}\label{lm:equiv:exponents}
Let $\kappa'=(k'_1,\dots,k'_n) \in (\Z_{\geq 0})^n$ be another tuple of $n$ natural numbers. Let $\pi': \hX' \to \CP^1$ be cyclic covering constructed from the curve $\CC_{\bb,\kappa'}: y^d=\prod_{i=1}^n(x-b_i)^{k'_i}$. Denote by $T'$ the automorphism of $\hX'$ which is induced by the map $(x,y) \mapsto (x,\zeta_d\cdot y)$.  Assume that
$$
k'_i -k_i \equiv 0 \mod d, \; \text{ for all } i \in \{1, \dots,n\}
$$
Then there is an isomorphism $f: \hX \to  \hX'$ such that $\pi=\pi'\circ f$, and $T'\circ f=f\circ T$.
\end{Lemma}
\begin{proof}
Denote by $\CC^*_{\bb,\kappa}$ and $\CC^*_{\bb,\kappa'}$ the preimages of $\C\setminus\{b_1,\dots,b_n\}$ in $\CC_{\bb,\kappa}$ and $\CC_{\bb,\kappa'}$ respectively. Let $r_i=\frac{k'_i-k_i}{d}, \; i=1,\dots,n$. Then the map
$$
\begin{array}{cccc}
f: & \CC^*_{\bb,\kappa} & \to &  \CC^*_{\bb,\kappa'}\\
   & (x,y) & \mapsto & (x,(x-b_1)^{r_1}\cdots(x-b_n)^{r_n}y)
\end{array}
$$
is an isomorphism. It is straightforward to check that $f$ extends to an isomorphism between $\hX$ and $\hX'$ satisfying the conditions stated in the lemma.
\end{proof}
\begin{Remark}\label{rk:lm:equiv:exponents}
Lemma~\ref{lm:equiv:exponents} can also be shown by using the topological construction of $\hX$ and $\hX'$.
\end{Remark}

The following lemma also follows immediately from the topological construction of $\hX$. Its proof is left to the reader.
\begin{Lemma}\label{lm:equiv:subsurfaces}
Let $\{I_1,\dots,I_m\}$ be a partition of the set $\{1,\dots,n\}$ such that $m\geq 3$. Let $\kappa'=(k'_1,\dots,k'_m)\in (\Z_{>0})^m$ be a sequence of positive integers such that  $k'_j \equiv \sum_{i\in I_j}k_i \mod d, \; j=1,\dots,m$.
Let $E_1,\dots,E_m$ be $m$ discs in $\C$ which are pairwise disjoint and satisfy $E_j\cap \Bc=\{x_i, \; i\in I_j\}$.
For each  $j=1,\dots,m$, choose a point $b'_j$ in $E_j$, and let $\hX'$ be the Riemann surface constructed from the curve  $y^d=\prod_{j=1}^m(x-b'_j)^{k'_j}$.
Denote by $T'$ the isomorphism of $\hX'$ induced by the map $(x,y)\mapsto (x,\zeta_d\cdot y)$.

Let $Y$ and $Y'$ be the preimages of $\CP^1 \setminus \left(\bigcup_{j=1}^m E_j\right)$ in $\hX$ and $\hX'$ respectively. Then there is a homeomorphism $\varphi: Y \to Y'$ such that $T '\circ \varphi = \varphi\circ T$.
\end{Lemma}

\subsection{Cohomology of subsurfaces invariant by $T$}\label{subsec:coh:inv:subsurf}
In view of Lemma~\ref{lm:equiv:exponents}, from now on we will make the following assumption
\begin{equation}\label{eq:cond:ki<d}
 0 < k_i < d, \quad  \text{ for all } i=1,\dots,n.
\end{equation}
By convention, all the subsurfaces in $\hX$ are considered without their boundary (that is subsurfaces are open subsets of $\hX$). If $Y$ is a subsurface of $\hX$, we denote by $Y^c$ the interior of its complement, and by $\partial Y$ a tubular neighborhood of its boundary. By a slight abuse of notation, one can write $\partial Y =  Y \cap Y^c$.
All cohomologies are considered with coefficients in $\C$, unless indicated otherwise.
Assuming that $Y$ is invariant by $T$, for all $q\in \Ub_d$, we will denote by $H^\bullet_c(Y)_q$ the $q$-eigenspace of the action of $T^*$ on $H^\bullet_c(Y)$.

In what follows $q$ will be a fixed primitive $d$-th root of unity.
\begin{Lemma}\label{lm:coh:disconnect:surf}
Let $\hX_0$ be a connected component of $\hX$.  Denote by $H^1(\hX_0)_{q^e}$ the $q^e$-eigenspace of the action of $T^{e*}$ on $H^1(\hX_0)$.  Then  there is an isomorphism between $H^1(\hX)_q$ and $H^1(\hX_0)_{q^e}$.
\end{Lemma}
\begin{proof}
Denote the components of $\hX$ by $\hX_k, \; k=0,\dots,e-1$, where $\hX_k=T^k(\hX_0)$.
Every  element $\eta\in H^1(\hX)$ is given by a tuple $(\eta_0,\dots,\eta_{e-1})$, with $\eta_k\in H^1(\hX)_k$. Since $T(\hX_k)=\hX_{k+1}$, it follows that $\eta \in H^1(\hX)_q$ if and only if $T^*(\eta_{k})=q\eta_{k-1}$. It follows in particular that $T^{e*}(\eta_0)=q^e\eta_0$, that is $\eta_0\in H^1(\hX_0)_{q^e}$. Conversely, given $\eta_0 \in H^1(\hX_0)_{q^e}$, by setting $\eta_k:=q^kT^{-k*}\eta_0$, and $\eta:=(\eta_0,\dots,\eta_{e-1})$, we have that $\eta \in H^1(\hX)_q$.
\end{proof}
As a consequence of Lemma~\ref{lm:coh:disconnect:surf} we get
\begin{Corollary}\label{cor:hX:not:connected}
Theorem~\ref{th:Menet:dim:signature} and Theorem~\ref{th:Menet:generator:set} hold whenever $e=\gcd(k_1,\dots,k_n,d) < d$.
\end{Corollary}
%The proof of the following lemma is straightforward and left for the reader.
\begin{Lemma}\label{lm:coh:annulus} \hfill
\begin{itemize}
\item[(i)] Let $U$ be a subsurfaces in $\hX$  whose components are cyclicly permuted by  $T$. Then we have
$$
\dim H^2_c(U)_q=\left\{
\begin{array}{ll}
0 & \text{ if } \quad  |\pi_0(U)| < d, \\
1 & \text{ if } \quad |\pi_0(U)| = d.
\end{array}
\right.
$$

\item[(ii)] Let $E$ be a disc in $D$ whose boundary is disjoint from the set $\Bcal$. Denote by $A$ the preimage of an annulus that contains the boundary of $E$ and disjoint from  $\Bcal$. Let $I:=\{i\in \{1,\dots,n\}, \; b_i \in E\}$. Then we have
$$
\dim_\C H^1_c(A)_q =\dim_\C H^2_c(A)_q=\left\{
\begin{array}{ll}
0 & \text{ if } q^{\sum_{i\in I} k_i} \neq 1, \\
1 & \text{ if } q^{\sum_{i\in I} k_i} = 1.
\end{array}
\right.
$$
\end{itemize}
\end{Lemma}
\begin{proof}
The assertions of the lemma follow from the dualities $H^2_c(U)_q \simeq H^0(U)^*_{\bar{q}}, H^2_c(A)_q\simeq H^0(A)^*_{\bar{q}}$, and $H^1_c(A)_{q} \simeq H^1(A)^*_{\bar{q}}$. The details are left to the reader.
\end{proof}
Recall that the set of marked points $\Bc$ is contained in the disc $D$.
Let $X$ denote the preimage of the interior of $D$ in $\hX$.
As usual we denote by $\partial X$ a tubular neighborhood of the boundary of $X$.
Note that the complement of $X$ in $\hX$ is the union of $e_0=\gcd(k_0,d)$ discs, and $\partial X$ is the union of $e_0$ annuli cyclicly permuted by $T$.
We consider $H^1_c(\partial X)_q$ as a subspace of $H^1(X)_q$ via the morphism  induced by  the inclusion $\partial X \hookrightarrow X$. Let $\{g_1,\dots,g_n\}$ be the generating set of $H^1(\hX)_q$ given in Theorem~\ref{th:Menet:generator:set}. Since the $g_i$ is represented by a closed $1$-form with support in $X$, we can consider $g_1,\dots,g_n$ as elements of $H^1_c(X)_q$.
The following lemma is implicitly proved in \cite{Menet:these}.
%Its proof is actually included in the proof of relation \eqref{eq:rel:gen:degen:case}.
\begin{Lemma}\label{lm:dual:border:neigh:infty}
Assume that $q^{k_1+\dots+k_n}=1$ then $H^1_c(\partial X)_q \subset H^1_c(X)_q$ is generated by
$$
g^*=\sum_{i=1}^{n-1}(\bar{q}^{k_1+\dots+k_i}-1)g_i.
$$
\end{Lemma}
\begin{proof}[Sketch of proof]
We consider an auxiliary surface $\hY$ defined as follows: for all $i=1,\dots,n$, let $D'_i$ be a small disc about $b_i$ contained in $D$ such that $D'_i\cap D'_j=\vide$ if $i\neq j$.
%We suppose that all the discs $D'_i$'s are contained in $D$.
For each $i\in\{1,\dots,n\}$, let $b_{i,1},\dots,b_{i,k_i}$ be $k_i$ distinct points in $D'_i$. Then $\hY$ is the Riemann surface obtained from the equation
$$
y^d=\prod_{i=1}^n\prod_{j=1}^{k_i}(x-b_{i,j}).
$$
By a slight abuse of notation, denote also by $T$ the automorphism of $\hY$ induced by  the map $(x,y) \mapsto (x,\zeta_dy)$.
Let $Y$ be the preimage of $D$ in $\hY$.
Let $X'$ and $Y'$ be the preimages of $D\setminus\left(\cup_{i=1}^nD'_i \right)$ in $\hX$ and in $\hY$ respectively. It follows from Lemma~\ref{lm:equiv:subsurfaces} that there is a homeomorphism $\varphi: X' \to Y'$ equivariant with respect to  the actions of $T$. Therefore we have $H^1_c(X')_q\simeq H^1_c(Y')_q$.

Let $Y'_i$ be the preimage of the interior of the disc $D'_i$ in $\hY$. Since $q^{k_i}\neq 1$ for all $i\in \{1,\dots,n\}$, we have $H^1_c(X')_q \simeq H^1_c(X)_q$, and $H^1_c(Y)_q$ admits the following  decomposition
\begin{equation}\label{eq:H:1:hY:q:decomp}
H^1_c(Y)_q=H^1_c(Y')_q\oplus H^1_c(Y'_1)_q\oplus\dots\oplus H^1_c(Y'_n)_q.
\end{equation}
Let $\phi: H^1_c(Y)_q \to H^1_c(X)_q$ be the composition of the projection $H^1(\hY) \to H^1_c(Y')_q$ given by \eqref{eq:H:1:hY:q:decomp} followed by the isomorphism $\varphi^* : H^1_c(Y')_q \to H^1_c(X')_q\simeq H^1_c(X)_q$.

\medskip

For all $i\in\{1,\dots,n\}$ and $j\in \{1,\dots,k_i-1\}$, let $c_{k_1+\dots+k_{i-1}+j}$ be a simple arc in $D'_i$ joining $b'_{i,j}$ to $b'_{i,j+1}$. Let $c_{k_1+\dots+k_i}$ be a simple arc in $D$ joining $b'_{i,k_i}$ to $b'_{i+1,1}$ (here we use the convention $n+1\equiv 1$). The arcs $\{c_l, \; l=1,\dots,k_0\}$ (recall that $k_0=k_1+\dots+k_n$) are chosen so that if $l\neq l'$ then $c_l$ and $c_{l'}$ are either disjoint or only meet at their unique common endpoint. In particular, the union of those arcs forms the boundary of a disc $B$ in $\C$.
%Let $B^c$ denote the complement of $B$ in $\CP^1$.
%Note that by construction, the boundary of $B$ and the boundary of $D$ are homotopic via a homotopy which does not meet the points in $\Bc'=\{b'_{i,j}, \; i=1,\dots,n, \; j=1,\dots,k_i\}$.

The preimage of each arc $c_l$ in $Y$ consists of  $d$ simple arcs with the same endpoints. We label those arcs by $\{c_{l,m}, \; m=0,\dots,d-1\}$, where $c_{l,m}=T^m(c_{l,0})$. Define
$$
\cc_l:=\sum_{m=0}^{d-1}q^m\cdot c_{l,m}.
$$
One readily checks that $\cc_l$ is a $1$-cycle, thus it represents an element of $H_1(Y,\C)$. Let $\hg_l \in H^1_c(Y)$ denote the Poincar\'e dual of $\cc_l$. By definition, $\hg_l$ is a closed $1$-form with compact support such that
$$
\int_{\cc_l}\beta=\int_{\hY}\hg_l\wedge \beta,
$$
for all closed $1$-form $\beta$ on $Y$. Remark that $T_*(\cc_l)=\bar{q}\cdot\cc_l$. Therefore, $T^*(\hg_l)=q\cdot\hg_l$ that is $\hg_l\in H^1_c(Y)_q$.

If $l=k_1+\dots+k_{i-1}+j$ where  $j\in\{1,\dots,k_i-1\}$, then since the support of $\cc_l$ is contained in $Y'_i$. Thus $\hg_l$ is represented by a closed $1$-form with compact support in $Y'_i$,  which means that $\hg_l\in H^1_c(Y'_i)_q$, and hence $\phi(\hg_l)=0$.
It is shown in \cite{Menet:these} that  the elements of the family $\{g_1,\dots,g_n\}$ can be defined as
$$
g_i:=\frac{1-q}{\sqrt{d}}\phi(\hg_{k_1+\dots+k_i}) \in H^1_c(X)_q.
$$
Recall that the arcs $\{c_l, \; l=1,\dots,k_0\}$ form the boundary of a disc $B$ in $D$. The preimage of $B$ in $Y$ consists of $d$ discs denoted by $B_0,\dots,B_{d-1}$, where $B_m=T^m(B_0)$. It is not difficult to check that
\begin{equation}\label{eq:Y:rel:homology:a}
\sum_{l=1}^{k_0}\cc_l=\sum_{m=1}^{d-1} q^m\partial B_m.
\end{equation}
Therefore, we have $\hg_1+\dots+\hg_{k_0}=0 \in H^1_c(Y)_q$, which implies that $g_1+\dots+g_n=0 \in H^1_c(X)_q$.

Let $S$ denote the boundary of the disc $D$. Since $q^k_0=1$, the preimage $\hS$ of $S$ in $Y$ consists of $d$ circles denoted by $S_0,\dots, S_{d-1}$, where $S_m=T^m(S_0)$.  Note that each component of $\hS$ bounds a disc in $\hY$  containing  a point in the preimage of $\infty$.
In \cite{Menet:these}, it is shown that
\begin{equation}\label{eq:Y:rel:homology:b}
\sum_{l=1}^{k_0}\bar{q}^l\cdot\cc_l \sim \sum_{m=0}^{d-1}q^{m-1}\cdot S_m
\end{equation}
where $\sim$ is the homologuous equivalence relation of $1$-cycles in $Y$.
Let $\ss$ be the homology class of $\sum_{m=0}^{d-1}q^{m-1}\cdot S_m$ in $H_1(Y)$.
Note that we have $T_*(\ss)=\bar{q}\cdot\ss$.
The dual $\eta \in H^1_c(Y)$ of $\ss$ a generator of $H^1_c(\partial Y)_q$.
Since $\partial X$ can be identified with $\partial Y$ via the map $\varphi$, we can consider $\eta$ as a generator of $H^1_c(\partial X)_q$.
It follows from \eqref{eq:Y:rel:homology:b} that we have
$$
\sum_{l=1}^{k_0}\bar{q}^l\cdot \hg_l=\eta \in H^1_c(Y)_q.
$$
Combining with \eqref{eq:Y:rel:homology:a}, we get
$$
\eta=\sum_{l=1}^{k_0}(\bar{q}^l-1)\cdot\hg_l.
$$
Applying $\phi$ and notice that $\bar{q}^{k_0}-1=0$, we get the desired conclusion.
\end{proof}

\begin{Lemma}\label{lm:coh:preimage:D}
Let $\iota_X: H^1_c(X)_q \to H^1(\hX)_q$ be the morphism induced by the inclusion $X \hookrightarrow \hX$. Then $\iota_X$ is an isomorphism if $q^{k_1+\dots+k_n}\neq 1$. In the case $q^{k_1+\dots+k_n}=1$, $\iota_X$ is surjective and $\ker(\iota_X)=H^1_c(\partial X)_q$. In all cases we always have
$$
\dim H^1_c(X)_q=n-1.
$$
and $\{g_1,\dots,g_{n-1}\}$ is a basis of $H^1_c(X)_q$.
%$g^*=\sum_{i=1}^{n-1}(\bar{q}^{k_1+\dots+k_i}-1)g_i$, where the forms $(g_1,\dots,g_{n-1})$ are given in Theorem~\ref{th:Menet:generator:set} and considered as elements of $H^1_c(X)_q$ .
\end{Lemma}
\begin{proof}
The complement $X^c$ of $X$ in $\hX$ is a disjoint union of discs. Therefore, $H^1_c(X^c)_q=H^1_c(X^c)=\{0\}$. The Mayer-Vietoris exact sequence associated to the decomposition $\hX=X\cup X^c$ gives
 \begin{equation}\label{eq:MV:seq:X:in:hX}
 0 \to H^1_c(\partial X)_q \to H^1_c(X)_q \to H^1(\hX)_q \to H^2_c(\partial X)_q \to H^2_c(X^c)_q \to 0.
 \end{equation}
If $q^{k_1+\dots+k_n}\neq 1$, then $|\pi_0(\partial X)|=|\pi_0(X^c)| < d$. Therefore  $H^1_c(\partial X)_q=H^2_c(\partial X)_q=H^2_c(X^c)=\{0\}$ by Lemma~\ref{lm:coh:annulus}. The exact sequence \eqref{eq:MV:seq:X:in:hX} then yields
 $$
 0\to H^1_c(X)_q \to H^1(\hX_q) \to 0
 $$
 which means that $\iota_X$ is an isomorphism. In particular, we have $\dim H^1_c(X)_q=\dim H^1(\hX)_q=n-1$.
 By Theorem~\ref{th:Menet:generator:set}, $\{g_1,\dots,g_n\}$ is a generating family of $H^1(\hX)_q$ which satisfies $g_1+\dots+g_n=0$. Thus $\{g_1,\dots,g_{n-1}\}$ is a basis of  both $H^1(\hX)_q$ and $H^1_c(X)_q$.

\medskip

Assume now that $q^{k_1+\dots+k_n}=1$. In this case, $e_0=d$. Hence $\dim H^1_c(\partial X)_q=\dim H^2_c(\partial X)_q=\dim H^2_c(X^c)= 1$ by Lemma~\ref{lm:coh:annulus}. The morphism $H^2_c(\partial X)_q \to H^2_c(X^c)_q$ is an isomorphism because it is a surjective linear map between two vector spaces of the same dimension. The exact sequence \eqref{eq:MV:seq:X:in:hX} then yields
 $$
 0 \to H^1_c(\partial X)_q \to H^1_c(X)_q \to H^1(\hX)_q \to 0
 $$
 which means that $\iota_X$ is surjective and $\ker(\iota_X)=H^1_c(\partial X)_q$. In particular,
 $$
 \dim H^1_c(X)_q=\dim H^1_c(\partial X)_q+\dim H^1(\hX)_q=1+(n-2)=n-1.
 $$
By Lemma~\ref{lm:dual:border:neigh:infty}, $H^1_c(\partial X)_q \subset \Span(g_1,\dots,g_{n-1})$. Since $\iota_X(\Span(g_1,\dots,g_{n-1}))=H^1(\hX)_q$, we conclude that $\Span(g_1,\dots,g_{n-1})=H^1_c(X)_q$, which means that $\{g_1,\dots,g_{n-1}\}$ is a basis of $H^1_c(X)_q$.
\end{proof}

The following  lemma will be frequently used in the sequel.
\begin{Lemma}\label{lm:coh:preimage:disc}
Let $E \subset D$ be a disc such that $E\cap \Bc=\{b_1,\dots,b_r\}$ with $r\geq 1$, and $Y$ the preimage of $E$ in $\hX$. Let $\iota_Y: H^1_c(Y)_q \to H^1(\hX)_q$ be the natural morphism induced by the inclusion $Y\hookrightarrow \hX$.
Let $\{g_1,\dots,g_n\}$ be the generating family of $H^1(\hX)_q$ in Theorem~\ref{th:Menet:generator:set}. Define
\begin{equation}\label{eq:def:boundary:coh:class}
g^*_r:=\sum_{l=1}^{r-1}(\bar{q}^{k_1+\dots+ k_l}-1)g_l \in H^1_c(Y)_q.
\end{equation}
Then  we have
\begin{itemize}
\item[(a)] $\dim H^1_c(Y)_q=r-1$ and $\{g_1,\dots,g_{r-1}\}$ is a basis of $H^1_c(Y)_q$.

\item[(b)] $\iota_Y$ is injective if $q^{k_1+\dots+k_r}\neq 1$ or $q^{k_1+\dots+k_r}=1$ and $r < n$.  In the case $r=n$ and $q^{k_1+\dots+k_n}=1$, $\iota_Y$ is surjective and $\ker(\iota_Y)$ is generated by $g^*_n$.

\item[(c)] The restriction of the intersection form $\langle.,.\rangle$ to $\Vbb:=\mathrm{Im}(\iota_Y)$ is non-degenerate if and only if $q^{k_1+\dots+k_r}\neq 1$, or  $q^{k_1+\dots+k_r}=1$ and $r=n$.
In the case $q^{k_1+\dots+k_r}=1$ and $r<n$, the element $\iota_Y(g^*_r)$ is non-zero in $H^1(\hX)$, and satisfies $\langle \iota_Y(g^*_r),\iota_Y(g')\rangle=0$ for all $g'\in H^1_c(Y)_q$.
\end{itemize}
\end{Lemma}
\begin{proof}\hfill
\begin{itemize}
\item[(a)] Consider the Riemann surface $\hY$ constructed from the curve $y^d=\prod_{i=1}^r(x-b_i)^{k_i}$. Let $T_{\hY}$ be the automorphism of $\hY$ induced by the map $(x,y) \to (x,\zeta_d y)$ and by $H^\bullet(\hY)_q$ the $q$-eigenspace of the actions of $T^*_{\hY}$ on $H^\bullet(\hY)_q$.
Note that the surface $\hY$ is not necessarily connected, but since none of the $k_i$ is divisible by $d$, it has less than $d$ components. It follows that $H^0(\hY)_q=H^2(\hY)_q=\{0\}$ (recall that $q$ is a primitive $d$-th root of unity) and Theorems~\ref{th:Menet:dim:signature} and Theorem~\ref{th:Menet:generator:set} holds for $\hY$.

The the preimage of $E$ in $\hY$ is homeomorphic to $Y$ by a homeomorphism that is equivariant with respect to the actions of $T$ and $T_{\hY}$ (see Lemma~\ref{lm:equiv:subsurfaces}).
Therefore, we can consider $Y$ as a subsurface of $\hY$, and identify $H^1_c(Y)_q$ with the $q$-eigenspace of the action of $T^*_{\hY}$ on $H^1_c(Y)$. Thus we obtain (a) from Lemma~\ref{lm:coh:preimage:D}. \\

\item[(b)] Consider the splitting $\hX=Y\cup Y^c$. Since $H^0(\hX)_q=\{0\}$ (because $|\pi_0(\hX)|=e <d$), from the Mayer-Vietoris sequence, we get the following exact sequence
\begin{equation}\label{eq:MV:seq:Y:n:Yc:in:hX}
0 \to H^1_c(\partial Y)_q \overset{\psi}{\to} H^1_c(Y)_q\oplus H^1_c(Y^c)_q \overset{\phi}{\to} H^1(\hX)_q.
\end{equation}
If $q^{k_1+\dots+k_r}\neq 1$ then $H^1_c(\partial Y)_q=\{0\}$, hence $\iota_Y$ is injective.
Assume now that $q^{k_1+\dots+k_r}=1$.
If $r=n$ then the conclusion follows from Lemma~\ref{lm:coh:preimage:D} and Lemma~\ref{lm:dual:border:neigh:infty}.
We are left with the case $q^{k_1+\dots+k_r}= 1$ and  $r<n$.
%Note that in this case $\dim H^1_c(\partial Y)_q=1$ and is generated by $g^*_r$ by the result of \cite[Chap. 3]{Menet:these}.

Let $\kk: H^1_c(\partial Y)_q \to H^1_c(Y)_q$ and $\kk': H^1_c(\partial Y)_q \to H^1_c(Y^c)_q$ be the morphisms induced by the inclusions of $\partial Y$ into $Y$ and $Y^c$ respectively.
Consider $Y$ as a subsurface of $\hY$. Since $\hY\setminus Y$ is a union of $d$ discs, we derive from the Mayer-Vietoris exact sequence associated to the decomposition $\hY=Y\cup(\hY\setminus Y)$ that $\kk$ is injective.

%From Lemma~\ref{lm:dual:border:neigh:infty} (applied to the surface $\hY$), we get that $\kk$ is injective.
%We claim that  $\kk'$ is injective as well.
Similarly, let $\hY'$ be the Riemann surface  constructed from the curve   $y^d=\prod_{i=r+1}^n(x-b_i)^{k_i}$. Then $Y^c$ can be identified with a subsurface of $\hY'$ such that $\hY'\setminus Y^c$ is the union of $d$ discs.
By  considering the Mayer-Vietoris exact sequence associated with the decomposition $\hY'=Y^c\cup (\hY'\setminus Y^c)$, we get that $\kk'$ is injective.

\medskip

Denote by $\iota_{Y^c}: H^1_c(Y^c)_q \to H^1(\hX)_q$ the morphism induced from the inclusion $Y^c\hookrightarrow \hX$.
By definition, the morphism $\phi: H^1(Y)_q\oplus H^1(Y^c)_q \to H^1(\hX)_q$ is given by $\phi(v,w)=\iota_Y(v)-\iota_{Y^c}(w)$ for all $(v,w)\in H^1_c(Y)_q\oplus H^1_c(Y^c)_q$.
In particular, we have $\iota_Y(v)=\phi(v,0)$. Assume now that $v\in \ker(\iota_{Y})$.
Then $(v,0) \in \ker(\phi)=\Im(\psi)$.
This means that there exists $u \in H^1_c(\partial Y)_q$ such that $(v,0)=(\kk(u),\kk'(u))$. But since $\kk'$ is injective, we must have $u=0$. Therefore, $v=\kk(u)=0$, which implies that $\iota_{Y}$ is injective. This completes the proof of (b).\\

\item[(c)] Let $\langle.,.\rangle_{\hY}$ denote the intersection form on $H^1(\hY)$.  %Let $\hY$  and $\iota'_Y: H^1_c(Y)_q \to H^1(\hY)_q$ be as above.
For any closed $1$-forms $\mu$ and $\eta$ with compact supports in $Y$ we have
    \begin{equation}\label{eq:inters:form:equal}
    \langle \mu,\eta\rangle=\frac{\imath}{2}\int_{\hX}\mu\wedge\bar{\eta}= \frac{\imath}{2}\int_{Y}\mu\wedge\bar{\eta}= \frac{\imath}{2}\int_{\hY}\mu\wedge\bar{\eta}=\langle \mu,\eta\rangle_{\hY}.
    \end{equation}
    If $q^{k_1+\dots+k_r}\neq 1$ then $H^1_c(Y)_q$ is isomorphic to both $\Vbb$ and  $H^1(\hY)_q$. Equality \eqref{eq:inters:form:equal} means that the Hermitian spaces $(\Vbb,\langle.,.\rangle)$ and $(H^1(\hY)_q,\langle.,.\rangle_{\hY})$ are isometric. Since $\langle.,.\rangle_{\hY}$ is non-degenerate by Theorem~\ref{th:Menet:dim:signature}, we conclude that the restriction of $\langle.,.\rangle$ to $\Vbb$ is also non-degenerate.
    In the case $q^{k_1+\dots+k_r}=1$ and $r=n$, we have $\hY=\hX$ hence $\Vbb=H^1(\hX)=H^1(\hY)$, and the same conclusion holds.
    Finally, in the case $r<n$ and $q^{k_1+\dots+k_r}=1$, by Lemma~\ref{lm:dual:border:neigh:infty}, $g^* _r$ is a generator of $H^1_c(\partial Y)_q$. Since every element $g' \in H^1_c(Y)_q$ can be represented by a closed $1$-form with support disjoint from $\partial Y$, we have  $\langle g^*_r, g'\rangle=0$.
\end{itemize}
\end{proof}

\begin{Remark}\label{rk:boundary:coh:class}
In the case $q^{k_1+\dots+k_r}=1$, one can directly check that $\langle g^*_r,g_i\rangle=0$ for all $i=1,\dots,r-1$ from the formulas in Theorem~\ref{th:Menet:generator:set}.
\end{Remark}

As a consequence of Lemma~\ref{lm:coh:preimage:disc}, we get the following
\begin{Corollary}\label{cor:Dehn:twist}
Assume that $r\geq 2$. Let $\gamma$ be the boundary of $E$, and $\tau_\gamma$ the Dehn twist about $\gamma$.
Denote by $\Vbb^\perp$ the orthogonal of $\Vbb$ in $H^1(\hX)_q$ with respect to the intersection form. Then we have
\begin{itemize}
\item[$\bullet$] The restriction of $\rho_q(\tau_\gamma)$ to $\Vbb$ is equal to $q^{k_1+\dots+k_r}\cdot \id_{\Vbb}$.\\

\item[$\bullet$] Assume that $q^{k_1+\dots+k_r} \neq 1$.
Then $\Vbb\simeq H^1_c(Y)_q$, $H^1(\hX)_q=\Vbb\oplus \Vbb^\perp$, and for all $v\in H^1(\hX)_q$
\begin{equation}\label{eq:Dehn:twist:act:dec}
\rho_q(\tau_\gamma)(v)=q^{k_1+\dots+k_r}\cdot v_0+v_1
\end{equation}
where  $v_0$ and $v_1$ are respectively the orthogonal projections of $v$ in $\Vbb$ and $\Vbb^\perp$.
\end{itemize}
\end{Corollary}
\begin{proof}
By Lemma~\ref{lm:coh:preimage:disc}(a), $H^1_c(Y)_q$  is generated by $\{g_1,\dots,g_{r-1}\}$.
It follows from Theorem~\ref{th:Menet:Dehn:twist} that $\rho_q(\tau_\gamma)_{|\Vbb}= q^{k_1+\dots+k_r}\cdot \id_{\Vbb}$.

Assume now that $q^{k_1+\dots+k_{r-1}}\neq 1$. Then the restriction of $\langle.,.\rangle$ to $\Vbb$ is non-degenerate by Lemma~\ref{lm:coh:preimage:disc}(c). Hence $H^1(\hX)_q$ is the direct sum of $\Vbb$ and $\Vbb^\perp$.
By direct calculations using the formulas of Theorem~\ref{th:Menet:generator:set} and Theorem~\ref{th:Menet:Dehn:twist}, one readily  checks that \eqref{eq:Dehn:twist:act:dec} holds for all $g_i, \; i=1,\dots,n-1$.  Since $\{g_1,\dots,g_{n-1}\}$ is a generating set   of $H^1(\hX)_q$, the corollary follows.
\end{proof}

\begin{Lemma}\label{lm:coh:subsurf:inject}
Let $E_0$ be a topological disc in $D$, and $E_1,\dots,E_\ell$ be $\ell$ discs pairwise disjoint in $E_0$. We suppose that the boundary of $E_j$ does not intersect the set $\Bc$. Let $I_0, I_1,\dots,I_\ell$ be the subsets of $\{1,\dots,n\}$ such that $b_i \in E_j$ if and only if $i \in I_j$. We suppose that $I_j \neq \vide$ and $I_j \subsetneq I_0$ for all $j=1,\dots,\ell$.

Let $E:=E_0\setminus\left(\cup_{j=1}^\ell E_j\right)$. Let $Y_0$ and $Y$ be the preimages of $E_0$ and $E$ in $\hX$ respectively.
Then the morphism  $\iota_{Y,Y_0}: H^1_c(Y)_q\to H^1_c(Y_0)_q$   which is induced by the inclusion $Y \hookrightarrow Y_0$ is injective.
\end{Lemma}
\begin{proof}
Let $Y_j$ be the preimage of $E_j$ in $\hX$ for $j=1,\dots,\ell$, and $Y':=\cup_{j=1}^\ell Y_j$.
Recall that $Y_0$ is an open subsurface of $\hX$. Therefore $H^0_c(Y_0)=\{0\}$. The Mayer-Vietoris sequence associated with the decomposition $Y_0=Y\cup Y'$ implies the following exact sequence
$$
0 \to H^1_c(\partial Y')_q \overset{\psi}{\to} H^1_c(Y)_q \oplus H^1_c(Y')_q \overset{\phi}{\to} H^1_c(Y_0)_q.
$$
We have $H^1_c(\partial Y')_q=\oplus_{j=1}^\ell H^1_c(\partial Y_j)_q$ and  $H^1_c(Y')_q=\oplus_{j=1}^\ell H^1_c(Y_j)_q$.
Let $\kk: H^1_c(\partial Y')_q \to H^1_c(Y)_q$ and $\kk': H^1_c(\partial Y')_q \to H^1_c(Y')_q$ be the morphisms induced by the inclusions of $\partial Y'$ into $Y$ and $Y'$ respectively.  For all $(u_1,\dots,u_\ell) \in \oplus_{j=1}^\ell H^1_c(\partial Y_j)_q$, we have $\kk'(u_1,\dots,u_\ell)=(\kk'_1(u_1),\dots,\kk'_\ell(u_\ell))$, where $\kk'_j: H^1_c(\partial Y_j)_q \to H^1_c(Y_j)_q$ is the inclusion morphism.
By the same argument as in Lemma~\ref{lm:coh:preimage:disc} (b), one readily shows that all the $\kk'_j$'s are injective,  hence $\kk'$ is injective.

%We claim that  $\kk'_j$ is injective for all $j\in \{1,\dots,\ell\}$.
%To see this, we embed $Y_j$ in the compact Riemann surface $\hY_j$ constructed from the plane curve   $y^d=\prod_{i\in I_j}(x-b_i)^{k_i}$.
%The complement $U_j$ of $Y_j$ in $\hY_j$ is a finite disjoint union of discs.
%%Therefore  $H^1_c(U_j)_q= H^1_c(U_j)=\{0\}$ (where $H^1_c(U_j)_q$ is the $q$-eigenspace of the action of $T^*_{\hY_j}$ on $H^1_c(U_j)$).
%The Mayer-Vietoris exact sequence associated to the decomposition $\hY_j=Y_j\cup U_j$ gives
%$$
%0 \to H^1_c(\partial Y_j)_q  \to H^1_c(Y_j)_q\oplus H^1_c(U_j)_q \simeq H^1_c(Y_j)_q
%$$
%which means that $\kk'_j$ is injective. As an immediate consequence, we get that $\kk'$ is injective.

Denote by $\iota_{Y',Y_0}: H^1_c(Y')_q \to H^1_c(Y_0)_q$ the morphism induced from the inclusion $Y'\hookrightarrow Y_0$.  By definition, for all $(v,w)\in H^1_c(Y)_q\oplus H^1_c(Y')_q$, $\phi(v,w)=\iota_{Y,Y_0}(v)-\iota_{Y',Y_0}(w)$. In particular, we have $\iota_{Y,Y_0}(v)=\phi(v,0)$. Assume now that $v\in \ker(\iota_{Y,Y_0})$. Then $(v,0) \in \ker(\phi)=\Im(\psi)$. This means that there exists $u \in H^1_c(\partial Y')_q$ such that $(v,0)=(\kk(u),\kk'(u))$. But since $\kk'$ is injective, we must have $u=0$. Therefore, $v=\kk(u)=0$, which implies that $\iota_{Y,Y_0}$ is injective.
\end{proof}

The following proposition generalizes Lemma~\ref{lm:coh:preimage:disc}.
\begin{Proposition}\label{prop:coh:subsurf:embedding}
Let $E_0,E_1,\dots,E_\ell, Y_0, Y$ be as in Lemma~\ref{lm:coh:subsurf:inject}.
Define $\hk_j:=\sum_{i \in I_j}k_i, \;  j=0,1,\dots,\ell$. Then we have
\begin{equation}\label{eq:subsurf:partition:dim:1}
\dim H^1_c(Y)_q=\left\{\begin{array}{ll}
\ell & \text{ if } I_0=\cup_{j=1}^\ell I_j \text{ and }  q^{\hk_j}=1 \; \text{ for all } j=1,\dots,\ell,\\
|I_0|-\sum_{j=1}^\ell |I_j|+\ell-1 & \text{ otherwise}.
\end{array}
\right.
\end{equation}
Let $\Vbb$ denote the image of $H^1_c(Y)_q$ under the natural morphism $\iota_Y: H^1_c(Y)_q \to H^1(\hX)_q$. Then
\begin{equation}\label{eq:subsurf:partition:dim:2}
\dim \Vbb = \left\{
\begin{array}{ll}
\dim H^1_c(Y)_q-1 & \text{ if } I_0=\{1,\dots,n\} \text{ and } q^{\hk_0} = 1, \\
\dim H^1_c(Y)_q & \text{ otherwise.}
\end{array}
\right.
\end{equation}
Moreover, the restriction of $\langle.,.\rangle$ to $\Vbb$ is non-degenerate if and only if $q^{\hk_j}\neq 1$ for all $j=1,\dots,\ell$ and either $q^{\hk_0}\neq 1$, or $q^{\hk_0}=1$ and $I_0=\{1,\dots,n\}$.
\end{Proposition}
\begin{proof}
Let $Y_j$ be the preimage of the interior of $E_j$ in $\hX$ for $j=1,\dots,\ell$.
Since $Y_j$ is an open subsurface of $\hX$ we have $H^0_c(Y_j)=0$, for $j=0,\dots,\ell$.
Since the $k_i$'s are not  divisible by $d$,  we have $|\pi_0(Y_j)|< d$ and hence $H^2_c(Y_j)_q=\{0\}$ by Lemma~\ref{lm:coh:annulus}.

Consider the Mayer-Vietoris sequence associated with the decomposition $Y_0=Y\cup Y'$,  where $Y'=\cup_{j=1}^\ell Y_j$. Since $H^0_c(Y_j)=H^2_c(Y_j)_q=\{0\}$ for all $j=0,\dots,\ell$, we get the following exact sequence
$$
0 \to H^1_c(\partial Y')_q \to H^1_c(Y)_q \oplus H^1_c(Y')_q \to H^1_c(Y_0)_q \to H^2_c(\partial Y')_q \to H^2_c(Y)_q \to 0.
$$
Note that $\dim H^1_c(\partial Y')_q = \dim H^2_c(\partial Y')_q$ by Lemma~\ref{lm:coh:annulus}(ii), while $\dim H^1_c(Y_j)_q=|I_j|-1$ by Lemma~\ref{lm:coh:preimage:disc}. Therefore
$$
\dim H^1_c(Y)_q=|I_0|-\sum_{j=1}^\ell|I_j|+\ell-1+\dim H^2_c(Y)_q.
$$
By Lemma~\ref{lm:coh:annulus}(i),
$$
\dim H^2_c(Y)_q=\left\{
\begin{array}{ll}
1 & \text{ if } |\pi_0(Y)|=d,\\
0 & \text{ otherwise.}
\end{array}
\right.
$$
The number of components of $Y$ is equal to the number of components of the Riemann surface $\hY$ constructed from the curve defined by
\begin{equation}\label{eq:curve:assoc:partition}
y^d=\prod_{j=1}^\ell(x-\hb_j)^{\hk_j}\cdot \prod_{i\in I_0\setminus(\cup_{j=1}^\ell I_j)} (x-b_i)^{k_i}
\end{equation}
where $\hb_j$ is an arbitrary point in the disc $E_j$. The number of components of $\hY$ is $d$ if and only if all the exponents on the right hand side of \eqref{eq:curve:assoc:partition} are divisible by $d$.
Recall that the $0< k_i < d$ for all $i=1,\dots,n$.
Therefore, $\dim H^2_c(Y)_q=1$ if and only if $I_0=I_1\cup\dots\cup I_\ell$, and $d  \; | \; \hk_j$ for all  $j=1,\dots,\ell$. This proves \eqref{eq:subsurf:partition:dim:1}.

Since $\iota_Y=\iota_{Y_0}\circ\iota_{Y,Y_0}$, \eqref{eq:subsurf:partition:dim:2} follows from Lemma~\ref{lm:coh:subsurf:inject} and  Lemma~\ref{lm:coh:preimage:disc}(b).
It remains to show the last assertion on the restriction of $\langle.,.\rangle$ to $\Vbb=\iota_Y(H^1_c(Y)_q)$.
Consider $Y$ as a subsurface of the Riemann surface $\hY$ obtained from \eqref{eq:curve:assoc:partition}, and let
$\iota'_Y: H^1_c(Y)_q\to H^1(\hY)_q$ be the inclusion morphism.
Note that the complement of $Y$ in $\hY$ is a disjoint union of discs. Thus, it follows from the Mayer-Vietoris exact sequence that
$$
\ker(\iota'_Y)= H^1_c(\partial Y)_q = \bigoplus_{j=0}^\ell H^1_c(\partial Y_j)_q.
$$
Let $\langle.,.\rangle_{\hY}$ denote the intersection form on $H^1(\hY)$.
By construction, for all $\mu,\eta\in H^1_c(Y)$,  we have
\begin{equation}\label{eq:inters:form:equal:V:hY}
\langle \iota_Y(\mu),\iota_Y(\eta)\rangle = \langle \iota'_Y(\mu),\iota'_Y(\eta)\rangle_{\hY}.
\end{equation}
If $q^{\hk_j}\neq 1$ for all $j=0,\dots,\ell$, then $\Vbb\simeq  H^1_c(Y)_q \simeq H^1(\hY)_q$.
It follows from \eqref{eq:inters:form:equal:V:hY} that the Hermitian spaces $(\Vbb,\langle.,.\rangle)$ and  $(H^1(\hY)_q,\langle.,.\rangle_{\hY})$ are isometric.
Since $(H^1(\hY)_q,\langle.,.\rangle_{\hY})$ is non-degenerate by Theorem~\ref{th:Menet:dim:signature}, so is $(\Vbb,\langle.,.\rangle)$.

In the case  $q^{\hk_j}\neq 1$ for all $j=1,\dots,\ell$, but $q^{\hk_0}=1$ and $I_0=\{1,\dots,n\}$, we have $\Vbb\simeq H^1(\hY)_q\simeq H^1_c(Y)_q/H^1_c(\partial Y_0)_q$, and the same conclusion holds.
Finally, in the case $q^{\hk_j}=1$ for some $j\in \{1,\dots,\ell\}$, or $q^{\hk_0}=1$ and $|I_0| < n$, then there is an element $\mu \in H^1_c(\partial Y)_q \subset H^1_c(Y)_q$ such that $\iota_{Y}(\mu) \neq 0 \in H^1(\hX)_q$, but $\iota'_Y(\mu)=0 \in H^1(\hY)_q$.
For all $\eta\in H^1_c(Y)_q$, we have
$$
\langle \iota_Y(\mu),\iota_Y(\eta)\rangle= \langle \iota'_Y(\mu),\iota'_Y(\eta)\rangle_{\hY}=\langle 0, \iota'_Y(\eta)\rangle = 0
$$
which means that the restriction of $\langle.,.\rangle$ to $\Vbb$ is degenerate.
\end{proof}

\subsection{A consequence of the lantern relation}\label{subsec:lantern:rel}
Let $S$ be a four-holed sphere, that is $S$ is homeomorphic to the sphere with four discs removed. We can represent $S$ as the complement of the union of three disjoint (closed) discs  $E_1, E_2, E_3$ in a disc $E_0$.
In what follows, if $\gamma$ is a simple closed curve in $S$, then we denote by $\tau_\gamma$ the Dehn twist about $\gamma$ considered as an element of $\MCG(S)$.

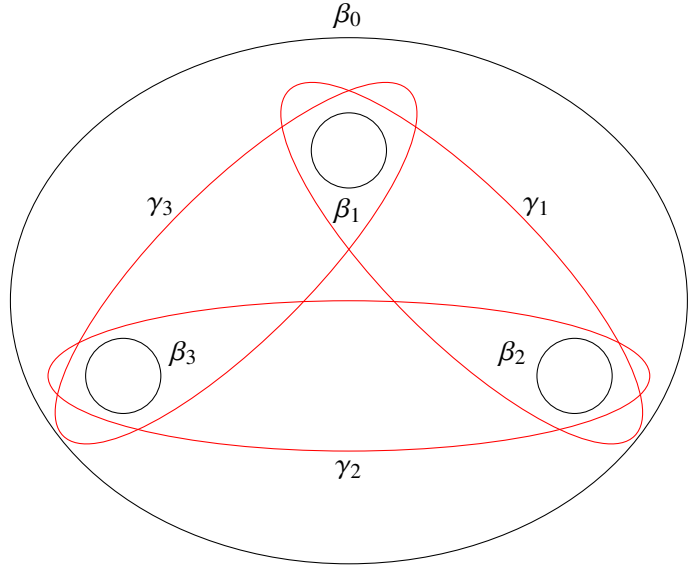
\begin{figure}[htb]
\begin{tikzpicture}[scale=0.5]
\foreach \x in {(6,0), (-6,0), (0,6)} \draw \x circle (1);
\draw[red] (0,0) ellipse (8 and 2);

\draw[rotate around={135:(3,3)}, red] (3,3) ellipse (6.5 and 2);
\draw[rotate around={45:(-3,3)}, red] (-3,3) ellipse (6.5 and 2);

\draw (0,2) ellipse (9 and 7);
\draw (0,9) node[above] {$\beta_0$};
\draw (0,5) node[below] {$\beta_1$};
\draw (5,0) node[above left]   {$\beta_2$};
\draw (-5,0) node[above right]   {$\beta_3$};
\draw (5,4.5) node   {$\gamma_1$};
\draw (0,-2.5) node   {$\gamma_2$};
\draw (-5,4.5) node   {$\gamma_3$};
\end{tikzpicture}
\caption{The lantern relation}
\label{fig:lantern:rel}
\end{figure}
Let $\beta_0,\dots,\beta_3$ denote the borders of $E_0, \dots, E_3$ respectively.
Let $\gamma_1,\gamma_2,\gamma_3$ be three simple closed curves in $E_0$ such that $\gamma_i$ borders a disc that contains $E_i$ and $E_{i+1}$ but not $E_{i+2}$ with the convention $E_j\equiv E_{j-3}$ if $j>3$. The configuration of $\gamma_1,\gamma_2,\gamma_3$ is shown in Figure~\ref{fig:lantern:rel}. The lantern relation (see~\cite{FM12,Dehn,J79}) asserts that we have
\begin{equation}\label{eq:lantern}
\tau_{\gamma_1}\cdot \tau_{\gamma_2}\cdot \tau_{\gamma_3}=\tau_{\beta_0}\cdot \tau_{\beta_1}\cdot \tau_{\beta_2}\cdot \tau_{\beta_3}
\end{equation}
in $\Mod(S)$.

Now let $E_0$ be an open disc in $D$ whose boundary does not intersect the set $\Bc$. Let $I_0 \subset \{1,\dots,n\}$ be the set of indices $i$ such that $b_i\in E_0$. Let $E_1,E_2,E_3$ be three open disjoint  discs  in $E_0$ whose boundaries are also disjoint. We suppose that each of $E_1, E_2, E_3$ contains some points in the set $\Bcal$ and that $\Bc\cap E_0$ is contained in $E_1\cup E_2 \cup E_3$. Let $I_1,I_2,I_3$ be the sets of indices $i$ such that $b_i$ is contained in $E_1,E_2,E_3$ respectively.  Define
$$
\hk_j:=\sum_{i\in I_j}k_i, \; j=0,\dots,3.
$$
The lantern relation \eqref{eq:lantern} implies the  following
\begin{Proposition}\label{prop:lantern:conseq}
Let $E:=E_0\setminus(\cup_{i=1}^3\ol{E}_i)$, and $\gamma_1,\gamma_2,\gamma_3$ be the simple closed curve as in Figure~\ref{fig:lantern:rel}. Denote by $Y$ the preimage of $E$ in $\hX$, and by $\Vbb$ the image of $H^1_c(Y)_q$ under the inclusion morphism $\iota_Y: H^1_c(Y)_q \to H^1(\hX)_q$.
%Let $\tilde{\tau}_{\gamma_j}$ be the lift of $\tau_{\gamma_j}$ to $\hX$ that is identity in a neighborhood of the preimage of $\infty$.
Assume that $q^{\hk_j}\neq 1$ for all $j=0,\dots,3$.
Then we have
\begin{itemize}
\item[(i)] $\dim \Vbb=2$ and the restriction of the intersection form $\langle.,.\rangle$ to $\Vbb$ is non-degenerate.

\item[(ii)] $\Vbb$ is preserved by the action of $\rho_q(\tau_{\gamma_1}), \rho_q(\tau_{\gamma_2}), \rho_q(\tau_{\gamma_3})$ and
$$
\left(\rho_q(\tau_{\gamma_3})\cdot\rho_q(\tau_{\gamma_2})\cdot\rho_q(\tau_{\gamma_1})\right)_{|\Vbb}=q^{\hk_0}\cdot \id_\Vbb.
$$
\end{itemize}
\end{Proposition}
\begin{proof}
That $\dim \Vbb=2$ and the restriction of $\langle.,.\rangle$ to $\Vbb$ is non-degenerate follows from Proposition~\ref{prop:coh:subsurf:embedding}.
Since the Dehn twists $\tau_{\gamma_j}$'s preserve $E$, their lifts preserve $Y$. Therefore, the space $H^1_c(Y)_q$ is invariant by the actions of $\rho_q(\tau_{\gamma_j})$'s.  The lantern relation \eqref{eq:lantern} then implies that
$$
(\rho_q(\tau_{\gamma_3})\cdot\rho_q(\tau_{\gamma_2})\cdot\rho_q(\tau_{\gamma_1}))_{|\Vbb}=(\rho_q(\tau_{\beta_3})\cdot\rho_q(\tau_{\beta_2})\cdot\rho_q(\tau_{\beta_1})\cdot\rho_q(\tau_{\beta_0}))_{|\Vbb}.
$$
Up to homotopy, $\tau_{\beta_j}$ is identity outside the disc $E_j$. By definition, the lift $\tilde{\tau}_{\beta_j}$ of $\tau_{\beta_j}$ is identity in a neighborhood of the preimage of $\infty$ in $\hX$. Therefore, for $j=1,2,3$, $\tilde{\tau}_{\beta_j}$ is identity on $Y$. It follows that  $\tilde{\tau}^*_{\beta_j|\Vbb}=\id_{\Vbb}$ for $j=1,2,3$.
Let $Y_0$ be the preimage of $E_0$ in $\hX$, and $\Vbb_0$  the image of $H^1_c(Y_0)_q$ in $H^1(\hX)_q$. By Corollary~\ref{cor:Dehn:twist}, we know that the restriction of $\rho_q(\tau_{\beta_0})$ to $\Vbb_0$ is equal to $q^{\hk_0}\cdot\id_{\Vbb_0}$. Since $\Vbb$ is a subspace of $\Vbb_0$, we get that
$$
(\rho_q(\tau_{\gamma_3})\cdot\rho_q(\tau_{\gamma_2})\cdot\rho_q(\tau_{\gamma_1}))_{|\Vbb}=\rho_q(\tau_{\beta_0})_{|\Vbb}= q^{\hk_0}\cdot \id_\Vbb.
$$
\end{proof}

%****************************************************************************
%****************************************************************************
\section{Zariski closures in unitary groups}~\label{sec:preparation:unitary:grp}
In this section, we collect some elementary results on unitary Lie groups (that is matrix groups that preserve some non-degenerate Hermitian form) and their Lie algebra.
In what follows, $V$ will be a  $\C$-vector space of dimension $r \geq 2$ endowed with a non-degenerate Hermitian form $\langle.,.\rangle$. Denote by $\U(V)$ (resp. by $\SU(V)$) the group of automorphisms of $V$ that preserve $\langle.,.\rangle$  (resp. and have determinant $1$). Recall that both $\U(V)$ and $\SU(V)$ are real Lie groups, and that the Lie algebra of $\SU(V)$, denoted by $\su(V)$,  is a simple real Lie algebra.
We denote by $\U_d(V)$ the subgroup of $\U(V)$ which consists of elements $A\in \U(V)$ such that $\det(A)^d=1$. This is a Lie subgroup of $\U(V)$ whose identity component is equal to $\SU(V)$.

\medskip

Let $\Gamma$ be a subgroup of $\U_d(V)$. Denote by $\Gb$ the Zariski closure of $\Gamma$ in $\U(V)$, and by $\Gb_0$ the identity component of $\Gb$. Note that $\Gb_0$ is a Lie subgroup of $\SU(V)$. We denote by $\gg$ its Lie algebra.

\begin{Lemma}\label{lm:density:sign:1:1}
Assume that $\dim V=2$ and that  $\Gamma$ is generated by three elements $\alpha,\beta,\gamma$ in $\U_d(V)$ such that
\begin{itemize}
\item[(i)] Each of $\{\alpha,\beta,\gamma\}$ is either of finite order or unipotent, none of them is multiple of identity.

\item[(ii)] There are three distinct points $a,b,c$ in $\Pb(V)$ such that $a$ is fixed by $\alpha$, $b$ by $\beta$, and $c$ by $\gamma$.

\item[(iii)] There exists $\lambda \in \C^*$ such that $\alpha\cdot\beta\cdot\gamma=\lambda\Id_V$.
\end{itemize}
Then $\Gb_0=\SU(V)$  if $\langle.,.\rangle$ has signature $(1,1)$.
\end{Lemma}
\begin{proof}
We first show that $\Gamma$ is infinite. Assume that $\Gamma$ is finite, then so is it image in $\PU(V)$. Since $\langle.,.\rangle$ has signature $(1,1)$, we have $\PU(V)\simeq \PU(1,1)$. It is a well known fact that all elements in any finite subgroup of $\PU(1,1)$ have two common fixed points in $\CP^1$. But this would lead to a contradiction to the hypothesis (ii), hence the claim is proved.
%Since none of $\{\alpha,\beta,\gamma\}$ is a multiple of identity, each of them has at most two fixed points in $\Pb(V)$. But since we have three distinct fixed points for the actions of $\alpha,\beta,\gamma$ on $\Pb(V)$, they cannot have two common fixed points and the claim follows.

Since $\Gamma$ is infinite, its Zariski closure $\Gb$ is a Lie subgroup of positive dimension in $\U_d(V)$. In particular, the identity component $\Gb_0$ of $\Gb$ is a Lie subgroup of positive dimension in $\SU(V)\simeq \SU(1,1) \simeq \SL(2,\R)$.

Assume that $\Gb_0 \varsubsetneq \SL(2,\R)$. Then up to a conjugation, $\Gb_0$ is one of the following subgroups of $\SL(2,\R)$: $\SO(2,\R)$, $A=\left\{\left(\begin{smallmatrix} e^{t} & 0 \\ 0 & e^{-t}\end{smallmatrix} \right), \; t \in \R \right\}, \; N=\left\{ \left( \begin{smallmatrix} 1 & * \\ 0 & 1 \end{smallmatrix}\right)\right\}, \; P= \left\{\left(\begin{smallmatrix} * & * \\ 0 & * \end{smallmatrix} \right)\right\}$.

\begin{itemize}
\item[$\bullet$] If $\Gb_0\simeq \SO(2,\R)$, then all the elements of $\Gb_0$ fix two points $\C\cdot v$ and $\C\cdot v'$ in $\Pb(V)$, for some $v,v' \in V$ such that $\langle v,v \rangle >0$ and $\langle v', v' \rangle < 0$. By definition $\alpha,\beta,\gamma$ normalize $\Gb_0$. Therefore, each of $\{\alpha,\beta,\gamma\}$ stabilizes the set $\{\C\cdot v, \C\cdot v'\}$. But $\alpha,\beta,\gamma$ can not exchange the lines $\C\cdot v$ and $\C \cdot v'$, since they must preserve the norms of the vectors in $V$ with respect to $\langle .,.\rangle$. Therefore, all of $\alpha,\beta,\gamma$ fix $\C\cdot v$ and $\C\cdot v'$. This is again impossible because there are $3$ distinct points in $\Pb(V)$ that are fixed points of one of $\alpha,\beta,\gamma$. Thus, we conclude that $\Gb_0\not\simeq \SO(2,\R)$. \\

\item[$\bullet$] If $\Gb_0\simeq N$ or $\Gb_0\simeq P$ then all the elements of $\Gb_0$ fix a unique point $\C\cdot v \in \Pb(V)$, where $\langle v,v \rangle=0$. Since $\alpha,\beta,\gamma$ normalize $\Gb_0$, they  also fix the line  $\C\cdot v$. We claim that $\C\cdot v$ is the unique fixed point of $\alpha,\beta,\gamma$ in $\Pb(V)$. If $\alpha$ fixes another point in $\Pb(V)$ then it is diagonalizable. Therefore, it cannot be unipotent. By assumption, $\alpha$  must be of finite order. But any element of finite order in $\U(1,1)$ cannot fix a point $\C\cdot v \in \Pb(V)$ with $\langle v, v\rangle=0$. Thus $\alpha$ is unipotent and $\C\cdot v$ is the unique fixed point of $\alpha$ in $\Pb(V)$. The same argument holds for $\beta$ and $\gamma$. But we have again a contradiction to the  assumption that there are three distinct points in $\Pb(V)$ each of which is fixed by one of $\{\alpha,\beta,\gamma\}$.  We can then conclude that $\Gb_0\not\simeq N$ and $\Gb_0\not\simeq P$. \\

\item[$\bullet$] Finally, if $\Gb_0\simeq A$, then all the elements of $\Gb_0$ fix two points $\C\cdot v$ and $\C\cdot v'$ in $\Pb(V)$, where $\langle v,v\rangle =\langle v',v'\rangle =0$. Since $\alpha,\beta,\gamma$ normalize $\Gb_0$, they stabilize $\{\C\cdot v, \C\cdot v'\}$. We have seen that each of $\alpha,\beta,\gamma$ cannot have two fixed points in the set $\{\C\cdot w \in \Pb(V), \; \langle w,w\rangle=0\}$. Therefore, all of them exchange the lines $\C\cdot v$ and $\C\cdot v'$. But we have by assumption $ \alpha\cdot\beta\cdot\gamma = \lambda \Id$, which implies that the action of $\alpha\cdot\beta\cdot\gamma$ on $\Pb(V)$ is identity. We thus have again a contradiction, which means that  $\Gb_0\not\simeq A$, and hence we must have $\Gb_0 = \SU(V)$.
\end{itemize}
\end{proof}

\begin{Lemma}\label{lm:density:sign:2:0}
Let $V, \Gamma, \alpha, \beta, \gamma$ be as in Lemma~\ref{lm:density:sign:1:1}. Assume that the signature of $\langle.,.\rangle$ is $(2,0)$ or $(0,2)$. Let $\ol{\alpha}, \ol{\beta}, \ol{\gamma}$ be the projections of $\alpha,\beta,\gamma$ in $\PU(V)$, and  $n_\alpha, n_\beta, n_\gamma$ be the orders of $\ol{\alpha}, \ol{\beta},\ol{\gamma}$ respectively. Without loss of generality, we can assume that $n_\alpha \geq n_\beta \geq n_\gamma$. If we have  $n_\alpha >5$ and $n_\beta >2$, then $\Gb_0=\SU(V)$.
\end{Lemma}
\begin{proof}
In this case $\alpha, \beta,\gamma$ cannot be unipotent, hence $n_\alpha,n_\beta,n_\gamma$ are all finite.
We first show that the projection $\ol{\Gamma}$ of $\Gamma$ in $\PU(V)\simeq \PU(2)$ is infinite. Assume that $\ol{\Gamma}$ is finite, then it follows from the classification of finite groups acting on $\CP^1$, we have that $\ol{\Gamma}$ is isomorphic to one of the following finite groups: the cyclic group $\Z/p\Z$  with $p \in \Z_{\geq 2}$, the dihedral group $D_{2p}$ with $p \in \Z_{\geq 2}$,  the alternating group $\Acal_4$, the symmetric group $\Scal_4$, and the alternating group $\Acal_5$.
Since the order of any element of $\Acal_4, \Scal_4, \Acal_5$ cannot be greater than $5$, it follows that $\ol{\Gamma}$ is not isomorphic to one of $\{\Acal_4, \Scal_4, \Acal_5\}$.

\begin{itemize}
\item[$\bullet$] If $\ol{\Gamma}\simeq \Z/p\Z$, then all the elements of $\ol{\Gamma}$ have two common fixed points in $\Pb(V)$ which contradicts the hypothesis that there are three distinct points $a,b,c \in \Pb(V)$ each of which is fixed by one of $\ol{\alpha}, \ol{\beta}, \ol{\gamma}$.

\item[$\bullet$] Assume that $\ol{\Gamma}\simeq D_{2p}, \; p \in\Z_{\geq 2}$. Since $n_\alpha >5$ and $n_\beta >2$, $\ol{\alpha}$  and $\ol{\beta}$ must be  rotations. But we have $\ol{\alpha}\cdot\ol{\beta}\cdot\ol{\gamma}=\id$. Therefore $\ol{\gamma}$ is also a rotation, which means that $\ol{\Gamma}$ is contained in $\Z/p\Z \varsubsetneq D_{2p}$. This contradiction implies that $\ol{\Gamma} \not\simeq D_{2p}$.  Thus, $\ol{\Gamma}$ must be infinite.
\end{itemize}

Since $\Gamma$ is infinite, its Zariski closure $\Gb$ is a Lie subgroup of $\U(V)$ of positive dimension. The identity component $\Gb_0$ of $\Gb$ is a Lie subgroup of $\SU(V)$.  All non-trivial proper Lie subgroup of $\SU(V)\simeq \SU(2)$ is isomorphic to $\SO(2,\R)$. This implies in particular that all elements of $\Gb_0$ have two common fixed points $e,e'$ in $\Pb(V)$. Since $\ol{\alpha}, \ol{\beta},\ol{\gamma}$ normalize $\Gb_0$, they stabilize the set $\{e,e'\}$. Note that any element of $\SU(V)$ that fixes both $e$ and $e'$ is contained in $\Gb_0$.

Since we have $\ol{\alpha}\cdot\ol{\beta}\cdot\ol{\gamma}=\id$, either all of them fixe both $e$ and $e'$, or exactly two them permute $e$ and $e'$.  If $\ol{\alpha},\ol{\beta},\ol{\gamma}$ fix both $e$ and $e'$ then we have a contradiction to the assumption that there are three distinct points each of which is fixed by one of $\ol{\alpha},\ol{\beta},\ol{\gamma}$. Hence $e$ and $e'$ are permuted by two transformations in $\ol{\alpha},\ol{\beta}, \ol{\gamma}$. We now claim that if $\varphi \in \PU(V)$ exchanges $e$ and $e'$, then $\varphi^2=\id$. This is because $\varphi^2$ has four fixed points in $\Pb(V)$: the  two fixed points of $\varphi$ and $\{e,e'\}$. This implies that two of $\ol{\alpha},\ol{\beta},\ol{\gamma}$ have order $2$. Since this is excluded by the hypothesis on $n_\alpha, n_\beta, n_\gamma$, we must have $\Gb_0=\SU(V)$.
\end{proof}

The following result plays a key role in the proof of Theorem~\ref{th:main:Zariski:dense} and Theorem~\ref{th:main:Zar:dense:bis}.

\begin{Proposition}\label{prop:density:induction}
Let $V'$ be proper subspace of $V$ such that $\dim V' \geq 2$ and the restriction of $\langle.,.\rangle$ to  $V'$  is non-degenerate.
Let $V'':=V'{}^\perp$. We embed $\U(V')$ (resp. $\U(V'')$) into $\U(V)$ be setting its action on $V''$  (resp. on $V'$) to be identity.
Assume  that there exist $\gamma\in \Gamma, \; v\in V\setminus\{0\}$, and $\lambda \in \C^*$ such that
\begin{itemize}
\item[(i)] $\gamma(x)=x +\lambda\cdot\langle x, v\rangle v$ for all $x \in V$,

\item[(ii)] $V' \nsubset v^\perp$ and $V'' \nsubset v^\perp$.
\end{itemize}
Then if $\Gb_0$ contains $\SU(V')\times\SU(V'')$, then $\Gb_0=\SU(V)$.
\end{Proposition}
\begin{proof}
We first remark that $V=V'\oplus V''$ and the restriction of $\langle.,.\rangle$ to $V''$ is also non-degenerate.   Let $v'$ and $v''$ be respectively the projections of $v$ in $V'$ and $V''$. Note that $v'\neq 0$ and $v''\neq 0$ (otherwise either $V''\subset v^\perp$ or $V'\subset v^\perp$).
Since the restriction of $\langle.,.\rangle$ to $V'$ (resp. to $V''$) is non-degenerate, there exists $w'\in V'$ (resp. $w''\in V''$) such that $\langle w',v'\rangle = \langle w',v\rangle \neq 0$ (resp. $\langle w'',v''\rangle=\langle w'',v\rangle \neq 0$).

Let $r= \dim V'$ and $r''=\dim V''$. Define $W':=V'\cap v^{\perp}$  and $W'':= V''\cap v^\perp$.
We then have $\dim W'=r'-1$ and $\dim W''=r''-1$.
We choose a basis $\{w'_1,\dots,w'_{r'}\}$ of $V'$ as follows:
\begin{itemize}
\item[$\bullet$] if $\langle v', v \rangle = \langle v', v' \rangle \neq 0$ then $\{w'_1,\dots,w'_{r'-1}\}$  is a basis of $W'$, and $w'_{r'}=v'$.

\item[$\bullet$] if $\langle v',v \rangle =0$, then $w'_1=v'$, $\{w'_1,\dots,w'_{r'-1}\}$ is a basis of $W'$, and $w'_{r'}=w'$.
\end{itemize}
We choose a basis $\{w''_1,\dots,w''_{r''}\}$ of $V''$ in a similar way,  namely
\begin{itemize}
\item[$\bullet$] if $\langle v'', v \rangle = \langle v'', v'' \rangle \neq 0$ then $\{w''_1,\dots,w''_{r''-1}\}$  is a basis of $W''$ and $w''_{r''}=v''$.

\item[$\bullet$] if $\langle v'',v \rangle =0$, then $w''_1=v''$, $\{ w''_1,\dots,w''_{r''-1}\}$ is a basis of $W''$, and $w''_{r''}=w''$.
\end{itemize}
By (i), we have $\gamma(w'_i)=w'_i$ for $i=1,\dots,r'-1$, and $\gamma(w''_i)=w''_i$, for $i=1,\dots,r''-1$. We also have
$$
\gamma(w'_{r'})=w'_{r'}+\lambda\langle w'_r, v \rangle (v'+v''), \text{ where } \lambda\langle w'_r, v \rangle \neq 0,
$$
and
$$
\gamma(w''_r)= w''_r+ \lambda \langle w''_r,v \rangle (v'+v''), \text{ where } \lambda \langle w''_r,v \rangle  \neq 0.
$$
Therefore, the matrix $R$ of $\gamma$ in the basis $\Wc:=\{w'_1,\dots, w'_{r'},w''_1,\dots,w''_{r''}\}$ has the form
\begin{itemize}
\item[(a)]  $\small R = \left(\begin{array}{cccc} I_{r'-1} & 0 & 0 & 0 \\ 0  & a & 0 & c \\ 0  & 0 & I_{r''-1}  &  0 \\ 0  & b  & 0  & d \end{array} \right)$ if $\langle v',v\rangle \neq 0$ and $\langle v'',v\rangle \neq 0$, \\

\item[(b)] $\small R = \left(\begin{array}{cccc} I_{r'-1} & 0 & 0 &  0 \\ 0  & a & 0  &  c \\ 0  & b & I_{r''-1}  &  d \\ 0  & 0  & 0 & 1 \end{array} \right)$ if $\langle v',v \rangle \neq 0$ and $\langle v'',v\rangle=0$,\\

\item[(c)]  $\small R = \left(\begin{array}{cccc} I_{r'-1} & a & 0 & c \\ 0  & 1 & 0 & 0 \\ 0  & 0 & I_{r''-1}  &  0 \\ 0  &   b  & 0  & d \end{array} \right)$ if $\langle v',v\rangle=0$ and $\langle v'',v \rangle \neq 0$,\\

\item[(d)]  $\small R = \left(\begin{array}{cccc} I_{r'-1} & a & 0 & c \\ 0  & 1 & 0 & 0 \\ 0  & b & I_{r''-1}  &  d \\ 0  & 0 & 0  & 1 \end{array} \right)$ if $\langle v',v\rangle =\langle v'',v\rangle=0$,
\end{itemize}
with $a,b,c,d\in \C$. Note that we have  $bc\neq 0$ in all cases.
We will only provide the proof for  case (a),  the other cases follow from similar arguments.

\medskip

Let $\gg_\C:= \gg\otimes_\R\C$ be the complexified Lie algebra of $\gg$.
By assumption $\gg$ contains $\su(V')\oplus\su(V'')$. Therefore, $\gg_\C$ contains $\ssl(V')\oplus \ssl(V'')\simeq \ssl(r',\C)\oplus\ssl(r'',\C)$.
We will show that $\gg_\C \simeq \ssl(V)\simeq \ssl(r,\C)$, which is enough to conclude because if $\Gb_0 \subsetneq \SU(V)$, then $\gg\subsetneq \su(V)$, and $\gg_\C\subsetneq \su(V)\otimes_\R\C \simeq \ssl(r,\C)$. To this purpose, we first consider the adjoint action of $\gamma$ on $\gg_\C$.
Note that in  case (a) the matrix of $\gamma^{-1}$ is $R^{-1}=\left(\begin{smallmatrix} I_{r'-1} & 0 & 0 & 0 \\ 0  & a' & 0 & c' \\ 0  & 0 & I_{r''-1}  &  0 \\ 0  & b'  & 0  & d' \end{smallmatrix} \right)$ with $b'c'\neq 0$.

\medskip

Let us write $R=\left(\begin{smallmatrix} A & C \\ B & D \end{smallmatrix} \right)$ and $R^{-1}=\left(\begin{smallmatrix} A' & C' \\ B' & D' \end{smallmatrix} \right)$, where
\begin{itemize}
\item[$\bullet$] $A=\left(\begin{smallmatrix}I_{r'-1} & 0 \\ 0 & a \end{smallmatrix} \right), A'=\left(\begin{smallmatrix}I_{r'-1} & 0 \\ 0 & a' \end{smallmatrix} \right) \in \Mb_{r'\times r'}(\C)$,

\item[$\bullet$] $B=\left(\begin{smallmatrix} 0 & 0 \\ 0 & b \end{smallmatrix}\right), \; B'=\left(\begin{smallmatrix} 0 & 0 \\ 0 & b' \end{smallmatrix}\right) \in \Mb_{r''\times r'}(\C)$,

\item[$\bullet$] $C=\left(\begin{smallmatrix} 0 & 0 \\ 0 & c \end{smallmatrix}\right), \; C'=\left(\begin{smallmatrix} 0 & 0 \\ 0 & c' \end{smallmatrix}\right) \in \Mb_{r'\times r''}(\C)$,

\item[$\bullet$] $D=\left(\begin{smallmatrix}I_{r''-1} & 0 \\ 0 & d \end{smallmatrix} \right), D'=\left(\begin{smallmatrix}I_{r''-1} & 0 \\ 0 & d' \end{smallmatrix} \right) \in \Mb_{r''\times r''}(\C)$.

%\item[$\bullet$] $R=\left( \begin{smallmatrix} A & v &  \\ {}^tw & d & \\  &   & I_{r''} \end{smallmatrix}\right),  R^{-1} = \left( \begin{smallmatrix} A' & v' &  \\ {}^tw' & d' & \\ & & I_{r''} \end{smallmatrix} \right)$, where $v=\left(\begin{smallmatrix}0 \\ b \end{smallmatrix}\right), w=\left(\begin{smallmatrix}0 \\ c \end{smallmatrix}\right), v'=\left(\begin{smallmatrix}0 \\ b' \end{smallmatrix}\right), w'=\left(\begin{smallmatrix}0 \\ c' \end{smallmatrix}\right) \in \C^{r'+1}$.
\end{itemize}
An element $X$ of $\ssl(V')$ considered as subspace of $\ssl(V)$ is given in the basis $\Wc$ by a matrix of the form $\left(\begin{smallmatrix} X' & 0 \\ 0 & 0 \end{smallmatrix} \right)$, with $X'\in \Mb_{r'\times r'}(\C), \Tr(X')=0$. We have
\begin{align*}
\Ad(\gamma)(X)=R\cdot X \cdot R^{-1} & =\left(\begin{array}{cc} A & C  \\ B & D \end{array} \right) \cdot \left(\begin{array}{cc} X' & 0  \\ 0 & 0_{r''} \end{array} \right)\cdot \left(\begin{array}{ccc} A' & C' \\ B' & D' \end{array} \right) = \left(\begin{array}{cc} AX'A' & AX'C'   \\ BX'A' & BX'C' \end{array} \right).
\end{align*}
We denote by $E_{ij}^{(m)}$ the matrix in $\Mb_{m\times m}(\C)$ whose  entry at the $i$-th row and $j$-th column is  $1$, and all other entries are zero. Since $r'\geq 2$ by assumption, the matrix $E_{r',1}^{(r')}$ is traceless. Therefore $E^{(r)}_{r',1}=\left(\begin{smallmatrix} E_{r',1}^{(r')} & 0 \\ 0 & 0 \end{smallmatrix}\right) \in  \gg_\C$. Since we have $E^{(r')}_{r',1}\cdot C'=0$, it follows $\Ad(\gamma)(E_{r'+1,1}^{(r)})=\left(\begin{smallmatrix} A E_{r',1}^{(r')}A' & 0 \\ BE^{(r')}_{r',1}A' & 0 \end{smallmatrix}\right) \in  \gg_\C$.
Since $\Tr(AE^{(r')}_{r',1}A')=0$,
$\left(\begin{smallmatrix}AE_{r',1}^{(r')}A' & 0 \\ 0 & 0 \end{smallmatrix} \right)$ is in $\ssl(V')$.
It follows that $\left(\begin{smallmatrix} 0 & 0 \\ BE^{(r')}_{r',1}A'& 0\end{smallmatrix}\right)=b\cdot E^{r}_{r,1} \in \gg_\C$.
Since $b\neq 0$, we have $E_{r,1}^{(r)}  \in \gg_\C$.

Similarly, we have $E^{(r)}_{1,r'}= \left(\begin{smallmatrix} E^{(r')}_{1,r'} & 0 \\ 0 & 0  \end{smallmatrix} \right)\in \ssl(V')\subset \ssl(V)$. Since $B\cdot E^{(r')}_{1,r'}=0$, $\Tr(AE^{(r')}_{1,r'}A')=0$, and $c'\neq 0$, by considering $\Ad(\gamma)(E^{(r)}_{1,r'})$ we get  that $E^{(r)}_{1,r}  \in \gg_\C$.

We consider the adjoint actions of $E^{(r)}_{r,1}$ and $E^{(r)}_{1,r}$ on $\gg_\C$. Let $X=\left(\begin{smallmatrix} X' & 0 \\ 0 & 0 \end{smallmatrix}\right)$, where $X'\in \Mb_{r'\times r'}(\C)$ such that $\Tr(X')=0$.  We have
\begin{align*}
\adj(E_{r,1}^{(r)})(X)  = E_{r,1}^{(r)}X-XE_{r,1}^{(r)}  & = \left(\begin{array}{cc} 0 \dots 0  & 0 \cdots 0 \\ {}^tv_1X' & 0 \dots  0 \end{array}\right) \in \gg_\C\\
\adj(E_{1,r}^{(r)})(X) = E_{1,r}^{(r)}X-XE_{1,r}^{(r)} &  =\left(\begin{array}{cc} 0 \dots 0 & -X'v_1 \\ 0 \dots 0 & 0 \end{array} \right) \in \gg_\C.
\end{align*}
where $v_1={}^t\left(1, 0, \dots, 0\right)$.
Since $\ssl(r',\C)$ acts transitively on $\C^{r'}$, it follows that $\gg_\C$ contains all the matrices of the form $\left(\begin{smallmatrix} 0 \dots 0 & 0 \dots 0  \\ {}^tv & 0 \dots 0 \end{smallmatrix} \right)$ and $\left(\begin{smallmatrix} 0 \dots 0 & v   \\ 0 \dots 0 & 0  \end{smallmatrix}\right)$, with $v \in \C^{r'}$.

\medskip

Now let $X=\left(\begin{smallmatrix} 0 & 0 \\ 0 & X''\end{smallmatrix} \right)$, where $X''\in \Mb_{r''\times r''}(\C)$ such that $\Tr(X'')=0$, and $w_1:={}^t(0,\dots,0,1)\in \C^{r''}$. We have
\begin{align*}
\adj(E_{r,1}^{(r)})(X)  = E_{r,1}^{(r)}X-XE_{r,1}^{(r)}  & = \left(\begin{array}{cc} 0 & 0 \dots 0   \\ -X''w_1 & 0 \dots 0 \end{array}\right) \in \gg_\C\\
\adj(E_{1,r}^{(r)})(X) =E_{1,r}^{(r)}X-XE_{1,r}^{(r)}  & = \left(\begin{array}{cc}  0 \dots 0 & {}^tw_1X'' \\ 0 \dots 0  & 0 \end{array} \right) \in \gg_\C.
\end{align*}
Taking $X''=E^{(r'')}_{i,1}$ and $X''=E^{(r'')}_{1,i}$, for $i=2,\dots,r''$, we see that $\gg_\C$ contains the matrices $E^{(r)}_{i,1}$ and $E^{(r)}_{1,i}$, for $i=r'+1,\dots,r$.
By considering $\adj(E^{(r)}_{i,1})(X)$ and $\adj(E^{(r)}_{1,i})(X)$ with $X \in \left(\begin{smallmatrix} \ssl(r',\C) & 0 \\ 0 & 0 \end{smallmatrix} \right)$, and using the irreducibility of $\ssl(r',\C)$ over $\C^{r'}$, we obtain that $\gg_\C$ contains all the matrices of the form $\left(\begin{smallmatrix} 0_{r'} & Y \\ Z & 0_{r''} \end{smallmatrix} \right)$.  Finally, we have
$$
\adj(E^{(r)}_{1,r})(E^{(r)}_{r,1})=E^{(r)}_{1,1}- E^{(r)}_{r,r} \in \gg_\C.
$$
Therefore,
\begin{itemize}
\item[$\bullet$] $\left(\begin{smallmatrix} \ssl(r',\C)  & 0 \\ 0 & \ssl(r'',\C) \end{smallmatrix}\right) \subset \gg_\C$,

\item[$\bullet$] $\left\{\left(\begin{smallmatrix} 0_{r'} & Y \\ Z & 0_{r''} \end{smallmatrix}\right), \; Y \in \Mb_{r'\times r''}(\C), Z \in \Mb_{r''\times r'}(\C)\right\} \subset \gg_\C$,

\item[$\bullet$] $\mathrm{diag}(1,0,\dots,0,-1) \in \gg_\C$.
\end{itemize}
It follows that $\gg_\C=\ssl(r,\C)$, and the proposition is proved.
\end{proof}

\begin{Remark}\label{rk:induction:lem}
In cases (c) and (d), we do not have $E^{(r')}_{r',1}\cdot C' = 0$ anymore.  If $r'\geq 3$, it is enough to replace $E^{(r')}_{r',1}$ by  $E^{(r')}_{r',2}$. If $r'=2$ and $r''\geq 2$, instead of consider $\Ad(\gamma)(E^{(r)}_{r',1})$ and $\Ad(\gamma)(E^{(r)}_{1,r'})$, we consider $\Ad(\gamma)(E^{(r)}_{r,r'+1})$ and $\Ad(\gamma)(E^{(r)}_{r'+1,r})$ (note that $E^{(r)}_{r,r'+1}$ and $E^{(r)}_{r'+1,r}$ belong to  $\left(\begin{smallmatrix} 0 & 0 \\ 0 & \ssl(V'') \end{smallmatrix}\right)$), and the same arguments allow us to conclude.

Since the restriction of $\langle.,.\rangle$ to $V''$ is non-degenerate, in case (d), we must have $\dim V'' \geq 2$.  Therefore, the only case remaining is case (c) when $(r',r'')=(2,1)$. In this particular case we proceed as follows: as $E^{(3)}_{1,1}-E^{(3)}_{2,2}\in \ssl(V') \subset \gg_\C$, we have
$$
\Ad(\gamma)(E^{(3)}_{1,1}-E^{(3)}_{2,2})=\left(\begin{array}{ccc}
1 & a & c \\ 0 & 1 & 0 \\ 0 & b & d
\end{array}\right)\cdot\left( \begin{array}{ccc}
1 & 0 & 0 \\ 0 & -1 & 0 \\ 0 & 0 & 0
\end{array}
\right)\cdot\left( \begin{array}{ccc}
1 & a' & c' \\ 0 & 1 &  0 \\ 0 & b' & d'
\end{array}
\right)=\left(\begin{array}{ccc}
1 & a'-a & c' \\ 0 & -1 & 0 \\ 0 & -b & 0
\end{array}
\right) \in \gg_\C
$$
It follows that $Y:=\left(\begin{smallmatrix} 0 & 0 & c' \\ 0 & 0 &0 \\ 0 & -b & 0 \end{smallmatrix}\right) \in \gg_\C$.
Now $Z:=\adj(Y)(E^{(3)}_{2,1})=\left( \begin{smallmatrix}
0 & 0 & 0 \\ 0 & 0 & -c' \\ -b & 0 & 0
\end{smallmatrix}
\right) \in \gg_\C$. Since $\adj(Y)(Z)=bc'\cdot\left( \begin{smallmatrix} -1 & 0 & 0\\ 0 & -1 & 0 \\ 0 & 0 & 2\end{smallmatrix}\right)$,
we get that $\gg_\C$ contains the Cartan algebra $\hh=\{\mathrm{diag}(x_1,x_2,x_3) \in \Mb_3(\C), \;  x_1+x_2+x_3=0\}$ of $\ssl(V)$.
One obtains then the desired conclusion by using standard arguments on the adjoint representation of $\hh$ .
\end{Remark}

\section{Zariski density on each factor}\label{sec:Zar:dense:on:factor}
%Let $\hX, T, \kappa=(k_1,\dots,k_n)$ be as in \textsection~\ref{subsec:cyclic:covers:intro}.
Throughout this section, we will assume that $\gcd(k_1,\dots,k_n,d)=1$, and  $q=e^{-\frac{2\pi\imath k}{d}}$ is a primitive $d$-th root of unity, which means that $\gcd(k,d)=1$.
Let $\Gb$ be the Zariski closure of $\rho_q(\PB_n)$ in $\U(H^1(\hX)_q)$, and $\Gb^0$ the identity component of $\Gb$.
Our goal is to prove the following
\begin{Theorem}\label{th:Zar:density:on:factor}
Assume that the sequence of numbers $(\left\{\frac{kk_1}{d}\right\},\dots,\left\{\frac{kk_n}{d}\right\})$ satisfies the condition in Definition~\ref{def:good:weights}. Then $\Gb^0$ is equal to $\SU(H^1(\hat{X})_q)$.
\end{Theorem}

\begin{Remark}\label{rk:Zar:density:max}
It is a well known fact that $\PB_n$ is generated by the family of Dehn twists $\{\alpha_{i,j}, \; 1 \leq i < j \leq n\}$. By Corollary~\ref{cor:det:alpha:ij} $\det(\rho_q(\alpha_{i,j}))=q^{k_i+k_j} \in \Ub_d$. Thus $\rho_q(\PB_n) \subset \U_d(H^1(\hX)_q)$. Since $\SU(H^1(\hX)_q)$ is the identity component  of $\U_d(H^1(\hX)_q)$, Theorem~\ref{th:Zar:density:on:factor} means that if $(\left\{\frac{kk_1}{d}\right\},\dots,\left\{\frac{kk_n}{d}\right\})$ satisfies the condition in Definition~\ref{def:good:weights}, then $\Gb^0$ is maximal.
\end{Remark}

To prove Theorem~\ref{th:Zar:density:on:factor} we will find a sequence $\{0\}=V_0 \subsetneq V_1 \subsetneq \dots \subsetneq  V_\ell=H^1(\hX)_q$ of subspaces of $H^1(\hX)_q$ such that $\dim V_{i+1}-\dim V_i \in \{1,2\}$ and the restriction of $\langle.,.\rangle$ to each $V_i$ is non-degenerate.
We then show that $\Gb^0$ contains $\SU(V_i)$, for all $i=1,\dots,\ell$, by induction using Proposition~\ref{prop:density:induction}.

In preparation for the proof of Theorem~\ref{th:Zar:density:on:factor}, let us introduce the following notation. For all $r\in  \{1,\dots,n\}$
\begin{itemize}
\item[$\bullet$] $E_r \subset D$ is a topological disc such that $E_r\cap \Bc=\{b_1,\dots,b_r\}$. We also assume that $E_{r-1}\subset E_r$.

\item[$\bullet$] $Y_r$ is the preimage of $E_r$ in $\hX$.

\item[$\bullet$] $\Vbb_r$ is the image of $H^1_c(Y_r)_q$ in $H^1(\hX)_q$.

\item[$\bullet$] $\hk_r:=k_1+\dots+k_r$.

%\item[$\bullet$] $\bar{k}_r \in \{0,\dots,d-1\}$ is the rest of the division of $\hk_r$ by $d$.
\end{itemize}
For any pair $(r,s)$ such that $1 \leq s < r \leq n$,
\begin{itemize}
\item[$\bullet$] $A_{r,s}:=E_r\setminus E_s$,

\item[$\bullet$] $Z_{r,s}$ is preimage of $A_{r,s}$ in $\hX$,

\item[$\bullet$] $\Wbb_{r,s}$ is the image of $H^1_c(Z_{r,s})_q$ in $H^1_q(\hX)$.
\end{itemize}

We start by
\begin{Lemma}\label{lm:exist:gd:partition}
Let $\{x_1,\dots,x_n\}$, with $n\geq 3$, be a sequence of real numbers such that $0 < x_i < 1$ for all $i=1,\dots,n$, and
$$
1 < x_1+\dots+x_n < n-1.
$$
For all $j\in \{1,\dots,n\}$, define $s_j:=\sum_{i=1}^j x_i$.
Then there exists $r\in \{3,\dots,n\}$ which satisfies $s_{r-2}, s_r \notin \N$, and
\begin{equation}\label{eq:good:triple}
1 < s_r-\lfloor s_{r-2} \rfloor < 2.
\end{equation}
\end{Lemma}
\begin{proof}
If $s_2 \leq 1$, then we choose $r \geq 3$ to be the first number such that $s_r >1$. Such an $r$ exists because $s_n>1$. We  have $0 < s_{r-2} < s_{r-1} \leq 1$. Therefore, $\lfloor s_{r-2}\rfloor=0$.  Since $x_r< 1$, we must have $s_r <2$, and \eqref{eq:good:triple} follows.

\medskip

If $s_2 > 1$, let $r \geq 3$ be the first number such that $s_r < r-1$. Again, such an $r$ exists because $s_n < n-1$. By assumption, we have
\begin{equation}\label{eq:good:triple:cond1}
r-2 \leq s_{r-1} < s_r < r-1,
\end{equation}
and
\begin{equation}\label{eq:good:triple:cond2}
r-3 \leq s_{r-2} < r-2.
\end{equation}
Note that we must have $s_{r-2}> r-3$, because if $s_{r-2}=r-3$ then $s_{r-1}=s_{r-2}+x_{r-1}=r-3+x_{r-1} < r-3+1=r-2$.
From \eqref{eq:good:triple:cond2}, we get $\lfloor s_{r-2}\rfloor=r-3$.
The desired inequalities then follows from \eqref{eq:good:triple:cond1}.
\end{proof}

As a consequence of Lemma~\ref{lm:exist:gd:partition} we get the following
\begin{Lemma}\label{lm:exist:W:r:r-2:sign:1:1}
Assume that the sequence $(\left\{\frac{kk_1}{d}\right\},\dots,\left\{\frac{kk_n}{d}\right\})$ satisfies the condition of Definition~\ref{def:good:weights} (a).
Then there exists $r\in \{3,\dots,n\}$ such that $q^{\hk_r}\neq 1, q^{\hk_{r-2}}\neq 1$,  $\dim \Wbb_{r,r-2}=2$, and  the restriction of $\langle.,.\rangle$ to $\Wbb_{r,r-2}$ has signature $(1,1)$.
\end{Lemma}
\begin{proof}
Applying Lemma~\ref{lm:exist:gd:partition} to the sequence $(\left\{ \frac{kk_1}{d}\right\},\dots,\left\{\frac{kk_n}{d}\right\})$, we see that there is some index $r \geq 3$ we have $d \; \nmid \; \hk_{r-2}$, $d \; \nmid \; \hk_r$, and
$$
1 < \sum_{i=1}^r\left\{\frac{kk_i}{d}\right\}-\left\lfloor \sum_{i=1}^{r-2}\left\{\frac{kk_i}{d}\right\} \right\rfloor <2.
$$
Since
$$
\sum_{i=1}^{r-2}\left\{\frac{kk_i}{d}\right\}-\left\lfloor \sum_{i=1}^{r-2}\left\{\frac{kk_i}{d}\right\} \right\rfloor=\left\{ \sum_{i=1}^{r-2}\left\{\frac{kk_i}{d}\right\}\right\} = \left\{\frac{\sum_{i=1}^{r-2}kk_i}{d}\right\} = \left\{\frac{k\hat{k}_{r-2}}{d}\right\}
$$
we get
$$
1 < \left\{\frac{k\hk_{r-2}}{d}\right\}+\left\{\frac{kk_{r-1}}{d}\right\}+\left\{\frac{kk_{r}}{d}\right\} <2.
$$
Let $\hZ_{r,r-2}$ be the Riemann surface constructed from the curve $y^d=(x-\hb_{r-2})^{\hk_{r-2}}(x-b_{r-1})^{k_{r-1}}(x-b_r)^{k_r}$, where $\hb_{r-2}$ is an arbitrary point in the disc $E_{r-2}$.
Since $q^{\hk_{r-2}}\neq 1$ and $q^{\hk_r}\neq 1$, by Proposition~\ref{prop:coh:subsurf:embedding},  $(\Wbb_{r,r-2},\langle.,.\rangle)$ is isomorphic to $(H^1(\hZ_{r,r-2})_q,\langle.,.\rangle_{\hZ_{r,r-2}})$. It follows from Theorem~\ref{th:Menet:dim:signature} that the signature $(r_q, s_q)$ of $\langle.,.\rangle_{\hZ_{r,r-2}}$ is given by
$$
r_q=\left\lfloor\left\{\frac{k\hk_{r-2}}{d}\right\}+\left\{\frac{kk_{r-1}}{d}\right\}+\left\{\frac{kk_r}{d}\right\} \right\rfloor=1
\quad \text{ and } \quad  s_q= \left\lfloor 3-\left\{\frac{k\hk_{r-2}}{d}\right\}-\left\{\frac{kk_{r-1}}{d}\right\}-\left\{\frac{kk_r}{d}\right\} \right\rfloor=1
$$
(in this case $\eps_0=0$ since $q^{\hk_r}\neq 1$). The lemma is then proved.
\end{proof}

\begin{Proposition}\label{prop:Zar:density:dim:2}
Assume that the sequence $(\left\{\frac{kk_1}{d}\right\},\dots,\left\{\frac{kk_n}{d}\right\})$ satisfies the condition of Definition~\ref{def:good:weights} (a).
Let $r \in \{3,\dots,n\}$ be the index in Lemma~\ref{lm:exist:W:r:r-2:sign:1:1}.
Let $\alpha_r$ be the boundary of the disc $E_{r-1}$, $\beta_r$ the boundary of a disc in $E_r$ that contains only $\{b_{r-1}, b_r\}$ and disjoint from $E_{r-2}$, and $\gamma_r$ the boundary of a disc that contains $E_{r-2}$ and $b_r$ but not $b_{r-1}$. The configuration of $\alpha_r, \beta_r, \gamma_r$ is shown in Figure~\ref{fig:config:lantern:in:E:r}.

\begin{figure}[htb]
\begin{tikzpicture}[scale=0.5]
\foreach \x  in {(-7,0), (0,0), (4,0), (8,0)} \filldraw[gray] \x circle(3pt);
\draw[very thick] (0,0) ellipse (11 and 6);
\draw[very thick] (-3.5,0) ellipse (4.5 and 3);

\draw[very thin, red] (6,0) ellipse (3 and 2);
\draw[very thin, red] (2,0) arc (0:270: 5.5 and 4);
\draw[very thin, red] (-3.5,-4) --(7,-4) arc (-90:0:2)-- (9,0) arc (0:180:1) -- (7,-1) arc (0:-180: 2.5  and 2) -- (2,0);
\draw[very thin, red] (-2.5,0) ellipse (7.5 and 5);
\draw (-7,0) node[below] {$b_1$}
      (-3.5,0) node {$\dots$}
      (0,0) node[below] {$b_{r-2}$}
      (4,0) node[below]  {$b_{r-1}$}
      (8,0) node[below]  {$b_r$};
\draw (-3.5,-1.5) node {$E_{r-2}$};
\draw (2,4) node[above] {$\alpha_r$};
\draw (6,2) node[above] {$\beta_r$};
\draw (4,-4) node[below] {$\gamma_r$};
\end{tikzpicture}

\caption{Configurations of the curves $\alpha_r,\beta_r,\gamma_r$ in $A_{r,r-2}$}
\label{fig:config:lantern:in:E:r}
\end{figure}
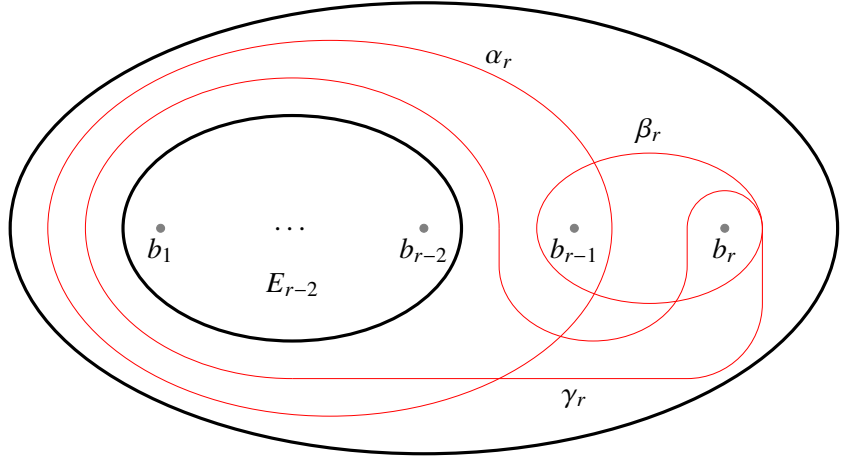
%Let then the actions of $\tilde{\tau}^*_{\alpha_r},\tilde{\tau}_{\beta_r}, \tilde{\tau}_{\gamma_r}$ on $H^1(\hX)_q$ preserve $\Wbb_{r,r-2}$.
Let $A,B,C$ be respectively the matrices of the restrictions of $\rho_q(\tau_{\alpha_r}),\rho_q(\tau_{\beta_r}), \rho_q(\tau_{\gamma_r})$ to $\Wbb_{r,r-2}$. Then the identity component of the Zariski closure of the group generated by $\{ A,B,C\}$ is equal to $\SU(\Wbb_{r,r-2})$.
\end{Proposition}
\begin{proof}
It follows from Lemma~\ref{lm:coh:preimage:disc} and Proposition~\ref{prop:coh:subsurf:embedding} that we have
\begin{itemize}
\item[$\bullet$] $\dim \Vbb_r=r-1$,

\item[$\bullet$] $\dim \Vbb_{r-2}=r-3$,

\item[$\bullet$] $\dim \Wbb_{r,r-2}=2$.
\end{itemize}
Moreover, the restrictions of $\langle.,.\rangle$ to both $\Vbb_r$ and $\Vbb_{r-2}$ are  non-degenerate since $q^{\hk_{r-2}}\neq 1$ and $q^{\hk_r}\neq 1$.

We can consider $\Vbb_{r-2}$ as a subspace of $\Vbb_r$.
We claim that $\Wbb_{r,r-2}=\Vbb_{r-2}^\perp \cap \Vbb_r$. Indeed,  by construction we have $\Wbb_{r,r-2}\subset \Vbb_{r-2}^\perp \cap \Vbb_r$. Since $\dim \Vbb_{r-2}^\perp \cap \Vbb_r=\dim  \Vbb_r- \dim \Vbb_{r-2}=2=\dim \Wbb_{r,r-2}$, the claim follows.

Recall that the element  $g_i$ in the generating set $\{g_1,\dots,g_n\}$ of $H^1(\hX)_q$ is represented by a closed $1$-form with compact support in the preimage of a disc that contains $b_i, b_{i+1}$ but none of the other points in $\Bcal$. In particular, we  have $g_{r-1} \in \Wbb_{r,r-2} \subset \Vbb_r$.

Consider now the element $g_{r-2}$. We can write
$$
g_{r-2}=g'_{r-2}+g''_{r-2}
$$
with $g'_{r-2}\in \Vbb_{r-2}$ and $g''_{r-2}\in \Vbb_{r-2}^\perp\cap\Vbb_r=\Wbb_{r,r-2}$. Since $\{g_1,\dots,g_{r-1}\}$ is a basis of $\Vbb_r$ by Lemma~\ref{lm:coh:preimage:disc}, and $\{g_1,\dots,g_{r-3}\}\subset \Vbb_{r-2}$, it follows that $(g''_{r-2},g_{r-1})$ is a basis of $\Wbb_{r,r-2}$.

The Dehn twists $\tau_{\alpha_r},\tau_{\beta_r},\tau_{\gamma_r}$ clearly preserve the annulus $A_{r,r-2}$. Therefore, their lifts $\tilde{\tau}_{\alpha_r}, \tilde{\tau}_{\beta_r}, \tilde{\tau}_{\gamma_r}$ preserve $Z_{r,r-2}$. As a consequence, $\tilde{\tau}^*_{\alpha_r}, \tilde{\tau}^*_{\beta_r}, \tilde{\tau}^*_{\gamma_r}$ preserve $\Wbb_{r,r-2}$ as well as the orthogonal decomposition $\Vbb_r=\Vbb_{r-2}\oplus\Wbb_{r,r-2}$.

We now compute the matrices of $\tilde{\tau}^*_{\alpha_r|\Wbb_{r,r-2}}$ and $\tilde{\tau}^*_{\beta_r|\Wbb_{r,r-2}}$ in the basis $(g''_{r-2},g_{r-1})$.
By Theorem~\ref{th:Menet:Dehn:twist}, the action of  $\tilde{\tau}^*_{\alpha_r}$ on $H^1(\hX)_q$ satisfies
\begin{equation}\label{eq:Dtwist:a:in:E:r:action:1}
\tilde{\tau}^*_{\alpha_r}(g_{r-2})=q^{\hk_{r-1}}g_{r-2}= q^{\hk_{r-1}}.g'_{r-2}+q^{\hk_{r-1}}g''_{r-2},
\end{equation}
and
\begin{equation}\label{eq:Dtwist:b:in:E:r:action:2}
\tilde{\tau}^*_{\alpha_r}(g_{r-1})=g_{r-1}+\sum_{l=1}^{r-2}(q^{k_{l+1}+\dots+k_{r-1}}-q^{\hk_{r-1}})g_l=g_{r-1}+(q^{k_{r-1}}-q^{\hk_{r-1}})g''_{r-2}+v, \; \text{ with } v\in \Vbb_{r-2}.
\end{equation}
Since $\tau_{\alpha_r}$ preserves the decomposition $\Vbb_r=\Vbb_{r-2}\oplus\Wbb_{r,r-2}$, from \eqref{eq:Dtwist:a:in:E:r:action:1}, we must have $\tilde{\tau}^*_{\alpha_r}(g''_{r-2})=q^{\hk_{r-1}}g''_{r-2}$, and from \eqref{eq:Dtwist:b:in:E:r:action:2} we get $\tilde{\tau}^*_{\alpha_r}(g_{r-1})=g_{r-1}+(q^{k_{r-1}}-q^{\hk_{r-1}})g''_{r-2}$ (that is $v=0$). As a consequence, the restriction of $\tilde{\tau}^*_{\alpha_r}$ to $\Wbb_{r,r-2}$ is given by the matrix
$$
A=\left(
\begin{array}{cc}
q^{\hk_{r-2}+k_{r-1}} & q^{k_{r-1}}-q^{\hk_{r-2}+k_{r-1}}\\
0 & 1
\end{array}
\right)
$$
in the basis $(g''_{r-2},g_{r-1})$.
The action of $\tilde{\tau}^*_{\beta_r}$ satisfies $\tilde{\tau}^*_{\beta_r}(g_{r-1})=q^{k_{r-1}+k_r}g_{r-1}$, and
$$
\tilde{\tau}^*_{\beta_r}(g_{r-2})=g_{r-2}+(1-q^{k_r})g_{r-1},
$$
which implies that $\tilde{\tau}^*_{\beta_r}(g''_{r-2})=g''_{r-2}+(1-q^{k_r})g_{r-1}$. Therefore, the matrix of $\tilde{\tau}^*_{\beta_r|\Wbb_{r,r-2}}$ in the basis $(g''_{r-2},g_{r-1})$ is
$$
B=\left( \begin{array}{cc}
1 & 0 \\
1-q^{k_r} & q^{k_{r-1}+k_r}
\end{array}
\right).
$$
The determination of the images of $g_{r-2}$ and $g_{r-1}$ by  $\tilde{\tau}^*_{\gamma_r}$ by direct calculations is quite involved since the basis in Theorem~\ref{th:Menet:Dehn:twist} adapted to $\tau_{\gamma_r}$ is different from $\{g_1,\dots,g_{r-1}\}$.
For this reason, we will make use of the lantern relation applied to the four-holed sphere $S=E_r\setminus \left(E_{r-2}\cup D_{r-1}\cup D_r \right)$, where $D_{r-1}$ and $D_r$ are two (disjoint) small discs about $b_{r-1}$ and $b_r$ respectively. Note that the Dehn twists about the borders of $D_{r-1}$ and $D_r$ are trivial in $\PB_n$. Let $C$ be the matrix of $\tilde{\tau}^*_{\gamma_r|\Wbb_{r,r-2}}$ in the basis $(g''_{r-2},g_{r-1})$, then from Proposition~\ref{prop:lantern:conseq}, we get $A\cdot B\cdot C =q^{\hk_r}\cdot\Id_2$. Therefore,
$$
C=q^{\hk_r}\cdot B^{-1}\cdot A^{-1}= \left(
\begin{array}{cc}
q^{k_r} & -q^{k_{r-1}+k_r}(1-q^{\hk_{r-2}})\\
-\bar{q}^{k_{r-1}}(1-q^{k_r}) & 1-q^{k_r}+q^{\hk_{r-2}+k_r}
\end{array}
\right).
$$
Observer that $g''_{r-2}+\bar{q}^{k_{r-1}}g_{k-1}$ is an eigenvector of $C$ with the associated eigenvalue $q^{\hk_{r-2}+k_r}$.

Since $\tilde{\tau}^*_{\alpha_r}, \tilde{\tau}^*_{\beta_r},\tilde{\tau}^*_{\gamma_r}$ are either of finite order, or unipotent, so are $A, B, C$. We have $A\cdot B\cdot C$ is identity on $\Pb(\Wbb_{r,r-2})$, and  $\C\cdot g''_{r-2}, \C\cdot g_{r-1}, \C\cdot(g''_{r-2}+\bar{q}^{k_{r-1}}g_{r-1})$ are respectively fixed points of $A, B, C$ on $\Pb(\Wbb_{r,r-2})$.
We thus get the desired conclusion by applying Lemma~\ref{lm:density:sign:1:1}.
\end{proof}

For all $s\in \{1,\dots,n\}$, let $\Gamma_{s}$ be the subgroup of $\PB_n$ generated by  the Dehn twists about simple closed curves in the disc $E_s$.  Remark that $\Gamma_s \simeq \PB_s$ is generated by the Dehn twists $\alpha_{i,j}$ with $1 \leq i < j \leq s$.
Let $\Gb_s$ be the Zariski closure of the group $\{\tilde{\tau}^*_{|\Vbb_s}, \; \tau \in \PB_s\}$ in $\U(\Vbb_s)$,  and $\Gb^0_s$ be the identity component of $\Gb_s$.

%\medskip

%It follows from Theorem~\ref{th:Menet:alphaij}, that for all $\alpha_{i,j}$, with $1\leq i <j \leq s$, we have $\tilde{\alpha}^*_{i,j|\Vbb_s} \in \U_d(\Vbb_s)$.
%(recall that $\U(\Vbb_s)$ is the group of automorphisms $\alpha$ of  $\Vbb_s$ preserving $\langle.,.\rangle$ such that $\det(\alpha)\in \Ub_d$).
%It follows that we have $\Gb_s^0\subset \SU(\Vbb_s)$. In particular, we have $\Gb^0=\Gb^0_n \subset \SU(\Vbb_n)=\SU(H^1(\hX)_q)$.

\medskip

For all $s'\in \{1,\dots,s-1\}$, denote by $\Gamma_{s,s'}$ the subgroup of $\PB_n$ that is generated by Dehn twists about simple closed curves contained in the annulus $A_{s,s'}$.
\begin{Lemma}\label{lm:Dehn:twists:in:annuli}
For all $\tau \in \Gamma_{s,s'}$, we have $\tilde{\tau}^*(\Wbb_{s,s'})=\Wbb_{s,s'}$ and  $\tilde{\tau}_{|\Vbb_{s'}}=\lambda\cdot\id_{\Vbb_{s'}}$ for some $\lambda\in \Ub_q$.
\end{Lemma}
\begin{proof}
It is enough to prove the lemma in the case $\tau$ is a Dehn twist about a simple closed curve $\gamma$ contained in the annulus $A_{s,s'}$. Up to a homotopy $\tau$ preserves $A_{s,s'}$. Therefore, $\tilde{\tau}$ preserves the preimage $Z_{s,s'}$ of $A_{s,s'}$ in $\hX$, and $\tilde{\tau}^*(\Wbb_{s,s'})=\Wbb_{s,s'}$.

The simple closed curve $\gamma$ is  the boundary of a disc $D_\gamma \subset E_s$. Either $D_\gamma$ is disjoint from $E_{s'}$ or $D_\gamma$ contains $E_{s'}$ in its interior. By Theorem~\ref{th:Menet:Dehn:twist}, if $D_\gamma$ is disjoint from $E_{s'}$ then $\tilde{\tau}^*_{\Vbb_{s'}}=\id_{\Vbb_{s'}}$. If $E_{s'}$ is contained in $D_\gamma$, then for all $g_j, \, j=1,\dots,s'-1$, we have $\tilde{\tau}^*(g_j)=q^{\hk(\gamma)}\cdot g_j$, where $\hk(\gamma)=\sum_{b_i\in D_\gamma}k_i$.
Since $\{g_1,\dots,g_{s'-1}\}$ is a basis of $\Vbb_{s'}$ (by Lemma~\ref{lm:coh:preimage:disc}), we have $\tilde{\tau}^*_{\Vbb_{s'}}=q^{\hk(\gamma)}\cdot\id_{\Vbb_{s'}}$. This completes the proof of the lemma.
\end{proof}

By Lemma~\ref{lm:Dehn:twists:in:annuli}, for $\tau \in \Gamma_{s,s'}$, $\tilde{\tau}^*_{|\Wbb_{s,s'}}\in \U(\Wbb_{s,s'})$.  Let $\Gb_{s,s'} \subset \U(\Wbb_{s,s'})$ denote the Zariski closure of the group $\{\tilde{\tau}^*_{|\Wbb_{s,s'}}, \; \tau \in \Gamma_{s,s'}\}$ and $\Gb^0_{s,s'}$ its identity component.

\subsection*{Proof of Theorem~\ref{th:Zar:density:on:factor}: case (a)}
\begin{proof}
Assume now that the sequence $(\left\{\frac{kk_1}{d}\right\},\dots,\left\{\frac{kk_n}{d}\right\})$ satisfies the condition of Definition~\ref{def:good:weights} (a), that is
$$
1 < \left\{\frac{kk_1}{d}\right\},\dots,\left\{\frac{kk_n}{d}\right\} < n-1.
$$
Let $r\in \{3,\dots,n\}$ be the index in Lemma~\ref{lm:exist:W:r:r-2:sign:1:1}.
\begin{Claim}\label{clm:Zar:dense:W:r:s}
We have $\Gb^0_r=\SU(\Vbb_r)$.
\end{Claim}
\begin{proof}
The claim will be proved by induction.
By Proposition~\ref{prop:Zar:density:dim:2}, we already have $\Gb_{r,r-2}^0=\SU(\Wbb_{r,r-2})$. Assume now that for some $s \in \{1,\dots,r-2\}$ we have  $q^{\hk_{s}}\neq 1$ and $\Gb^0_{r,s}=\SU(\Wbb_{r,s})$.
If $s=1$, then we are done  since $\Gamma_{r,1}=\Gamma_r$ and $\Wbb_{r,1}=\Vbb_r$.
So let us assume that $s>1$. We will show that there exists $s'< s$ such that $q^{\hk_{s'}}\neq 1$ and $\Gb^0_{r,s'}=\SU(\Wbb_{r,s'})$.
We have two cases
\begin{itemize}
\item[$\bullet$] $q^{\hk_{s-1}}\neq 1$. In this case, we take $s'=s-1$. It follows from Proposition~\ref{prop:coh:subsurf:embedding} that $\dim\Wbb_{s,s-1}=1$, and that the restriction of $\langle.,.\rangle$ to $\Wbb_{s,s-1}$ is non-degenerate, which means that $\langle v,v \rangle \neq 0$ if $v$ is a generator of $\Wbb_{s,s-1}$. It also follows from Proposition~\ref{prop:coh:subsurf:embedding} that $(\Wbb_{r,s-1},\langle.,.\rangle)$ is non-degenerate,  and we have the following orthogonal decomposition $\Wbb_{r,s-1}=\Wbb_{r,s}\oplus \Wbb_{s,s-1}$.

For all $\tau\in \Gamma_{r,s}$, the restriction of $\tilde{\tau}^*$ to $\Vbb_s$ is the multiplication of the identity by a constant in $\Ub_d$ (cf. Lemma~\ref{lm:Dehn:twists:in:annuli}).
Since $\Wbb_{s,s-1} \subset \Vbb_s$, there is a scalar $\lambda(\tilde{\tau}^*) \in \Ub_d$ such that $\tau^*(v)=\lambda(\tau^*)v$ for all $v \in \Wbb_{s,s-1}$.
This implies that  $\tilde{\tau}^*_{|\Wbb_{r,s-1}}$ is given by the matrix $\left(\begin{smallmatrix} \tilde{\tau}^*_{|\Wbb_{r,s}} & 0 \\ 0 & \lambda(\tilde{\tau}^*)\end{smallmatrix} \right) \in \U(\Wbb_{r,s})\times\Ub_d$.
Since the Zariski closure of $\{\tilde{\tau}^*_{|\Wbb_{r,s}}, \;  \tau \in \Gamma_{r,s}\}$ contains $\SU(\Wbb_{r,s})$ by assumption, it follows that the Zariski closure of $\{\tilde{\tau}^*_{|\Wbb_{r,s-1}}, \; \tau \in \Gamma_{r,s}\}$ contains $\SU(\Wbb_{r,s})\times\{1\}$.
Note that $\dim \Wbb_{s,s-1}=1$, $\SU(\Wbb_{s,s-1})=\{1\}$.
Therefore $\Gb^0_{r,s-1}$ contains the product $\SU(\Wbb_{r,s})\times\SU(\Wbb_{s,s-1})$.

The group $\Gamma_{r,s-1}$ contains the Dehn twist $\alpha_{s-1,s}$, that is the Dehn twist with support in a small disc containing $b_{s-1}$ and $b_s$.  One can readily check that  the hypotheses of Proposition~\ref{prop:density:induction} are satisfied, with $V=\Wbb_{r,s-1}, V'=\Wbb_{r,s}, V''=\Wbb_{s,s-1}$, $\gamma=\alpha_{s-1,s}$, and $v=g_{s-1}$.
Thus we can conclude that $\Gb^0_{r,s-1}=\SU(\Wbb_{r,s-1})$.\\

\item[$\bullet$] $q^{\hk_{s-1}}=1$. In this case, we must have $s-1>1$, since by assumption, $q^{k_i}\neq 1$ for all $i \in \{1,\dots,n\}$. We then have $q^{\hk_{s-2}}=q^{\hk_{s-1}-k_{s-1}}=\bar{q}^{k_{s-1}}\neq 1$. We will show that $\Gb^0_{r,s-2}=\SU(\Wbb_{r,s-2})$. Again, we have that the restrictions of $\langle.,.\rangle$ to $\Wbb_{r,s-2}, \Wbb_{r,s}, \Wbb_{s,s-2}$ are all non-degenerate, and $\Wbb_{r,s-2}=\Wbb_{r,s}\oplus  \Wbb_{s,s-2}$ is an orthogonal direct sum.
By Proposition~\ref{prop:coh:subsurf:embedding}, $\dim \Wbb_{s,s-2}=2$.
We now observe that $\left\{k\hk_{s-2}/d\right\}+\left\{kk_{s-1}/d\right\}=1$. This is because $0< \left\{k\hk_{s-2}/d\right\}+\left\{kk_{s-1}/d\right\} < 2$, and $\left\{k\hk_{s-2}/d\right\}+\left\{kk_{s-1}/d\right\}\in \N$ since $q^{\hk_{s-1}}=1$.
As a consequence we  have
    $$
    1 <  \left\{\frac{k\hk_{s-2}}{d}\right\} + \left\{\frac{kk_{s-1}}{d}\right\}+ \left\{\frac{kk_{s}}{d}\right\} < 2
    $$
    It follows from Proposition~\ref{prop:Zar:density:dim:2} that $\Gb^0_{s,s-2}=\SU(\Wbb_{s,s-2})$.
Since for all $\tau\in \Gamma_{s,s-2}$,    $\tilde{\tau}^*_{|\Wbb_{r,s}}=\id_{\Wbb_{r,s}}$, $\Gb^0_{r,s-2}$ contains the group
$\left(\begin{smallmatrix} \Id_{\Wbb_{r,s}} & 0 \\ 0 & \SU(\Wbb_{s,s-2}) \end{smallmatrix}\right)$.

By Lemma~\ref{lm:Dehn:twists:in:annuli}, for all $\tau\in \Gamma_{r,s}$, the action of $\tilde{\tau}^*_{|\Wbb_{r,s-2}}$ is given by a matrix of the form $\left(\begin{smallmatrix} \tilde{\tau}^*_{|\Wbb_{r,s}} &  0 \\ 0 & \lambda\cdot\id_2\end{smallmatrix}\right)$, with $\lambda\in \Ub_d$.
The assumption $\Gb^0_{r,s}=\SU(\Wbb_{r,s})$ implies that $\Gb^0_{r,s-2}$ contains the group $\left(\begin{smallmatrix} \SU(\Wbb_{r,s}) & 0 \\ 0 & \Id_2 \end{smallmatrix}\right)$. Therefore $\Gb^0_{r,s-2}$ contains the product $\SU(\Wbb_{r,s})\times\SU(\Wbb_{s,s-2})$.

We now remark that $V=\Wbb_{r,s-2}, V'=\Wbb_{r,s}, V''=\Wbb_{s,s-2},\gamma=\alpha_{s-1,s}, v=g_{s-1}$ satisfy the hypothesis of Proposition~\ref{prop:density:induction} which proves the claim.
\end{itemize}
\end{proof}

\begin{Claim}\label{clm:Zar:dense:V:s}
For all $s\in \{r,\dots,n-1\}$, assume that $\Gb^0_s=\SU(\Vbb_s)$ and $q^{\hk_s}\neq 1$. Then either $s=n-1$ and $\Gb^0_s=\SU(H^1(\hX)_q)$, or there exists $s'\in \{s+1,s+2\}, \; s' \leq n$, such that $q^{\hk_{s'}}\neq 1$ and $\Gb^0_{s'}=\SU(\Vbb_{s'})$.
\end{Claim}
\begin{proof}
The proof of this claim follows the same lines as the proof of Claim~\ref{clm:Zar:dense:W:r:s}.
We have two cases
\begin{itemize}
\item[$\bullet$] $q^{\hk_{s+1}}\neq 1$. In this case, we will take $s'=s+1$. We have $\dim \Wbb_{s+1,s}=1$, the restriction of $\langle.,.\rangle$ to $\Vbb_{s+1}$ is non-degenerate, and the direct sum $\Vbb_{s+1}=\Vbb_s\oplus\Wbb_{s+1,s}$ is an orthogonal decomposition. By assumption $\Gb^0_{s+1}$ contains the group $\{\left(\begin{smallmatrix} \SU(\Vbb_s) & 0 \\ 0 & 1 \end{smallmatrix}\right)\}$. Thus, by applying Proposition~\ref{prop:density:induction}  with $\gamma=\alpha_{s,s+1}$ and $v=g_s$, we get the desired conclusion.

\item[$\bullet$] $q^{\hk_{s+1}}=1$. If $s=n-1$, then we have $\Vbb_s=\Vbb_{n-1}=\Vbb_n=H^1(\hX)_q$, and we are done. Assume that $s+1 <n$. We then have $q^{\hk_{s+2}}=q^{\hk_{s+1}+k_{s+2}}=q^{k_{s+2}}\neq 1$. Thus $(\Vbb_{s+2},\langle.,.\rangle)$ and $(\Wbb_{s+2,s},\langle.,.\rangle)$ are non-degenerate. We also have $\dim \Wbb_{s+2,s}=2$ and the signature of the restriction of $\langle.,.\rangle$ to $\Wbb_{s+2,s}$ is $(1,1)$ because
    $$
    1 < \left\{\frac{k\hk_s}{d}\right\} + \left\{\frac{kk_{s+1}}{d}\right\} + \left\{\frac{kk_{s+1}}{d}\right\} < 2.
    $$
    By Proposition~\ref{prop:Zar:density:dim:2}, $\Gb^0_{s+2,s}=\SU(\Wbb_{s+2,s})$, and therefore $\Gb^0_{s+2}=\SU(\Vbb_{s+2})$ by Proposition~\ref{prop:density:induction} (applied to $V=\Vbb_{s+2}, V'=\Vbb_s, V''=\Wbb_{s+2,s}, \gamma=\alpha_{s,s+1}, v=g_s$).
\end{itemize}
\end{proof}

Claim~\ref{clm:Zar:dense:V:s} implies that either $\Gb^0_n=\SU(\Vbb_n)=\SU(H^1(\hX)_q)$, or $\Gb^0_{n-1}=\SU(H^1(\hX)_q)$. Since $\Gb^0_{n-1}\subset \Gb^0_n=\Gb^0$, and we already know that $\Gb^0 \subset \SU(H^1(\hX)_q)$, the theorem is proved in this case.
%
%Finally, to see that $\Gb^0\subset \SU(H^1(\hX)_q)$, we remark that $\rho_q(\PB_n)$ is contained in the subgroup $U_d(H^1(\hX)_q)$ of $U(H^1(\hX)_q)$ which consists of $A \in U(H^1(\hX)_q)$ such that $\det^d(A)=1$. This is because $\PB_n$ is generated by the Dehn twists, and $\rho_q$ maps  all Dehn twists into $U_d(H^1(\hX)_q)$. This completes the proof of Theorem~\ref{th:Zar:density:on:factor} in the case $(\left\{\frac{kk_1}{d}\right\},\dots,\left\{\frac{kk_n}{d}\right\})$ satisfies the condition of Definition~\ref{def:good:weights} (a).
\end{proof}

\subsection*{Proof of Theorem~\ref{th:Zar:density:on:factor}: case (b)}\label{subsec:prf:th:Zar:dens:case:definite}
\begin{proof}
Assume now that $(\left\{\frac{kk_1}{d}\right\},\dots,\left\{\frac{kk_n}{d}\right\})$ satisfies Definition~\ref{def:good:weights} (b). We can assume that $$
0 < \left\{\frac{kk_1}{d}\right\},\dots,\left\{\frac{kk_n}{d}\right\} \leq 1
$$
and up to a renumbering of $(k_1,\dots,k_n)$ we have
\begin{itemize}
\item[$\bullet$] the order of $q^{k_1+k_2}$ is greater than 5, and

\item[$\bullet$] the order of either $q^{k_1+k_3}$ or $q^{k_2+k_3}$ is greater than 2.
\end{itemize}
Since the order of $\rho_q(\alpha_{i,j})$ is equal to the order of $q^{k_i+k_j}$, it  follows from Lemma~\ref{lm:density:sign:2:0} that $\Gb^0_3=\SU(\Vbb_3)$. Since $q^{\hk_s}\neq 1$ for all $ 3\leq s <n$, by Proposition~\ref{prop:coh:subsurf:embedding}, we have
\begin{itemize}
\item[$\bullet$]  $\dim \Vbb_s=s-1, \dim\Vbb_{s,s-1}=1$,

\item[$\bullet$]  the restrictions of the intersection form to $\Vbb_s$ and $\Wbb_{s,s-1}$ are non-degenerate,

\item[$\bullet$] the direct sum $\Vbb_s=\Vbb_{s-1}\oplus\Wbb_{s,s-1}$ is orthogonal.
\end{itemize}
Thus, by applying Proposition~\ref{prop:density:induction} successively, we get that $\Gb^0_s=\SU(\Vbb_s)$ for all $s=3,\dots,n-1$.
If $\hk_n=1$, then $\Vbb_{n-1}=\Vbb_n$ and we are done. Otherwise, we  apply again Proposition~\ref{prop:density:induction} to conclude.
\end{proof}

\section{Proof of  Zariski density}\label{sec:Zar:density}
In this section we will prove Theorem~\ref{th:main:Zar:dense:bis} and Theorem~\ref{th:main:Zariski:dense}.
Recall that  $H^1(\hX)_d \subset H^1(\hX)$ is the kernel of the action of $\phi_d(T^*)$ on  $H^1(\hX)$, where $\phi_d$ is the $d$-th cyclotomic polynomial. By choosing a basis of $H^1(\hX)_d$ with elements in  $H^1(\hX,\Z)$, we can identify $H^1(\hX)_d$ with $\C^N$ in such a way that  the action of $T^*$ on  $H^1(\hX)_d$ is given by a matrix with rational coefficients. In this basis, the intersection form is given by a skew-symmetric integral matrix. Thus the group  of endomorphisms of $H^1(\hX)_d$ that preserve the intersection form and commute with $T^*$ is  an algebraic linear group defined over $\Q$. We denote by $\Sp(\hX,\R)^T_d$  the set of real points of this group.

\medskip

For each $q\in \Ub_d$,  let $\Vbb^{(q)}:=H^1(\hX)_q$ and denote by $\Hb_q$ the restriction of the intersection form of $H^1(\hX)$ to $\Vbb^{(q)}$.
Recall that the signature $(r_q,s_q)$ of $\Hb_q$ is computed in Theorem~\ref{th:Menet:dim:signature}.
Let $\U(\Vbb^{(q)})$ (resp. $\SU(\Vbb^{(q)})$ be the group of automorphisms of $\Vbb^{(q)}$ that preserve $\Hb_q$ (resp. and have determinant $1$). The Lie algebras of $\U(\Vbb^{(q)})$ and $\SU(\Vbb^{(q)})$ will be denoted by $\uu(\Vbb^{(q)})$ and $\su(\Vbb^{(q)})$ respectively.
It is shown in \cite[\textsection 7]{McM13}, that
$$
\Sp(\hX,\R)^T_d \simeq \prod_{q \in \Ub_{d,{\rm prim}}^+}\U(\Vbb^{(q)}) \simeq  \prod_{q \in \Ub_{d,{\rm prim}}^+}\U(r_q,s_q)
$$
where $\Ub^+_{d,{\rm prim}}$ is the set of $d$-th primitive roots of unity $q$ such that $\mathrm{Im}(q) >0$. We also have the corresponding decomposition $H^1(\hX)_d=\bigoplus_{q\in \Ub_{d,{\rm prim}}}H^1(\hX,\C)_q$.
For all $q\in \Ub_{d,\prim}^+$, the $\U(\Vbb^{(q)})$ factor of $\Sp(\hX,\R)_d^T$ acts on $H^1(\hX,\C)_q\oplus H^1(\hX,\C)_{\bar{q}}$ via the morphism $A \mapsto (A\oplus \ol{A})$.
%For all $\tau\in \PB_n$, by definition $\rho_d(\tau)$ is the restriction of  $\rho(\tau)$ to $H^1(\hX)_d$.

For all $q \in \Ub_{d,\prim}^+$, let $\pi_q$ be the projection from $\Sp(\hX,\R)^T_d$ to $\U(\Vbb^{(q)})$.
By construction $\rho_d$ is a group morphism from $\PB_n$ to $\Sp(\hX,\R)^T_d$, and we have $\rho_q=\pi_q\circ\rho_d$.

\medskip

Let $K_d$ denote the cyclotomic field $\Q(\zeta_d)$. Let $q$ be a fixed primitive $d$-th root of unity, and  $\vv$ a basis  of $\Vbb^{(q)}$ whose members are elements of $H^1(\hX, K_d)$.  For all primitive $d$-th root of unity $q'$ in $\Ub_{d,\prim}$, there is a unique $\sigma \in \mathrm{Gal}(K_d)$ such that $\sigma(q)=q'$, and we have $\vv':=\sigma(\vv)$ is a basis of $\Vbb^{(q')}$.
For all $\tau \in \PB_n$,  we can express $\rho_q(\tau)$ in the basis $\vv$ as a matrix with coefficients in $K_d$.
Then the matrix of $\rho_{q'}(\tau)$ in the basis $\vv'$ satisfies $\rho_{q'}(\tau)=\sigma(\rho_q(\tau))$.

\medskip
%Thus $\rho_d(\PB_n) \subset \Sp(\hX,\Q)_d^T$.

Recall that $\hat{\Gb}$ is the Zariski closure of $\rho_d(\PB_n)$ in $\Sp(\hX,\R)_d^T$, and  $\hat{\Gb}^0$ is its identity component.
Let $\U_d(\Vbb^{(q)})$ denote the  subgroup of $\U(\Vbb^{(q)})$ consisting of elements with determinant in $\Ub_d$.
Since $\rho_q(\PB_n)\subset \U_d(\Vbb^{(q)})$, that is  it follows that $\hat{\Gb} \subset \prod_{q\in \Ub_{d,\prim}^+}\U_d(\Vbb^{(q)})$, and therefore
$$
\hat{\Gb}^0\subset \prod_{q\in \Ub_{d,\prim}^+}\SU(\Vbb^{(q)}).
$$
We first show
\begin{Lemma}\label{lm:proj:Zar:dense}
Let $q=e^{-\frac{2\pi\imath k}{d}}$ be an element of $\Ub_{d,\prim}^+$. If the Zariski closure of $\rho_q(\PB_n)$ contains $\SU(\Vbb^{(q)})$, then $\pi_q(\hat{\Gb}^0)=\SU(\Vbb^ {(q)})$.
\end{Lemma}
\begin{proof}
Since $\hat{\Gb}^0$ is an algebraic group, its projection $\hat{\Gb}^0_q:= \pi_q(\hat{\Gb}^0)$ is an algebraic subgroup of $\SU(\Vbb^{(q)})$. Clearly $\hat{\Gb}^0_{q}$ contains the group $\rho_q(\PB_n)\cap  \SU(\Vbb^{(q)})$. But by assumption  $\rho_q(\PB_n)\cap  \SU(\Vbb^{(q)})$ is Zariski dense in $\SU(\Vbb^{(q)})$. Thus we must have $\hat{\Gb}^0_{q}=\SU(\Vbb^{(q)})$.
\end{proof}

The following proposition is the key ingredient of the proofs of Theorem~\ref{th:main:Zar:dense:bis} and Theorem~\ref{th:main:Zariski:dense}.
\begin{Proposition}\label{prop:Lie:alg:morph:not:exist}
Let $q=e^{-\frac{2\pi\imath k}{d}}$ and $q'=e^{-\frac{-2\pi\imath k'}{d}}$ be two elements of $\Ub_{d,\prim}^+$.
Assume that the Zariski closures of $\rho_q(\PB_n)$ and $\rho_{q'}(\PB_n)$ contain $\SU(\Vbb^{(q)})$ and $\SU(\Vbb^{(q')})$ respectively. Suppose that there is a non-trivial morphism of Lie algebras  $\phi: \su(\Vbb^{(q)}) \to \su(\Vbb^{(q')})$ such that $\Ad(\rho_{q'}(\tau))\circ \phi =\phi\circ \Ad(\rho_q(\tau))$ for all $\tau\in \PB_n$.
If $\dim \Vbb^{(q)}=\dim \Vbb^{(q')} \geq 3$, or there exists a pair of indices $(i,j)$ such that $\gcd(k_i+k_j,d)=1$, then  $q=q'$.
\end{Proposition}
\begin{proof}
By assumption $\Im(\phi)$ is a non-trivial Lie sub-algebra of $\su(\Vbb^{(q')})$ that is invariant under the adjoint action of $\rho_{q'}(\PB_n)$. Since the Zariski closure of  $\rho_{q'}(\PB_n)$ contains $\SU(\Vbb^{(q')})$ and $\su(\Vbb^{(q')})\simeq \su(r_{q'},s_{q'})$ is a (real) simple Lie algebra we must have $\Im(\phi)=\su(\Vbb^{(q')})$.
We can view $\phi$ as a Lie algebra representation of $\su(\Vbb^{(q)})$. Since $\Im(\phi)=\su(\Vbb^{(q')})$, this representation is irreducible.

We have $\su(\Vbb^{(q)})\otimes_\R\C \simeq \su(r_q,s_q)\otimes_\R\C \simeq \ssl(r_q+s_q,\C)$, and  $\dim \Vbb^{(q')}=\dim \Vbb^{(q)}=r_q+s_q$.
It is a well known fact that any irreducible (complex) representation of $\ssl(r_q+s_q,\C)$ is isomorphic to either the identity, or the dual representation (see for instance \cite[\textsection 15, p.224]{FH:rep:thry}). This means that there exists $S\in \Mb_{r_q+s_q}(\C)$ such that
either  (a) $\phi(X)=S^{-1}\cdot X\cdot S$, or (b) $\phi(X)=-S^{-1}\cdot{}^tX\cdot S$, for all $X\in \su(\Vbb^{(q)})$.

Given $\tau\in \PB_n$, let us write $A:=\rho_q(\tau)$ and $A'=\rho_{q'}(\tau)$.
We have two cases
\begin{itemize}
\item[$\bullet$] Case (a): by assumption,  for all  $X \in \su(\Vbb^{(q)})$,  we have $A'{}^{-1}\cdot (S^{-1}\cdot X \cdot S) \cdot A'=S^{-1}\cdot (A^{-1}\cdot X \cdot A)\cdot S$, which is equivalent to
$$
(A\cdot S \cdot A'{}^{-1} \cdot S^{-1}) \cdot X \cdot (S \cdot A' \cdot S^{-1}\cdot A^{-1}) =X.
$$
This means that $S \cdot A' \cdot S^{-1}\cdot A^{-1}=\lambda\cdot \id$ for some $\lambda\in \C^*$.

\medskip

Let $\tau=\alpha_{i,j}$, with $1 \leq i < j \leq n$. By Theorem~\ref{th:Menet:alphaij}, the spectrum of $\rho_q(\alpha_{i,j})$ is $(q^{k_i+k_j},1,\dots,1)$, while the spectrum of $\rho_{q'}(\alpha_{i,j})$ is $(q'{}^{k_i+k_j},1,\dots,1)$.
Since $S\cdot \rho_{q'}(\alpha_{i,j}) \cdot S^{-1}=\lambda\cdot\rho_q(\alpha_{i,j})$,  we must have
\begin{equation}\label{eq:compare:spectrums}
\{q'{}^{k_i+k_j},\underset{r_q+s_q-1}{\underbrace{1,\dots,1}}\}=\{\lambda\cdot q^{k_i+k_j},\underset{r_q+s_q-1}{\underbrace{\lambda,\dots,\lambda}}\}.
\end{equation}
%We claim that for all pair of indices $(i,j)$, with $1 \leq i < j \leq n$, we must have $S\cdot\rho_{q'}(\alpha_{i,j})\cdot S^{-1}=\rho_q(\alpha_{i,j})$.
Assume first that $\dim \Vbb^{(q)}=r_q+s_q \geq 3$.
Since $r_q+s_q-1 \geq 2$, \eqref{eq:compare:spectrums} only occurs if we have $\lambda=1$ and $q^{k_i+k_j}=q'{}^{k_i+k_j}$.
In particular, we have $S\cdot\rho_{q'}(\alpha_{i,j})\cdot S^{-1}= \rho_q(\alpha_{i,j})$ for all $1 \leq i < j \leq n$.
Since $\PB_n$ is generated by the $\alpha_{i,j}$'s, it follows that we have $S\cdot\rho_{q'}(\tau)\cdot S^{-1}=\rho_{q}(\tau)$ for all $\tau$ in $\PB_n$. As a consequence, we get
\begin{equation}\label{eq:trace:equal:conj:q}
\Tr(\rho_{q'}(\tau))=\Tr(\rho_{q}(\tau)), \quad \hbox{for all $\tau$ in $\PB_n$.}
\end{equation}
Let $I$ be a subset of $\{1,\dots,n\}$ such that $|I|\geq 2$.
If $\tau$ is a Dehn twist about the border of a disc $E$ such that $E\cap\Bc=\{b_i,\; i\in I\}$, then from Theorem~\ref{th:Menet:Dehn:twist} we have
$$
\Tr(\rho_q(\tau))=(|I|-1)\cdot q^{\sum_{i\in I}k_i}+(r_q+s_q+1-|I|)
$$
and
$$
\Tr(\rho_{q'}(\tau))=(|I|-1)\cdot q'{}^{\sum_{i\in I}k_i}+(r_q+s_q+1-|I|).
$$
It follows from \eqref{eq:trace:equal:conj:q} that $q^{\sum_{i\in I}k_i}=q'{}^{\sum_{i\in I}k_i}$ for all  $I \subset \{1,\dots,n\}$ such that $|I|\geq 2$, and hence  $q^{k_i}=q'{}^{k_i}$ for all $i\in \{1,\dots,n\}$. Since $\gcd(k_1,\dots,k_n,d)=1$, we conclude that $q=q'$.

Assume now that $r_q+s_q=2$, and there exists $1 \leq i < j\leq n$ such that $\gcd(k_i+k_j,d)=1$. By taking $\tau=\alpha_{i,j}$, from \eqref{eq:compare:spectrums} we get that either $\lambda=1$ and $q'{}^{k_i+k_i}=q^{k_i+k_j}$, or $\lambda=q'{}^{k_i+k_j}$ and $q'{}^{k_i+k_i}=\bar{q}^{k_i+k_j}$. Since $\gcd(k_i+k_j,d)=1$, it follows that either $q'=q$ or $q'=\bar{q}$. By assumption, we have $\Im(q)>0$ and $\Im(q')>0$. Therefore, we must have $q'=q$.\\

\item[$\bullet$] Case (b): by a similar argument, we get that $S\cdot A'\cdot S^{-1}=\lambda \cdot {}^tA^{-1}$ for some $\lambda\in \C^*$. If $\dim \Vbb^{(q)} \geq 3$, then we must have $\lambda=1$ for all $\tau \in \PB_n$. As a consequence, we get that $q'{}^{k_i}=\bar{q}^{k_i}$ for all $i\in \{1,\dots,n\}$, and hence  $q'=\bar{q}$. But since $\Im(q)>0$ and $\Im(q')>0$, this is impossible. In the case, $\dim \Vbb^{(q)}=2$, and there exist $1 \leq i < j < n$ such that $\gcd(k_i+k_j,d)=1$, we conclude by comparing the spectra of $\rho_{q}(\alpha^{-1}_{i,j})$ and $\rho_{q'}(\alpha_{i,j})$.
\end{itemize}
\end{proof}

\subsection{Proof of Theorem~\ref{th:main:Zar:dense:bis}}\label{subsec:prf:th:main:Zar:dense:bis}
\begin{proof}
Let $\gf$ denote the Lie algebra of $\hat{\Gb}^0$. We have $\gf \subset \oplus_{q\in \Ub_{d,\prim}^+}\su(\Vbb^{(q)})$. Our goal is to show that
\begin{equation}\label{eq:Lie:alg:equal}
\gf=\bigoplus_{q\in \Ub_{d,\prim}^+}\su(\Vbb^{(q)}).
\end{equation}
Let us choose a numbering of the elements of $\Ub^+_{d,\prim}$ so that one can write $\Ub^+_{d,\prim}=\{q_1,\dots,q_m\}$ with $m=\varphi(d)/2$, where $\varphi(d)$ is the degree of the $d$-th cyclotomic polynomial $\phi_d$.

For all $i=1,\dots,m$, let $\pi_i$ be of the  projection $\oplus_{i=1}^m\su(\Vbb^{(q_i)})$ onto $\su(\Vbb^{(q_i})$.
We define $\fh_0:=\gf$, and for $k=1,\dots,m-1$, $\fh_k:=\ker(\pi_k)\cap\fh_{k-1} \subset \oplus_{i=k+1}^m\su(\Vbb^{(q_i)})$, or equivalently
$$
\fh_k=\gf\cap\bigcap_{i=1}^k\ker(\pi_ i).
$$
Observe that all of the $\fh_k$'s are invariant under the adjoint action of $\hat{\Gb}$.

\begin{Claim}\label{clm:proj:kernel}
For all $0 \leq k < i \leq m$, we have $\pi_i(\fh_k)=\su(\Vbb^{(q_i)})$.
\end{Claim}
\begin{proof}
Indeed, in the case $k=0$, from Lemma~\ref{lm:proj:Zar:dense}, we get $\pi_i(\gf)=\su(\Vbb^{(q_i)})$ for all $i\in \{1,\dots,m\}$. Assume that the claim holds for some $k\in\{0,\dots,m-1\}$ and all $i\in \{k+1,\dots,m\}$.
We consider $\pi_i(\fh_{k+1}) \subset \su(\Vbb^{(q_i)})$, with $i\geq k+2$.
By definition $\pi_i(\fh_{k+1})$ is a Lie subalgebra of $\su(\Vbb^{(q_i)})$ which is invariant under the adjoint action of $\hat{\Gb}_{i}:=\pi_{q_i}(\hat{\Gb})$. Since the Lie algebra $\su(\Vbb^{(q_i)})\simeq \su(r_{q_i}, s_{q_i})$ is simple, and $\hat{\Gb}_i$ contains $\SU(\Vbb^{(q_i)})$ by Theorem~\ref{th:Zar:density:on:factor}, we have that either $\pi_i(\fh_{k+1})=\{0\}$, or $\pi_i(\fh_{k+1})=\su(\Vbb^{(q_i)})$.

Assume that $\pi_i(\fh_{k+1})=\{0\}$. Consider the projection $\theta_{(k+1,i)}: \fh_{k} \to \su(\Vbb^{(q_{k+1})})\oplus \su(\Vbb^{(q_i)}), \xi \mapsto (\pi_{k+1}(\xi),\pi_i(\xi))$.
Denote by $\theta_{k+1}$ and $\theta_i$ the projections of $\su(\Vbb^{(q_{k+1})})\oplus \su(\Vbb^{(q_i)})$ onto $\su(\Vbb^{(q_{k+1})})$ and $\su(\Vbb^{(q_i)})$) respectively.

Let $\Vb:=\Im(\theta_{k+1,i})\subset \su(\Vbb^{q_{k+1}})\oplus\su(\Vbb^{(q_i)})$. Since $\pi_{k+1}(\fh_k)=\su(\Vbb^{(q_{k+1})})$ by assumption, and by definition $\pi_{k+1|\fh_{k}}=\theta_{k+1}\circ\theta_{(k+1,i)}$, we get that $\theta_{k+1}(\Vb)=\su(\Vbb^{(q_{k+1}})$.  The hypothesis  $\pi_i(\fh_{k+1})=\{0\}$ implies that $\ker(\theta_{k+1})\cap\Vb\subset \ker(\theta_i)$, which means that the restriction of $\theta_{k+1}$ to $\Vb$ is an isomorphism. Thus, $\Vb$ is the graph of a Lie algebra morphism $\phi: \su(\Vbb^{(q_{k+1})}) \to \su(\Vbb^{(q_i)})$. That $\fh_k$ is invariant under the adjoint action of $\hat{\Gb}$ implies  that $\phi$ satisfies
$$
\Ad(\rho_{q_i}(\tau))\circ\phi=\phi\circ \Ad(\rho_{q_{k+1}}(\tau))
$$
for all $\tau\in \PB_n$. It then follows from Proposition~\ref{prop:Lie:alg:morph:not:exist} that  $\phi(\su(\Vbb^{(q_{k+1})}))=\{0\}$. But would mean that $\theta_{i}(\Vb)=\pi_i(\fh_k)=\{0\}$, which is a contradiction to the induction hypothesis. Thus, we  conclude that $\pi_i(\fh_{k})=\su(\Vbb^{(q_i)})$ for all $0\leq k < i \leq m$.
\end{proof}

It follows from Claim~\ref{clm:proj:kernel} that we have $\fh_{m-1}=\su(\Vbb^{q_{m}})$, and by induction $\fh_k=\oplus_{i=k+1}^m\su(\Vbb^{(q_i)})$.  In particular we have
$$
\gf=\fh_0=\bigoplus_{i=1}^m\su(\Vbb^{(q_i)}).
$$
Theorem~\ref{th:main:Zar:dense:bis} is then proved.
\end{proof}

\subsection{Proof of Theorem~\ref{th:main:Zariski:dense}}\label{subsec:prf:th:main:Zar:dense}
\begin{proof}
Since for all $k \in \{1,\dots,d-1\}$ such that $\gcd(k,d)=1$, from the hypothesis that the sequence $\mu^{(k)}$ is good in the sense of Definition~\ref{def:good:weights}, we get that the Zariski closure of $\rho_q(\PB_n)$ contains $\SU(\Vbb^{(q)})$ for all $q\in \Ub_{d, \, {\rm prim}}$ by Theorem~\ref{th:Zar:density:on:factor}.
Thus Theorem~\ref{th:main:Zariski:dense} is a consequence of Theorem~\ref{th:main:Zar:dense:bis}.
\end{proof}

\begin{Remark}\label{rk:Zar:dense:false:in:dim:2}
In the case $\dim H^1(\hX)_q=2$, equality \eqref{eq:id:comp:Z:closure} may not hold  without the assumption that there exist $1 \leq i < j \leq n$ such that $\gcd(k_i+k_j,d)=1$.
\end{Remark}

\section{Horospherical subgroups}\label{sec:horospherical}
\subsection{Margulis arithmeticity criterion}\label{subsec:Margulis:crit}
To prove Theorem~\ref{th:main:arithm}, we will use the following criterion

\begin{Theorem}[Arithmeticity criterion]\label{th:Margulis:crit}
Let $G$ be a semisimple Lie group of real rank at least $2$ and $U$ be a non-trivial horospherical subgroup of $G$. Let $\Gamma$ be a discrete Zariski dense subgroup of $G$ that contains an irreducible lattice of $U$. Then $\Gamma$ is a non-cocompact irreducible arithmetic lattice of $G$.
\end{Theorem}

This criterion was conjectured by G.~Margulis in the mid 90's. While it has been proved in many cases by work of several authors~\cite{Oh:JAlg,Oh:IJM, Benoist:Oh:IMRN, Benoist:Oh:TG, Venky:Pacific}, the general version of this conjecture has only been proved in a recent work~\cite{BM:Duke} by Benoist and Miquel.
For for more details on the history and development  of the resolution this conjecture we refer to \cite{Benoist:survey, BM:Duke}.
We will eventually apply this criterion to $G=\hat{\Gb}^0$ and $\Gamma=\hat{\Gb}^0\cap \rho_d(\PB_n)$, where $\hat{\Gb}^0$ is the identity component of the Zariski closure of $\rho_d(\PB_n)$ in $\Sp(\hX,\R)^T_d$ under the assumption of Theorem~\ref{th:main:arithm}.

\subsection{Horospherical subgroup on each factor}\label{subsec:horosph:on:factor}
To lighten the discussion, we will suppose from now on that $d\; | \; (k_1+\dots+k_n)$ and that $n\geq 5$.
We will see in Lemma~\ref{lm:rho:P:n-1:finite:ind:rho:P:n} that the proof of Theorem~\ref{th:main:arithm} is not affected by those  additional assumptions.

Pick a primitive $d$-th root of unity $q$.
We will single out a horospherical subgroup $U_q$ of $\SU(\Vbb^{(q)})$ (recall that $\Vbb^{(q)}:=H^1(\hX)_q$) and study the structure of this group. In particular, we are interested in the conjugate action of the associated parabolic subgroup on $U_q$. The horospherical subgroup that will be used in the proof of Theorem~\ref{th:main:arithm} is simply the product of the $U_q$'s, $q\in\Ub^+_{d,{\rm prim}}$.
Most of the material in this section was known to Venkataramana~\cite{Venky:Annals, Venky:Invent}.

\medskip

Assume that there exists $m \in \{2,\dots,n-2\}$  such that $d \, | \, (k_1+\dots+k_m)$.
Let $E_{m-1} \subset E_m \subset E_{m+1}$ be three topological discs in $D$ such that $E_j\cap\Bc=\{b_1,\dots,b_j\}$.
Let $F_{m+1}=\CP^1\setminus E_{m+1}$ and $A_{m+1,m-1}=E_{m+1}\setminus E_{m-1}$.
The preimages of $E_{m-1}, F_{m+1}, A_{m+1,m-1}$ in $\hX$ are denoted by $Y_{m-1}, Z_{m+1},  Z_{m+1,m-1}$ respectively.
Let  $\Vbb^{(q)}_{m-1}, \Wbb^{(q)}_{m+1}, \Wbb^{(q)}_{m+1,m-1}$ be the images of $H^1_c(Y_{m-1})_q, H^1_c(Z_{m+1})_q, H^1_c(Z_{m+1,m-1})_q$ in $\Vbb^{(q)}$.
Recall that by Theorem~\ref{th:Menet:dim:signature}, we have $\dim \Vbb^{(q)}=n-2$.
\begin{Lemma}\label{lm:0:mod:d:direct:sum}
We have $\dim \Vbb^{(q)}_{m-1}=m-2$, $\dim \Wbb^{(q)}_{m+1}=n-2-m$, $\dim \Wbb^{(q)}_{m+1,m-1}=2$.
The restriction of the intersection form to each of these subspaces of $\Vbb^{(q)}$ is non-degenerate and we have the following orthogonal decomposition
\begin{equation}\label{eq:coh:dec:W:m-1:m+1}
\Vbb^{(q)}=\Wbb^{(q)}_{m+1,m-1}\oplus\Vbb^{(q)}_{m-1}\oplus\Wbb_{m+1}^{(q)}.
\end{equation}
\end{Lemma}
\begin{proof}
Observe that $q^{k_1+\dots+k_m}=1$, while $q^{k_1+\dots+k_{m-1}}=q^{-k_m}\neq 1$ and $q^{k_1+\dots+k_{m+1}}=q^{k_{m+1}}\neq 1$.
Hence the lemma follows from Lemma~\ref{lm:coh:preimage:disc}  and Proposition~\ref{prop:coh:subsurf:embedding}.
\end{proof}

Let $\{g_1,\dots,g_{n-1}\}$ be the generating family of $\Vbb^{(q)}$ provided by Theorem~\ref{th:Menet:generator:set}.
Since $q^{k_1+\dots+k_m}=1$, $H^1_c(\partial Y_m)$ is generated by
$$
w:=\sum_{i=1}^{m-1}(\bar{q}^{k_1+\dots+k_i}-1)g_i.
$$
In fact $w$ is the Poincar\'e dual of a combination of the components of  $\partial Y_m$.
In particular, we have $w \in \Wbb_{m+1,m-1}$ since the support of $w$ is contained in $Z_{m+1,m-1}$.
Note that we have $\langle w,w \rangle =0$, and
\begin{equation}\label{eq:value:prod:w:gm}
\langle g_{m},w\rangle = (q^{k_1+\dots+k_{m-1}}-1)\langle g_{m},g_{m-1}\rangle = (\bar{q}^{k_m}-1)\cdot(-\imath)\cdot\frac{(1-q)(1-\bar{q})}{1-\bar{q}^{k_m}}=\imath(1-q)(1-\bar{q}) \in \imath\R.
\end{equation}

\begin{Lemma}\label{lm:bases:decomp}
We have
\begin{itemize}
\item[(i)] $(g_1,\dots,g_{m-2})$ is a basis of $\Vbb^{(q)}_{m-1}$, \\

\item[(ii)] $(w,g_{m})$ is a basis of $\Wbb^{(q)}_{m+1,m-1}$, and \\

\item[(iii)] $(g_{m+2},\dots,g_{n-1})$ is a basis of $\Wbb^{(q)}_{m+1}$.
\end{itemize}
As a consequence $\Gcal^{(q)}:=(w,g_1,\dots,g_{m-2},g_{m+2},\dots,g_{n-1},g_{m})$ is a basis of $\Vbb^{(q)}$.
\end{Lemma}
\begin{proof}
Since $q^{k_1+\dots+k_{m-1}}=\bar{q}^{k_m}\neq 1$,  (i) follows from Lemma~\ref{lm:coh:preimage:disc}.
Since $\dim \Wbb^{(q)}_{m+1,m-1}=2$, and $w$ and $g_r$ are not collinear (if they were collinear then we must have $\langle g_m,w\rangle =0$), (ii) follows.
To prove (iii), we embed $Z_{m+1}$ into the Riemann surface $\hZ_{m+1}$ constructed from the curve defined by $y^d=\prod_{i=m+1}^n(x-b_i)^{k_i}$.
We then have $\Wbb^{(q)}_{m+1}\simeq H^1_c(Z_{m+1})_q \simeq H^1(\hZ_{m+1})_q$.
Since $q^{k_{m+1}+\dots+k_n}=1$, we have $\dim H^1(\hZ_{m+1})_q=n-m-2$, and $(g_{m+2},\dots,g_{n-1})$ is a basis of $H^1(\hZ_{m+1})_q$, from which (iii) follows.
\end{proof}

Let $\Lbb^{(q)}$ denote the line  generated by $w$ in $\Vbb^{(q)}$, and $\Wbb^{(q)}=\Vbb^{(q)}_{m-1}\oplus \Wbb^{(q)}_{m+1}$. Let $\hat{P}_q$ be the subgroup of $\U(\Vbb^{(q)})$ consisting of elements that preserve the partial flag
\begin{equation}\label{eq:flag:filtration}
\{0\}\subset \Lbb^{(q)} \subset \Lbb^{(q)}\oplus\Wbb^{(q)} \subset \Vbb^{(q)}.
\end{equation}
This is a parabolic subgroup of $\U(\Vbb^{(q)})$. The unipotent radical $U_q$ of $\hat{P}_q$ consists of elements that act trivially on the quotients of the filtration~\eqref{eq:flag:filtration}.
Note that $U_q$ is contained in $\SU(\Vbb^{(q)})$, and by definition, $U_q$ is a horospherical subgroup of $\SU(\Vbb^{(q)})$.

In the basis $\Gcal^{(q)}$ elements of $\hat{P}_q$ correspond to matrices of the form $\left(\begin{smallmatrix} * & * & * \\ 0 & X & * \\ 0 & 0 & * \end{smallmatrix}\right)$, with $X\in \Mb_{n-4}(\C)$, and elements of $U_q$ are given by matrices of the form $\left(\begin{smallmatrix} 1 & * & * \\ 0 & I_{n-4} & * \\ 0 & 0 & 1 \end{smallmatrix}\right)$.
The following lemma follows from  direct calculations
\begin{Lemma}\label{lm:unipotent:rad:prop}
Let $H$ be the matrix of the intersection form on $\Wbb^{(q)}$ in the basis $(g_1,\dots,g_{m-1},g_{m+2},\dots,g_n)$ and $\mu:=(1-q)(1-\bar{q})=-\imath\cdot\langle g_m,w\rangle$.
Let $\left(\begin{smallmatrix} 1 & {}^tx & a \\ 0 & I_{n-4} & x' \\ 0 & 0 & 1 \end{smallmatrix} \right)$ be an element of $U_q$ with $x,x'\in \C^{n-4}$, $a\in \C$. Then we have
\begin{itemize}
\item[$\bullet$] $x'=-\imath\cdot\mu\cdot H^{-1}\cdot\bar{x}$, and

\item[$\bullet$] $\Im(a)=-\frac{\mu}{2}\cdot{}^tx\cdot H ^{-1}\cdot \bar{x}$.
\end{itemize}
\end{Lemma}
By Lemma~\ref{lm:unipotent:rad:prop} we can identify $U_q$ with $\C^{n-4}\times\R$ by the map $\left(\begin{smallmatrix} 1 & {}^tx & a \\ 0 & I_{n-4} & y \\ 0 & 0 & 1 \end{smallmatrix} \right) \mapsto (x,\mathrm{Re}(a))$, where the group law on $\C^{n-4}\times \R$ is given by $(x,s)\cdot (y,t)=(x+y, s+t+\mu\mathrm{Im}({}^tx\cdot H^{-1}\cdot \bar{y}))$.

\medskip

Let $N_q$ denote the subgroup of $U_q$ that corresponds to $\{0\}\times\R$,
that is $N_q$ is the image of the injection from $\R$ into $U_q$ by the map $s \mapsto \left(\begin{smallmatrix} 1 & 0 & s \\ 0 & I_{n-4} & 0 \\ 0 & 0 & 1 \end{smallmatrix}\right)$.
We have following exact sequence
\begin{equation}\label{eq:exact:seq:U}
\{0\} \to N_q\simeq \R \to U_q \overset{\chi_q}{\to} \C^{n-4} \to \{0\}
\end{equation}
where the map $\chi_q: U_q \to U_q/N_q\simeq\C^{n-4}$ sends $(x,s)$ to $x$.

\medskip

Observe that $U_q/N_q\simeq \C^{n-4}$ can be identified with the space $\mathrm{Hom}(\Wbb^{(q)},\Lbb^{(q)})$ of linear maps $\Wbb^{(q)} \to \Lbb^{(q)}$.
By definition, for all $A\in U_q$ and $v \in \Wbb^{(q)}$, we have
$$
A(v)=v+\chi_q(A)(v).
$$
where $\chi_q(A)$ is viewed as an element of $\mathrm{Hom}(\Wbb^{(q)},\Lbb^{(q)})$.
\begin{Lemma}\label{lm:conjugate:P:on:U}
Let $A$ be an element of $\hat{P}_q$ which is represented by a matrix of the form $\left(\begin{smallmatrix} \lambda & {}^tu & a \\ 0 & C & u' \\ 0 & 0 & \lambda'\end{smallmatrix} \right)$ in the basis $\Gcal^{(q)}$.
Let $M$ be an element of $U_q$. Then $A\cdot M \cdot A^{-1} \in U_q$, and we have
\begin{equation}\label{eq:conjugate:P:on:U}
\chi_q(A\cdot M \cdot A^{-1})=\lambda\cdot\chi_q(M)\cdot C^{-1}.
\end{equation}
\end{Lemma}
\begin{proof}
We first remark that $\lambda'=\bar{\lambda}^{-1}$ because
$$
\langle w,g_m\rangle = \langle A(w),A(g_m)\rangle= \langle \lambda w, \lambda'g_m\rangle =\lambda\bar{\lambda}'\langle w,g_m\rangle.
$$
We have $A^{-1}=\left(\begin{smallmatrix} \lambda^{-1} & -\lambda^{-1}\cdot{}^tu\cdot C^{-1} & b \\ 0 & C^{-1} & -\bar{\lambda}\cdot C^{-1}\cdot u' \\ 0 & 0 & \bar\lambda\end{smallmatrix} \right)$.
Let $x:=\chi_q(M) \in \C^{n-4}$.
We can write $M=\left(\begin{smallmatrix} 1 & {}^tx & s \\ 0 & I_{n-4} & x' \\ 0 & 0 & 1\end{smallmatrix} \right)$, hence
$$
A\cdot M\cdot A^{-1}=\left(\begin{smallmatrix} \lambda & {}^tu & a \\ 0 & C & u' \\ 0 & 0 & \bar{\lambda}^{-1}\end{smallmatrix} \right)\cdot \left(\begin{smallmatrix} 1 & {}^tx & s \\ 0 & I_{n-4} & x' \\ 0 & 0 & 1\end{smallmatrix} \right) \cdot \left(\begin{smallmatrix} \lambda^{-1} & -\lambda^{-1}\cdot {}^tu\cdot C^{-1} & b \\ 0 & C^{-1} & -\bar{\lambda}\cdot C^{-1}\cdot u' \\ 0 & 0 & \bar{\lambda} \end{smallmatrix} \right) = \left(\begin{smallmatrix} 1 & \lambda\cdot{}^tx\cdot C^{-1} & * \\ 0 & I_{n-4} & \bar{\lambda}\cdot C^{-1}\cdot x' \\ 0 & 0 & 1 \end{smallmatrix} \right)
$$
and the lemma follows.
\end{proof}

\begin{Corollary}\label{cor:dual:action}
Let $A$ be an element of $\U(\Vbb^{(q)})$ such that $A$ restricts to the identity on $\Wbb^{(q)}_{m+1,m+1}$. Then $A$ preserves $\Wbb^{(q)}$, and for all  $M\in U_q$, we have $A\cdot M \cdot A^{-1} \in U_q$ and
$$
\chi_q(A\cdot M \cdot A^{-1})=\chi_q(M)\circ A^{-1}_{|\Wbb^{(q)}}
$$
(here we identify $U_q/N_q$ with $\mathrm{Hom}(\Wbb^{(q)},\Lbb^{(q)})$).
\end{Corollary}
\begin{proof}
Since $\Wbb^{(q)}=\Wbb^{(q)\perp}_{m+1,m-1}$, $A$ preserves $\Wbb^{(q)}$ as well. The matrix of $A$ in the basis $\Gcal^{(q)}$ has the form $\left(\begin{smallmatrix} 1 & 0 & 0 \\ 0 & C & 0 \\ 0 & 0 & 1 \end{smallmatrix}\right)$. The corollary then follows  immediately from \eqref{eq:conjugate:P:on:U}.
\end{proof}

Define $\Wcal^{(q)}:=\Hom((\Wbb^{(q)}\oplus\Lbb^{(q)})/\Lbb^{(q)},\Lbb^{(q)})$.
Elements of $\Wcal^{(q)}$ are linear maps $f: \Wbb^{(q)}\oplus\Lbb^{(q)} \to \Lbb^{(q)}$ such that $\Lbb^{(q)}\subset \ker(f)$.
By definition, any $A \in \hat{P}_q$ preserves the spaces $\Lbb^{(q)}$ and $\Wbb^{(q)}\oplus\Lbb^{(q)}$. Thus $A$ induces an automorphism on $\Wcal^{(q)}$ given by $f \mapsto (x \mapsto A\circ f \circ A^{-1}(x))$.
We thus have a representation $\Rcal_q: \hat{P}_q \to \Aut(\Wcal^{(q)})$.

By identifying $U_q/N_q$ with $\mathrm{Hom}(\Wbb^{(q)},\Lbb^{(q)}) \simeq \mathrm{Hom}((\Wbb^{(q)}\oplus\Lbb^{(q)})/\Lbb^{(q)},\Lbb^{(q)})$ we can consider  $\chi_q$ as a group morphism from $U_q/N_q$ onto $(\Wcal^{(q)},+)$.
We  summarize the content of Lemma~\ref{lm:conjugate:P:on:U} and Corollary~\ref{cor:dual:action} by the following
\begin{Proposition}\label{prop:action:conj:P}
The group $\hat{P}_q$ acts by conjugation on $U_q$, and we have
\begin{itemize}

\item[(a)] For all $A\in \hat{P}_q, \; M \in U_q$, $\chi_q(A\cdot M \cdot A^{-1})=\Rcal_q(A)(\chi_q(M))$.

\item[(b)] If $A\in \U(\Vbb^{(q)})$ is identity on $\Wbb^{(q)}_{m+1,m-1}$, then $A\in \hat{P}_q$ and the action of $\Rcal_q(A)$ on $\Wcal^{(q)}$ is isomorphic to the dual action of $A$ on $\Wbb^{(q)*}$.
\end{itemize}
\end{Proposition}

In what follows,  we identify $\PB_{m}$ with the subgroup of $\PB_n$  generated by the set of Dehn twists $\{\alpha_{i,j}, \; 1 \leq i < j \leq m\}$. We  denote by $\PB_{m+1,n}$ the subgroup of $\PB_n$ generated by the Dehn twists $\{\alpha_{i,j}, \; m+1 \leq i < j \leq n\}$.
We will investigate the action of $\rho_q(\PB_m)$ and of $\rho_q(\PB_{m+1,n})$ on $\Wcal^{(q)}$. The results in this section will be used to show that the intersection of $\rho_d(\PB_n)$ with $U:=\prod_{q\in \Ub^+_{d,{\rm prim}}}U_q$ contains a lattice of $U$.

Denote by $\Vcal^{(q)}_{m-1}$ the subspace of $\Wcal^{(q)}$ consisting of elements $\eta$  such that $\eta_{|\Wbb^{(q)}_{m+1}}\equiv 0$. Similarly, denote by  $\Wcal^{(q)}_{m+1}$ the subspace of $\Wcal^{(q)}$ consisting of elements $\eta$  such that $\eta_{|\Vbb^{(q)}_{m-1}}\equiv 0$. We clearly have  $\Wcal^{(q)}=\Vcal^{(q)}_{m-1}\oplus \Wcal^{(q)}_{m+1}$.
\begin{Lemma}\label{lm:action:PB:m}
For all $q\in \Ub_{d,\prim}$ and all  $\tau\in \PB_m$, we have $\rho_q(\tau) \subset \hat{P}_q$. Moreover, $\Rcal_q\circ\rho_q(\tau)$ preserves $\Vcal^{(q)}_{m-1}$ and acts trivially  on $\Wcal^{(q)}_{m+1}$.
\end{Lemma}
\begin{proof}
It is enough to show that $\rho_q(\alpha_{i,j}) \in \hat{P}_q$ for all $1 \leq i < j \leq m$. Since the boundary of $E_m$ is disjoint from the support of $\alpha_{i,j}$, we immediately get $\rho_q(\alpha_{i,j})(w)=w$.
If $j<m$, then the support of $\alpha_{i,j}$ does not meet the support of $g_m$, therefore we also have $\rho_q(\alpha_{i,j})(g_m)=g_m$, which means that $\rho_q(\alpha_{i,j})$ is identity on $\Wbb^{(q)}_{m+1,m-1}$, and we conclude by Proposition~\ref{prop:action:conj:P}.
%As a consequence $\rho_q(\alpha_{i,j})$ preserves $\Wbb^{(q)}=\Wbb^{(q)\perp}_{m+1,m-1}$ and the flag \eqref{eq:flag:filtration}.

Assume now that $j=m$. We will only consider the case $i=m-1$, the other cases follow by a renumbering of $\{b_1,\dots,b_{m-1}\}$.
Since $\rho_q(\alpha_{m-1,m})(g_i)=g_i$ for all $i\in \{1,\dots,m-3\}\cup\{m+2,\dots,n-1\}$, all we need to show is that $\rho_q(\alpha_{m-1,m})(g_{m-2}) \in \Lbb^{(q)}\oplus\Wbb^{(q)}$. By Theorem~\ref{th:Menet:alphaij}, we have
\begin{align*}
\rho_q(\alpha_{m-1,m})(g_{m-2}) & = g_{m-2}+(1-q^{k_m})g_{m-1} = g_{m-2}+(1-\bar{q}^{k_1+\dots+k_{m-1}})g_{m-1} \\
 & = g_{m-2}+\sum_{i=1}^{m-2}(\bar{q}^{k_1+\dots+k_i}-1)\cdot g_i -w.
\end{align*}
Therefore $\rho_q(\alpha_{m-1,m})(g_{m-2}) \in \Lbb^{(q)}\oplus\Vbb^{(q)}_{m-1}$, and hence $\rho_{q}(\alpha_{m-1,m})\in \hat{P}_q$.

For all $l \in \{m+2,\dots,n-1\}$, since the support of $\alpha_{i,j}$, with $1 \leq i < j \leq m$, is disjoint from the support of $g_l$, we have $\rho_q(\alpha_{i,j})(g_l)=g_l$. As a consequence, $\Rcal_q\circ \rho_q(\alpha_{i,j})$ preserves the space $\Vcal^{(q)}_{m-1}$.
\end{proof}

Similarly, we have
\begin{Lemma}\label{lm:action:PB:m+1:n}
For all $q \in \Ub_{d,\prim}$ and all $ \tau\in \PB_{m+1,n}$, we have $\rho_q(\tau) \in \hat{P}_q$. Moreover, the action of $\Rcal_q\circ\rho_q(\tau)$ on $\Wcal^{(q)}$ preserves the  decomposition $\Wcal^{(q)}=\Vcal^{(q)}_{m-1}\oplus\Wcal^{(q)}_{m+1}$ and is trivial on $\Vcal^{(q)}_{m-1}$.
\end{Lemma}
\begin{proof}
It is enough to prove the lemma for all $\alpha_{i,j}$ with $m+1 \leq i < j \leq n$. It is clear that for such an $\alpha_{i,j}$, the restriction of $\rho_q(\alpha_{i,j})$ on $\Lbb^{(q)}\oplus\Vbb^{(q)}_{m-1}$ is identity.
If $i\geq m+2$ then $\rho_q(\alpha_{i,j})(g_m)=g_m$, which means that $\rho_q(\alpha_{i,j})$ restricts to the identity on $\Wbb^{(q)}_{m+1,m-1}$. Thus we can conlude by Proposition~\ref{prop:action:conj:P} in this case.

Assume now that $i=m+1$.
We only need to consider the case $j=m+2$, the other cases follows from a renumbering of $\{b_{m+2},\dots,b_n\}$.
We need to show that $\rho_q(\alpha_{m+1,m+2})(\Wbb^{(q)}_{m+1}) \subset \Wbb^{(q)}_{m+1}\oplus\Lbb^{(q)}$.
Recall that $\{g_{m+2},\dots,g_{n-1}\}$ is a basis of $\Wbb^{(q)}_{m+1}$.
For all $s\geq m+3$, we have $\rho_q(\alpha_{m+1,m+2})(g_s)=g_s$.
For $s=m+2$, we have (see Theorem~\ref{th:Menet:alphaij})
$$
\rho_q(\alpha_{m+1,m+2})(g_{m+2}) = q^{k_{m+2}}(1-q^{k_{m+1}})\cdot g_{m+1}+g_{m+2}=q^{k_{m+1}+k_{m+2}}(\bar{q}^{k_{m+1}}-1)\cdot g_{m+1}+g_{m+2}.
$$
Recall that by assumption $q^{k_1+\dots+k_n}=1$. Therefore we have (see Theorem~\ref{th:Menet:generator:set})
$$
\sum_{i=1}^{n-1}(\bar{q}^{k_1+\dots+k_i}-1)\cdot g_i=0.
$$
Hence
$$
w=\sum_{i=1}^{m-1} (\bar{q}^{k_1+\dots+k_i}-1)\cdot g_i=-\sum_{i=m+1}^{n-1}(\bar{q}^{k_{m+1}+\dots+k_i}-1)\cdot g_i.
$$
In particular, we get
$$
(\bar{q}^{k_{m+1}}-1)\cdot g_{m+1}=-w-\sum_{i=m+2}^{n-1}(\bar{q}^{k_{m+1}+\dots+k_i}-1)\cdot g_i \in \Lbb^{(q)}\oplus \Wbb^{(q)}_{m+1}.
$$
It follows that $\rho_q(\alpha_{m+1,m+2})$ preserves the flag \eqref{eq:flag:filtration}, and $\rho_q(\alpha_{m+1,m+2}) \in \hat{P}_q$. We can then conclude that $\rho_q(\PB_{m+1,n})\subset \hat{P}_q$.
The second assertion of the lemma is also clear from the arguments  above.
\end{proof}

\begin{Lemma}\label{lm:action:PB:m:q:irred}
For all $q\in \Ub_{d,\prim}$, the actions of $\Rcal_q\circ\rho_q(\PB_{m})$ on $\Vcal^{(q)}_{m-1}$ and the action of $\Rcal_q\circ\rho_q(\PB_{m+1,n})$ on $\Wcal^{(q)}_{m+1}$ are irreducible.
\end{Lemma}
\begin{proof}
We will show that the action of $\Rcal_q\circ\rho_q(\PB_{m-1})$ on $\Vcal^{(q)}_{m-1}$ is irreducible, where $\PB_{m-1}$ is the subgroup of $\PB_m$ generated by $\{\alpha_{i,j}, \; 1 \leq i< j \leq m-1\}$.
Since $\rho_q(\PB_{m-1})$ acts trivially on $\Wbb_{m+1,m-1}$, by Proposition~\ref{prop:action:conj:P}, the action of $\Rcal_q\circ\rho_q(\PB_{m-1})$ on $\Vcal^{(q)}_{m-1}$ is isomorphic to the dual action of $\rho_q(\PB_{m-1})$ on $\Vbb^{(q)*}_{m-1}$. Thus it is enough to show that the action of $\rho_q(\PB_{m-1})$ on $\Vbb^{(q)}$ is irreducible.

Let $V^{(q)}\neq \{0\}$ be a $\rho_q(\PB_{m-1})$-invariant subspace of $\Vbb^{(q)}_{m-1}$. Consider a vector $v\in V^{(q)}\setminus\{0\}$. Since the family  $(g_1,\dots,g_{m-2})$ is a basis of $\Vbb^{(q)}_{m-1}$, and $\langle.,.\rangle$ is non-degenerate on $\Vbb^{(q)}_{m-1}$, there is some $i\in \{1,\dots,m-2\}$ such that $\langle v,g_i\rangle \neq 0$. Applying the twist $\alpha_{i,i+1}$ (which is an element of $\PB_{m-1}$) to $v$, we get (by Theorem~\ref{th:Menet:Dehn:twist})
$$
\rho_q(\alpha_{i,i+1})(v)=v-\imath\frac{(1-q^{k_i})(1-q^{k_{i+1}})}{(1-q)(1-\bar{q})}\langle v, g_i\rangle g_i \in V^{(q)}.
$$
It follows that $g_i\in V^{(q)}$. Applying successively $\alpha_{i-1,i},\dots,\alpha_{1,2}$ and $\alpha_{i+1,i+2}, \dots,\alpha_{m-2,m-1}$, we obtain that $\{g_1,\dots,g_{m-2}\}$ is contained in $V^{(q)}$. Therefore $V^{(q)}=\Vbb^{(q)}_{m-1}$, and the claim follows.

The irreducibility of the action of $\Rcal_q\circ\rho_q(\PB_{m+1,n})$ on $\Wcal^{(q)}_{m+1}$ follows from similar arguments.
\end{proof}

\subsection{Braids that are mapped in to $N_q$}  \label{subsec:inters:G:N:q:no:empty}
We will now show that the intersection $\rho_q(\PB_n)\cap N_q$ is non-empty.
\begin{Lemma}\label{lm:U:q:inters:Gam:non:triv}
If $m\geq 3$, then there exists $\tau \in \PB_m$ such that for all  $q\in \Ub_{d,\prim}$, $\rho_q(\tau) \in U_q$ and $\nu_q:=\chi_q(\rho_q(\tau)) \in \Vcal^{(q)}_{m-1}\setminus\{0\}$.
%Moreover, in an appropriate basis of $\Vcal^{(q)}_{m-1}$, $\nu_q$ is represented by a vector with coefficients in $K_d$.
Similarly, if  $m \leq n-3$, then there exists $\tau'\in \PB_{m+1,n}$ such that or all $q\in \Ub_{d,\prim}$, $\rho_q(\tau') \in U_q$ and $\nu'_q:=\chi_q(\rho_q(\tau')) \in \Wcal^{(q)}_{m+1}\setminus\{0\}$.
%Moreover, in an appropriate basis of $\Wcal^{(q)}_{m+1}$, $\nu'_q$ is represented by a vector with coefficients in $K_d$.
\end{Lemma}
\begin{proof}
We consider two Dehn twists $\alpha$ and $\beta$ with support in $E_m$ defined as follows: $\alpha:=\alpha_{m-1,m}$ is the Dehn twist about the boundary of a disc $D_{m-1} \subset E_m$ such that  $D_{m-1}\cap\Bcal=\{b_{m-1},b_{m}\}$, and  $\beta$ is the Dehn twist about the boundary of the disc $E_{m-1}$ (see Figure~\ref{fig:Dehn:twists:E:m}).

\begin{figure}[htb]
\begin{tikzpicture}[scale=0.5]
\foreach \x in {(-6,0), (0,0), (4,0), (8,0), (12,0), (18,0)} \filldraw[gray] \x circle (3pt);
\draw (-3,0) node {$\dots$};
\draw (15,0) node {$\dots$};
\draw (-1,0) ellipse (7 and 4);
\draw[thin, red] (2,0) ellipse (3 and 2);
\draw[thin, red] (-2.5,0) ellipse (4.5 and 3);
\draw[thin, red] (10,0) ellipse (3 and 2);
\draw[thin, red] (15,0) ellipse (5 and 3);
\draw (-6,0) node[below] {$b_1$}
      (0.3,0)  node[below] {$b_{m-1}$}
      (4,0)  node[below] {$b_m$};
\draw (2,-2) node[below] {$\alpha$}
      (0,3) node {$\beta$};
\draw (8,0) node[below] {$b_{m+1}$}
      (12,0) node[below] {$b_{m+2}$}
      (18,0) node[below] {$b_n$};
\draw (10,-2) node[below] {$\alpha'$}
      (15,-3.5) node {$\beta'$};
\draw (-1,-4.5) node {$E_m$};
\end{tikzpicture}
\caption{Dehn twists in $\hat{P}_q$.}
%\caption{Dehn twists preserving the flag \eqref{eq:flag:filtration}}
\label{fig:Dehn:twists:E:m}
\end{figure}
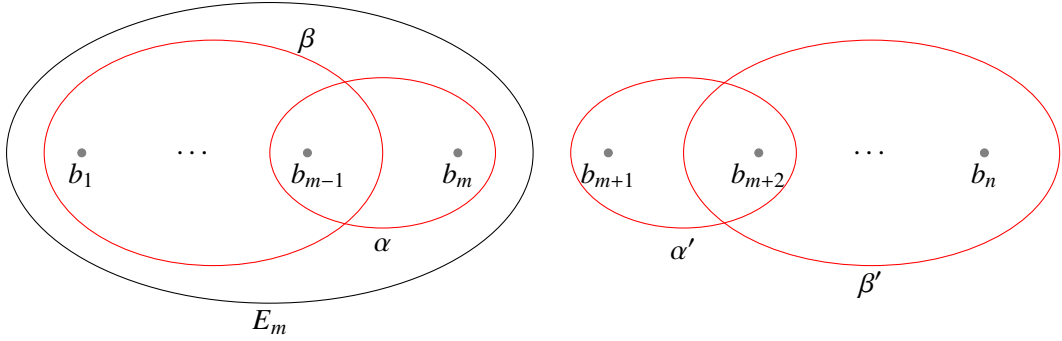
%The action of $\rho_q(\alpha)$ on the basis $\Gcal^{(q)}$ of $\Vbb^{(q)}$ is as follows
%\begin{itemize}
%\item[.] $\rho_q(\alpha)(w)=w$ (since the boundary of $E_m$  does not intersect the support of $\alpha$).
%
%\item[.] $\rho_q(\alpha)(g_i)=g_i$, for all $i<m-2$.
%
%\item[.] $\rho_q(\alpha)(g_{m-2})= g_{m-2}+(1-q^{k_m})g_{m-1} = g_{m-2} + \sum_{i=1}^{m-2}(\bar{q}^{k_1+\dots+k_i}-1)g_i -w$.
%
%\item[.] $\rho_q(\alpha)_{|\Wbb_{m+1}}=\Id_{\Wbb_{m+1}}$.
%\end{itemize}
%In particular, we see that $\rho_q(\alpha)$ preserves the flag \eqref{eq:flag:filtration}. Note however that $\rho_q(\alpha)$ is not necessarily in $\SU(\Vbb^{(q)})$ since $\det(\rho_q(\alpha))=q^{k_{m-1}+k_m}$.
%
%
%The action of $\rho_q(\beta)$ on the basis $\Gcal^{(q)}$ is as follows
%\begin{itemize}
%\item[.] $\rho_q(\beta)(w)=w$.
%
%\item[.] $\rho_q(\beta)_{|\Vbb_{m-1}}=q^{k_1+\dots+k_{m-1}}\cdot \Id_{\Vbb_{m-1}}=\bar{q}^{k_m}\cdot \Id_{\Vbb^{m-1}}$.
%
%\item[.] $\rho_q(\beta)_{|\Wbb_{m+1}}=\Id_{\Wbb_{m+1}}$.
%\end{itemize}
%Again, we see that $\rho_q(\beta)$ preserves the flag \eqref{eq:flag:filtration}.
We now consider $\tau:=[\alpha,\beta]$.
By direction calculations, we get that
\begin{itemize}
\item[.] $\rho_q(\tau)(w)=w$,

\item[.] $\rho_q(\tau)(g_i)=g_i$, for all $i=1,\dots,m-3$,

\item[.] $\rho_q(\tau)(g_{m-2})=g_{m-2}+(\bar{q}^{k_m}-1)\cdot w$,

\item[.] $\rho_q(\tau)(g_j)=g_j$, for all $j=m+2,\dots,n-1$.
\end{itemize}
Since $\det(\rho_q(\tau))=1$, we  must have $\rho_q(\tau)(g_m)\in g_m+ \Lbb^{(q)}\oplus \Wbb^{(q)}$.
As a consequence, $\rho_q(\tau)\in U_q$.
Let  $\nu_q(\tau):=\chi_q(\rho_q(\tau)) \in \mathrm{Hom}((\Wbb^{(q)}\oplus\Lbb^{(q)})/\Lbb^{(q)},\Lbb^{(q)}) \simeq \mathrm{Hom}(\Wbb^{(q)},\Lbb^{(q)})$.
By definition
$$
\nu_q(\tau)(g_i)= \left\{
\begin{array}{cl}
0 & \text{ if } i\in \{1,\dots,m-3\}\cup\{m+2,\dots,n-1\}\\
(\bar{q}^{k_m}-1)\cdot w & \text{ if } i=m-2.
\end{array}
\right.
$$
Since $\bar{q}^{k_m}\neq 1$, we have that $\nu_q(\tau) \neq 0$ which proves the first assertion of the lemma.

\medskip

For the second assertion,  remark that since $q^{k_1+\dots+k_n}=q^{k_1+\dots+k_m}=1$, we have
$\sum_{i=1}^{n-1}(\bar{q}^{k_1+\dots+k_i}-1)g_i=0$, and
$$
w=\sum_{i=1}^{m-1}(\bar{q}^{k_1+\dots+k_i}-1)g_i=-\sum_{i=m+1}^{n-1}(\bar{q}^{k_{m+1}+\dots+k_i}-1)g_i.
$$
Let $\alpha':=\alpha_{m+1,m+2}$ be the Dehn twist about the border of a disc $D_{m+1}$ outside of $E_m$ such that $D_{m+1}\cap\Bcal=\{b_{m+1}, b_{m+2}\}$, and $\beta'$ the Dehn twist about the border of a disc which is disjoint from $E_{m+1}$ and contains all the points in $\{b_{m+2},\dots,b_n\}$ (see Figure~\ref{fig:Dehn:twists:E:m}). By using the formulas of Theorem~\ref{th:Menet:alphaij} and Theorem~\ref{th:Menet:Dehn:twist}, one readily checks that $\tau':=[\alpha',\beta'] \in \PB_{m+1,n}$  satisfies the second assertion.
\end{proof}

\begin{Remark}\label{rk:Gamma:inters:U:q:Venky}
The existence of $\tau$ and $\tau'$ in Lemma~\ref{lm:U:q:inters:Gam:non:triv} was known to Venkataramana~\cite{Venky:Annals,Venky:Invent}.
\end{Remark}

\section{Proof of arithmeticity}\label{sec:prf:arithmetic}
\subsection{Reduction to the case $d$ divides $k_1+\dots+k_n$}\label{subsec:prf:arith:reduction}
We first show that to prove Theorem~\ref{th:main:arithm}, it is enough to consider the case $d \, | \, (k_1+\dots+k_n)$.
In what follows, we identify $\PB_{n-1}$ with the subgroup of $\PB_n$ generated by  $\{\alpha_{i,j}, \;  1 \leq i < j \leq n-1\}$.
\begin{Lemma}\label{lm:rho:P:n-1:finite:ind:rho:P:n}
Assume that $d\; | \; (k_1+\dots+k_n)$.  Then for all $q \in \pdroots$, $\rho_q(\PB_{n-1})$ is a normal subgroup of $\rho_q(\PB_n)$, and $\rho_q(\PB_n)/\rho_q(\PB_{n-1})$ is a finite cyclic group.
\end{Lemma}
\begin{proof}
We will show that for all $\alpha\in \PB_n$, there exist $\alpha'\in \PB_{n-1}$ and a  $d$-th root of unity $\lambda$ such that $\rho_q(\alpha)=\lambda\cdot\rho_q(\alpha')$. Clearly, it is enough to show this property for $\alpha\in \{\alpha_{i,n}, \; 1 \leq i <  n\}$.
We will only consider the case $\alpha=\alpha_{n-1,n}$ since the other cases follow from a renumbering of the points in $\{b_1,\dots,b_{n-1}\}$.

Since $d \; |\; (k_1+\dots+k_n)$, $(g_1,\dots,g_{n-2})$ is a basis of $\Vbb^{(q)}$, and we have the following relation
$$
\sum_{i=1}^{n-1}(\bar{q}^{k_1+\dots+k_i}-1)g_i=0
$$
which implies that $\sum_{i=1}^{n-2}(\bar{q}^{k_1+\dots+k_i}-1)g_i=(1-\bar{q}^{k_1+\dots+k_{n-1}})g_{n-1}=(1-q^{k_n})g_{n-1}$.
By Theorem~\ref{th:Menet:alphaij} we have $\rho_q(\alpha_{n,n-1})(g_i)= g_i$ for all $i < n-2$, while
\begin{align*}
\rho_q(\alpha_{n,n-1})(g_{n-2})&=  g_{n-2}+ (1-q^{k_n})g_{n-1}=g_{n-2}+\sum_{i=1}^{n-2}(\bar{q}^{k_1+\dots+k_i}-1)g_i \\
                               &= q^{k_{n-1}+k_n}g_{n-2}+\sum_{i=1}^{n-3}(\bar{q}^{k_1+\dots+k_i}-1)g_i.
\end{align*}
In this case, we take $\alpha'$ to be the Dehn twist about the border of a topological disc $E$ such that $E\cap \Bcal=\{b_1,\dots,b_{n-2}\}$.
By Theorem~\ref{th:Menet:Dehn:twist}, we have
$$
\rho_q(\alpha')(g_i)=\left\{
\begin{array}{ll}
\bar{q}^{k_{n-1}+k_n}g_i & \text{ if } i < n-2, \\
g_{n-2}+\bar{q}^{k_{n-1}+k_n}\sum_{i=1}^{n-3}(\bar{q}^{k_1+\dots+k_i}-1)g_i & \text{ if } i=n-2.
\end{array}
\right.
$$
One readily  checks that $\rho_q(\alpha)(g_i)=q^{k_{n-1}+k_n}\rho_q(\alpha')(g_i)$ for all $i=1,\dots,n-2$. The claim is then proved.
It follows immediately that $\rho_q(\PB_{n-1})$ is a normal subgroup of $\rho_q(\PB_n)$.

Let $Z(\U(\Vbb^{(q)}))\simeq \S^1$ denote the centre of $\U(\Vbb^{(q)})$.
Let $\Zb_n:=\rho_q(\PB_n)\cap Z(\U(\Vbb^{(q)}))$, and $\Zb_{n-1}:=\rho_q(\PB_{n-1})\cap Z(\U(\Vbb^{(q)}))$.
It also follows from the claim that $\rho_q(\PB_n)/\rho_q(\PB_{n-1})=\Zb_{n}/\Zb_{n-1}$.
Since $\rho_q(\PB_n)$ is contained in $\U_d(\Vbb^{(q)})$ (that is the set of $A\in \U(\Vbb^{(q)})$ such that $(\det A)^d=1$), $\Zb_n$ (and hence $\Zb_{n-1}$) is contained in the group of $d(n-2)$-th roots of unity.
In particular $\Zb_n$ and $\Zb_{n-1}$ are both finite, which implies that $\rho_q(\PB_n)/\rho_q(\PB_{n-1})$ is finite.
\end{proof}

\begin{Corollary}\label{cor:rho:P:n-1:finite:ind:P:n}
Assume that $d\, | \, (k_1+\dots+k_n)$, then $\rho_d(\PB_{n-1})$ is a normal subgroup of finite index in $\rho_d(\PB_n)$.
\end{Corollary}

As a consequence of Corollary~\ref{cor:rho:P:n-1:finite:ind:P:n}, we get

\begin{Lemma}\label{lm:reduction:d:divides:sum:ki}
To prove Theorem~\ref{th:main:arithm}, one only needs to consider the case $d \,| \, (k_1+\dots+k_n)$.
\end{Lemma}
\begin{proof}
Assume that $d \, \nmid \, (k_1+\dots+k_n)$. Let $k_{n+1}$ be the integer such that $0 < k_{n+1} <d$, and $k_{n+1}\equiv -(k_1+\dots+k_n) \mod d$.
Let $f$ be  an automorphism  of $\CP^1$ such that $f(\Bc\cup\{\infty\}) \subset \C$. Set $b'_i:=f(b_i)$ for all $i=1,\dots,n$, and $b'_{n+1}=f(\infty)\in \C$.
Then $\hX$ is isomorphic to the Riemann surface $\hX'$ constructed from the curve $y^d=\prod_{i=1}^{n+1}(x-b'_i)^{k_i}$. Moreover, the $\Z/d\Z$ actions on $\hX$ and $\hX'$ are equivariant with respect to isomorphism $\hX \to \hX'$. Therefore, we can identify  the representation $\rho_q: \PB_n\to \U(H^1(\hX)_q)$ with the restriction of the representation $\rho'_q: \PB_{n+1} \to \U(H^1(\hX')_q)$ to $\PB_n$. Since $\rho_d(\PB_n)$ has finite index in $\rho'_d(\PB_{n+1})$, if the conclusion of Theorem~\ref{th:main:arithm} holds for $\rho'_d(\PB_{n+1})$, it also holds for $\rho_d(\PB_n)$.
\end{proof}

\subsection{Zariski density}\label{subsec:prf:arith:Zar:density}
Having proved Lemma~\ref{lm:reduction:d:divides:sum:ki}, in the sequel we make the assumption that $d$ divides $k_1+\dots+k_n$. Note that the assumption
$$
n+1-\eps_0 \geq 5
$$
implies that $n\geq 5$ in this case.
From now on, we will suppose that $\gcd(k_1,\dots,k_n,d)=1$, and there exists $m \in \{2,\dots,n-1\}$ such that $d \, | \, (k_1+\dots+k_m)$.
Since $n\geq 5$, at least one of the following holds: (i) $m\geq 3$, (ii) $n-m \geq 3$. 
We first show that the Zariski closure of $\rho_q(\PB_n)$ contains $\prod_{q\in \Ub^+_{d, \prim}}\SU(\Vbb^{(q)})$.
\begin{Lemma}\label{lm:0:mod:d:good}
For all $k\in \Z$ such that $\gcd(k,q)=1$, the sequence $(\left\{\frac{kk_1}{d}\right\},\dots,\left\{\frac{kk_1}{d}\right\})$ satisfies Definition~\ref{def:good:weights}(a).
\end{Lemma}
\begin{proof}
Since $\gcd(k,d)=1$ and $0<k_i<d$, we have $d\; \nmid kk_i$, which means that $0< \left\{ \frac{kk_i}{d} \right\} < 1$  for all $i=1,\dots,n$. Therefore, we have
$$
0< \left\{\frac{kk_1}{d}\right\}+\dots+\left\{\frac{kk_m}{d}\right\} < m.
$$
By assumption $\frac{k(k_1+\dots+k_m)}{d} \in \Z$, therefore $\left\{kk_1/d\right\}+\dots+\left\{kk_m/d\right\} \in \N$, which means that
$$
1 \leq \left\{\frac{kk_1}{d}\right\}+\dots+\left\{\frac{kk_m}{d}\right\} \leq m-1.
$$
As a consequence
$$
1 \leq \left\{\frac{kk_1}{d}\right\}+\dots+\left\{\frac{kk_m}{d}\right\}< \left\{\frac{kk_1}{d}\right\}+\dots+\left\{\frac{kk_n}{d}\right\} < m-1+(n-m)=n-1.
$$
\end{proof}
By definition, the morphism $\rho_d$  takes values in $\Sp(\hX,\R)^T_d$. Since $\Sp(\hX,\R)^T_d\simeq\prod_{q\in \Ub^+_{d,\prim}}\U(\Vbb^{(q)})$, we will identify $\rho_d$ with the morphism $\tau \mapsto (\rho_q(\tau))_{q\in \Ub^+_{d,\prim}} \in \prod_{q\in \Ub^+_{d,\prim}}\U(\Vbb^{(q)})$.

\begin{Corollary}\label{cor:Zar:dense:d:divides:partial:sum}
If $(k_1,\dots,k_n)$ satisfies the hypothesis of Theorem~\ref{th:main:arithm}. Then the identity component of the Zariski closure  of $\rho_d(\PB_n)$ in $\Sp(\hX,\R)^T_d$ is equal to $\prod_{q\in \Ub^+_{d,\prim}}\SU(\Vbb^{(q)})$.
\end{Corollary}
\begin{proof}
By Lemma~\ref{lm:0:mod:d:good}, for all $k \in \N$ such that $\gcd(d,k)=1$, the sequence $(\left\{\frac{kk_1}{d}\right\}, \dots, \left\{\frac{kk_n}{d}\right\})$ is good. Moreover, since $n+1-\eps_0\geq 5$, we have $n-1-\eps_0 \geq 3$. Thus we can conclude by Theorem~\ref{th:main:Zariski:dense}.
\end{proof}

\subsection{Horospherical subgroup in product}\label{subsec:horo:product}
%Recall that $\Vbb:=H^1(\hX)_d \subset H^1(\hX)$ is the kernel of $\phi_d(T^*)$, where $\phi_d$ is the $d$-th cyclotomic polynomial. We have $\Vbb=\bigoplus_{q\in \Ub_{d,\prim}}\Vbb^{(q)}$. Define
%$$
%\Vbb_{m-1}:=\bigoplus_{q\in \Ub_{d,\prim}}\Vbb^{(q)}_{m-1}, \quad  \Wbb_{m+1}:= \bigoplus_{q\in \Ub_{d,\prim}}\Wbb^{(q)}_{m+1}, \quad \text{ and } \quad \Wbb:=\bigoplus_{q\in\Ub_{d, \, {\rm prim}}}\Wbb^{(q)}= \Vbb_{m-1}\oplus\Wbb_{m+1}.
%$$
Recall that for each $q\in \Ub_{d,\prim}$,  we have
$$
\Vcal_{m-1}^{(q)}=\Hom(\Vbb^{(q)},\Lbb^{(q)}), \Wcal^{(q)}=\Hom(\Wbb^{(q)}_{m+1}, \Lbb^{(q)}), \, \text{ and } \, \Wcal^{(q)}= \Hom(\Wbb^{(q)},\Lbb^{(q)})\simeq \Vcal_{m-1}^{(q)}\oplus\Wcal_{m+1}^{(q)}.
$$
Define
$$
\Wcal:=\prod_{q\in \Ub_{d,\prim}} \Wcal^{(q)}, \; \quad \Vcal_{m-1}:=\prod_{q\in \Ub_{d,\prim}} \Vcal^{(q)}_{m-1}, \; \quad  \Wcal_{m+1}=\prod_{q\in \Ub_{d,\prim}} \Wcal^{(q)}_{m+1}.
$$
Note that we also have $\Wcal=\Vcal_{m-1}\oplus\Wcal_{m+1}$.
Let $\pf_q$ denote the projection from $\Wcal$ onto $\Wcal^{(q)}$ for all $q\in \Ub_{d,\prim}$.

Define $\hat{P}:=\prod_{q\in \Ub_{d,\prim}}\hat{P}_q$, where $\hat{P}_q$ is a subgroup of $\U(\Vbb^{(q)})$ that preserves the flag \eqref{eq:flag:filtration}.
We have a group morphism $\Rcal: \hat{P} \to \Aut(\Wcal)$ such that  the action of $\Rcal(A)$ on the $\Wcal^{(q)}$-factor of $\Wcal$ is given by the $\hat{P}_q$-factor of $A$, for all $A \in \hat{P}$.

For all $\alpha\in \PB_n$,  let $\hat{\rho}_d(\alpha):=(\rho_q(\alpha))_{q\in \Ub_{d, {\rm prim}}} \in \prod_{q\in \Ub_{d, {\rm prim}}}\U(\Vbb^{(q)})$.
It follows from Lemma~\ref{lm:action:PB:m} and Lemma~\ref{lm:action:PB:m+1:n}   that $\hat{\rho}_d(\PB_m)$ and $\hat{\rho}_d(\PB_{m+1,n})$ are both contained in $\hat{P}$. Moreover, $\Rcal\circ\hat{\rho}_d(\PB_m)$ preserves $\Vcal_{m-1}$ and acts trivially on $\Wcal_{m+1}$, while  $\Rcal\circ\hat{\rho}_d(\PB_{m+1,n})$ preserves $\Wcal_{m+1}$ and acts trivially on $\Vcal_{m-1}$.
We are now in a position to prove the following

\begin{Proposition}\label{prop:irred:rep:PB:m}
Assume that
\begin{itemize}
\item[$\bullet$] $\gcd(d,k_1,\dots,k_m)=1$ if $m\geq 3$,

\item[$\bullet$] $\gcd(d,k_{m+1},\dots,k_n)=1$ if $n-m\geq 3$.
\end{itemize}
Let $V$  (resp. $W$) be a vector subspace of $\Vcal_{m-1}$ (resp. of $\Wcal_{m+1}$) invariant under the action of $\Rcal\circ\hat{\rho}_d(\PB_{m})$ (resp. of $\Rcal\circ\hat{\rho}_d(\PB_{m+1,n})$).
If $V$ contains a vector $v$ such that $\pf_q(v)\neq 0$, for all $q\in \Ub_{d,\prim}$, then $V=\Vcal_{m-1}$. Similarly, if $W$ contains a vector $w$ such that $\pf_q(w)\neq 0$ for all $q\in \Ub_{d, \prim}$, then $W=\Wcal_{m+1}$.
\end{Proposition}
\begin{proof}
We will only give the prove for $V$, the proof for $W$ follows the same lines. Note that if $m=2$ then $\dim \Vcal^{(q)}_{m-1}=\{0\}$ for all $q \in \Ub_{d, \prim}$, and there is nothing to prove. Therefore, we will always assume that $m\geq 3$ and $\gcd(d,k_1,\dots,k_m)=1$.
\begin{Claim}\label{clm:V:m-1:equiv:iso:not:exist}
Let $q$ and $q'$ be two elements of $\Ub_{d, \prim}$. If there exists an isomorphism $\psi: \Vcal^{(q)}_{m-1} \to \Vcal^{(q')}_{m-1}$ such that $(\Rcal_{q'}\circ\rho_{q'}(\tau))\circ\psi=\psi\circ(\Rcal_q\circ\rho_q(\tau))$,  for all $\tau\in \PB_{m}$, then we must have $q'=q$.
\end{Claim}
\begin{proof}
In what follows, by a slight abuse of notation if $\tau$ is an element of $\PB_m$, by $\rho_q(\tau)$ (resp. by $\Rcal_q\circ\rho_q(\tau)$) we will mean the restriction of $\rho_q(\tau)$ to $\Vbb^{(q)}_{m-1}$ (resp. the restriction of $\Rcal_q\circ\rho_q(\tau)$ to $\Vcal^{(q)}_{m-1}$).

Let $(v_1,\dots,v_{m-2})$ a basis of $\Vcal^{(q)}_{m-1}$. Then $(\psi(v_1),\dots,\psi(v_{m-2}))$ is a basis of $\Vcal^{(q')}_{m-1}$. The matrix of $\Rcal_q\circ\rho_q(\tau)$ in $(v_1,\dots,v_{m-2})$ and the matrix of $\Rcal_{q'}\circ\rho_{q'}(\tau)$ in $(\psi(v_1),\dots,\psi(v_{m-2}))$ are the same. In particular, we have $\Tr(\Rcal_q\circ\rho_q(\tau))=\Tr(\Rcal_{q'}\circ\rho_{q'}(\tau))$.

Let $I$ be a subset of $\{1,\dots,m-1\}$ such that $|I|\geq 2$, and $E$ be a disc in $E_{m-1}$ such that $E\cap \Bcal=\{b_i, \; i\in I\}$. Consider  the  Dehn twist $\tau$ about the border of $E$. Note that $\tau$ acts trivially on $\Wbb^{(q)}_{m+1,m-1}$.
It follows from Proposition~\ref{prop:action:conj:P} (b) that we have
$$
\Tr(\Rcal_{q}\circ\rho_q(\tau))=\Tr(\rho_q(\tau)^*)=\Tr(\rho_q(\tau^{-1}))=(|I|-1)\bar{q}^{\sum_{i\in I} k_i}+ m-1-|I|
$$
while
$$
\Tr(\Rcal_{q'}\circ\rho_{q'}(\tau))=\Tr(\rho_{q'}(\tau)^*)=\Tr(\rho_{q'}(\tau^{-1}))=(|I|-1)\bar{q}'{}^{\sum_{i\in I} k_i}+ m-1-|I|.
$$
Therefore we have $q^{\sum_{i\in I} k_i}=q'{}^{\sum_{i\in I}k_i}$ for all $I\subset \{1,\dots,m-1\}$ such that $|I|\geq 2$. If $m-1\geq 3$, then we must have $q^{k_i}=q'{}^{k_i}$ for all $i=1,\dots,m-1$. In the case $m-1=2$ (that is $m=3$), we consider the action of $\alpha_{2,3}$.
In this case $q^{k_1+k_2+k_3}=1$, $\Vbb_{m-1}^{(q)}=\langle g_1\rangle$, and $w=(\bar{q}^{k_1}-1)g_1+(\bar{q}^{k_1+k_2}-1)g_2$.
The action of $\rho_q(\alpha_{2,3})$ satisfies $\rho_q(\alpha_{2,3})(w)=w$, and
\begin{align*}
\rho_q(\alpha_{2,3})(g_1) & =g_1+(1-q^{k_3})g_2=g_1+(1-\bar{q}^{k_1+k_2})g_2\\
& =g_1+(\bar{q}^{k_1}-1)g_1-w \\
& =\bar{q}^{k_1}g_1 -w.
\end{align*}
It follows that $\Rcal_q\circ\rho_q(\alpha_{2,3})=q^{k_1}\Id_{\Vcal^{(q)}_2}$. By the same argument, we also have $\Rcal_{q'}\circ\rho_{q'}(\alpha_{2,3})=q'{}^{k_1}\Id_{\Vcal^{(q')}_2}$. Therefore we get $q^{k_1}=q'{}^{k_1}$. Since we already have $q^{k_1+k_2}=q'{}^{k_1+k_2}$, we  conclude that $q^{k_2}=q'{}^{k_2}$ as well. Now, from the assumption $d \; | \; (k_1+\dots+k_m)$ we have $\gcd(d,k_1,\dots,k_{m-1})=\gcd(d,k_1,\dots,k_m)=1$ which allows us to conclude that $q=q'$.
\end{proof}

Let $\ell:=|\Ub^+_{d,\prim}|=\varphi(d)/2$, where $\varphi$ is the Euler's phi function.  Note that $|\Ub_{d,\prim}|=2\ell$.
Let us choose a numbering of the elements of $\Ub_{d,\prim}$ so that one can write $\Ub_{d,\prim}=\{q_i, \; i=1,\dots,2\ell\}$.  To simplify the notation, we will write $\Vcal^i_\bullet, \Wcal^i_\bullet, \pf_i$ instead of $\Vcal^{(q_i)}_\bullet, \Wcal^{(q_i)}_\bullet, \pf_{q_i}$ respectively.
Set $W_0:=V$, and for $k=1,\dots,2\ell$, $W_k:=W_{k-1}\cap\ker\mathfrak{p}_k$, or equivalently
$$
W_k=V\cap\left(\bigcap_{i=1}^k\ker(\pf_i)\right).
$$

\begin{Claim}\label{clm:V:m-1:proj:inv:subsp}
For all $k\in\{0,\dots,2\ell-1\},  \; i \in \{k+1,\dots,2\ell\}$, we have $\mathfrak{p}_i(W_k)=\Vcal^{i}_{m-1}$.
\end{Claim}
\begin{proof}
We first prove the claim in the case $k=0$, that is $W_0=V$. By assumption $\mathfrak{p}_i(V) \subset \Vcal^{i}_{m-1}$ is a $\Rcal_{q_i}\circ\rho_{q_i}(\PB_{m})$-invariant subspace. Since the action of $\Rcal_{q_i}\circ\rho_{q_i}(\PB_{m})$ on $\Vcal^i_{m-1}$ is irreducible (by Lemma~\ref{lm:action:PB:m:q:irred}), and $\mathfrak{p}_i(V)$ contains $\pf_i(v)\neq 0$, we must have $\mathfrak{p}_i(V)=\Vcal^{i}_{m-1}$.

Assume now that the claim holds for $k=r$ and all $i\in \{r+1,\dots,2\ell\}$. Consider $\pf_i(W_{r+1})$ for some $i>r+1$. By construction, $\pf_i(W_{r+1})$ is $\Rcal_{q_i}\circ\rho_{q_i}(\PB_{m})$-invariant. Since the action of $\Rcal_{q_i}\circ\rho_{q_i}(\PB_{m})$ on $\Vcal^i_{m-1}$ is irreducible, either $\pf_i(W_{r+1})=\{0\}$ or $\pf_i(W_{r+1})=\Vcal^{i}_{m-1}$.

Consider the projection $\theta_{r+1,i}: W_{r} \to \Vcal^{r+1}_{m-1}\oplus\Vcal^{i}_{m-1}$. If $\pf_i(W_{r+1})=\{0\}$ then $\theta_{r+1,i}(W_k)$ is the graph of a morphism $\psi: \Vcal^{r+1}_{m-1} \to \Vcal^{i}_{m-1}$.  Observe that $\psi$ is an isomorphism because $\psi(\Vcal^{r+1}_{m-1})=\pf_i(W_r)=\Vcal^{i}_{m-1}$ by the induction hypothesis, and we have $\dim \Vcal^{r+1}_{m-1}=\dim \Vcal^{i}_{m-1}$.

By construction, the morphism $\psi$ satisfies $\psi\circ(\Rcal_{q_{r+1}}\circ\rho_{q_{r+1}}(\tau))=(\Rcal_{q_i}\circ\rho_{q_i}(\tau))\circ\psi$, for all $\tau \in \PB_{m}$.
But Claim~\ref{clm:V:m-1:equiv:iso:not:exist} then implies that $q_{r+1}=q_i$, which is impossible. Therefore we must have $\pf_i(W_{r+1})=\Vcal^i_{m-1}$. By induction, the claim follows.
\end{proof}
It follows from Claim~\ref{clm:V:m-1:proj:inv:subsp} that $W_{2\ell-1}=\Vcal^{2\ell}_{m-1}$. By induction, one readily gets that
$$
V=W_0=\prod_{i=1}^{2\ell}\Vcal^i_{m-1}=\Vcal_{m-1},
$$
which proves the proposition.
\end{proof}

\subsection{Intersection of the image of $\rho_d(\PB_n)$ with the product of horospherical subgroups} \label{subsec:inters:horosp}
Let
$$
G:=\prod_{q\in \Ub^+_{d,\prim}}\SU(\Vbb^{(q)}), \quad \text{ and } \quad  \Gamma:=\rho_d(\PB_n)\cap G.
$$
Observe that $G$ is a semisimple real algebraic Lie group defined over $\Q$.
By Corollary~\ref{cor:Zar:dense:d:divides:partial:sum}, $\Gamma$ is Zariski dense in $G$.
We will show that $\Gamma$ is a lattice in $G$ by using Theorem~\ref{th:Margulis:crit}.
To this purpose, we consider the  subgroup $U:=\prod_{q\in \Ub^+_{d,\prim}}U_q$ of $G$,  where $U_q$ is defined in \textsection~\ref{subsec:horosph:on:factor}.
This is the unipotent radical of the parabolic subgroup $P=\prod_{q\in \Ub^+(d, \prim)}P_q$ of $G$ ($P_q$ is the subgroup of $\SU(\Vbb^{(q)})$ that preserves the flag \eqref{eq:flag:filtration}).
Therefore, $U$ is a horospherical subgroup of $G$.

\begin{Theorem}\label{th:inters:w:horosp:lattice}
Let $\kappa=(k_1,\dots,k_n)$, with $n\geq 5$, be a sequence of positive integers such that $0 < k_i < d$, for all $i=1,\dots,n$, and $d \, | \, (k_1+\dots+k_n)$.
Assume that there exists $m\in \{2,\dots,n-2\}$ such that $d \, | \, (k_1+\dots+k_m)$, and 
\begin{itemize}
\item[$\bullet$] $\gcd(k_1,\dots,k_m,d)=1$ if $m\geq 3$,

\item[$\bullet$] $\gcd(k_{m+1},\dots,k_n,d)=1$ if $n-m \geq 3$.
\end{itemize}
Then the intersection $\rho_d(\PB_n)\cap U$ contains a lattice of $U$.
\end{Theorem}

Recall that $\ell=|\Ub^+_{d,\prim}|=\varphi(d)/2$, where $\varphi$  is the Euler's phi function. We have the following exact sequence
\begin{equation}\label{eq:product:upotent:ex:seq}
\{0\}\to N:=\prod_{q\in \Ub^+_{d,\prim}}N_q\simeq \R^{\ell} \to U:=\prod_{q\in \Ub^+_{d,\prim}}U_q \overset{\chi^+}{\to} \prod_{q\in \Ub^+_{d,\prim}}(U_q/N_q) \to \{0\}
\end{equation}
where $\chi^+((u_q)_{q\in \Ub^+_{d,\prim}})=(\chi_q(u_q))_{q\in \Ub^+_{d,\prim}}$ for all $(u_q)_{q\in \Ub^+_{d,\prim}} \in U$.

Recall that $U_q/N_q$ can be identified with $(\Wcal^{(q)},+)$.
Let us define $\Vcal^+_{m-1}=\prod_{q\in \Ub^+_{d,\prim}}\Vcal^{(q)}_{m-1}$,   $\Wcal^+_{m+1}=\prod_{q\in \Ub^+_{d,\prim}}\Wcal^{(q)}_{m+1}$,
and
$$
\Wcal^+=\prod_{q\in \Ub^+_{d,\prim}}\Wcal^{(q)}=\prod_{q\in \Ub^+_{d,\prim}}(\Vcal^{(q)}_{m-1}\oplus\Wcal^{(q)}_{m+1}).
$$
Note that $\Wcal^+ \simeq U/N= \prod_{q \in \Ub^+_{d, \prim}}(U_q/N_q)$, and we can  view $\chi^+$ as a group morphism from $U$ onto $(\Wcal^+,+)$.
Set $\hat{P}^+:=\prod_{q\in \Ub^+_{d,\prim}} \hat{P}_q$, where  $\hat{P}_q$ is the subgroup of $\U(\Vbb^{(q)})$ that preserves the flag \eqref{eq:flag:filtration}, and define
$$
\begin{array}{cccc}
\Rcal^+: &  \hat{P}^+  & \to &  \Aut(\Wcal^+)\\
         & (u_q)_{q\in \Ub^+_{d,\prim}} &  \mapsto &  (\Rcal_q(u_q))_{q\in \Ub^+_{d,\prim}}.
\end{array}
$$

Note that both $\rho_d(\PB_m)$ and $\rho_d(\PB_{m+1,n})$ are contained in $\hat{P}^+$, and the actions of  $\Rcal^+\circ\rho_d(\PB_m)$ and $\Rcal^+\circ\rho_d(\PB_{m+1,n})$  preserve  $\Vcal^+_{m-1}$ and $\Wcal^+_{m+1}$ respectively (see Lemma~\ref{lm:action:PB:m} and Lemma~\ref{lm:action:PB:m+1:n}).

\medskip

Recall that the  action of $\hat{P}_q$  on $U_q$ by conjugation satisfies $\chi_q(A\cdot * \cdot A^{-1})=\Rcal_q(A)(\chi_q(*))$, for all $A\in \hat{P}_q$ (see Proposition~\ref{prop:action:conj:P})

\begin{Proposition}\label{prop:orbits:basis:W}
Let $\tau$ and $\tau'$ be the elements of $\PB_n$  in Lemma~\ref{lm:U:q:inters:Gam:non:triv} (if exist). Let $\nu^+:=\chi^+(\rho_d(\tau)) \in \Vcal^+_{m-1}$ and $\nu'{}^+:=\chi^+(\rho_d(\tau')) \in \Wcal^+_{m+1}$.
Define
$$
\Sigma^+:=\{\Rcal^+\circ\rho_d(\alpha)(\nu^+), \; \alpha \in \PB_m\} \quad \text{ and } \quad  \Sigma'{}^+:= \{\Rcal^+\circ\rho_d(\alpha)(\nu'{}^+), \; \alpha \in \PB_{m+1,n}\}.
$$
Assume that $\gcd(d,k_1,\dots,k_m)=1$  if $m\geq 3$, and  $\gcd(d,k_{m+1},\dots,k_n)=1$  if $n-m\geq 3$.
Then  $\Span_\R(\Sigma^+)=\Vcal^+_{m-1}$,  $\Span_\R(\Sigma'{}^+)=\Wcal^+_{m+1}$, and therefore $\Span_\R(\Sigma^+\cup \Sigma'{}^+)=\Wcal^+$.
\end{Proposition}
\begin{proof}
Define $\Ub^-_{d, \prim}$ to be the set of primitive $d$-th root of unity $q$ such that $\Im(q) <0$, and
$$
\Vcal^-_{m-1}:=\prod_{d\in \Ub^-_{d,\prim}}\Vcal^{(q)}_{m-1}.
$$
Since $\Vcal_{m-1}=\Vcal^+_{m-1}\times \Vcal^-_{m-1}$, for all $\eta\in \Vcal_{m-1}$ we can write $\eta=(\eta^+,\eta^-)$, with $\eta^+\in \Vcal^+_{m-1}$ and $\eta^-\in \Vcal^-_{m-1}$.

Consider the element $\nu =(\nu_q)_{q\in \Ub_{d,\prim}}\in \Vcal_{m-1}$  where  $\nu_q:=\chi_q(\rho_q(\tau))$ for all $q\in \Ub_{d,\prim}$.
We have $\nu=(\nu^+,\nu^-)$ where $\nu^+=\chi^+(\rho_d(\tau))$.
Recall that $\hat{\rho}_d: \PB_n \to \prod_{q\in \Ub_{d,\prim}}\U(\Vbb^{(q)})$ is the morphism given by $\hat{\rho}_d(\alpha)=(\rho_q(\alpha))_{q\in \Ub_{d,\prim}}$ for all $\alpha\in \PB_n$, and $\Rcal: \hat{P}=\prod_{q\in \Ub_{d,\prim}} \hat{P}_q \to \Aut(\Wcal)$ is the morphism such that for all $A=(A_q)_{q\in \Ub_{d,\prim}} \in \hat{P}$ the action of $\Rcal(A)$ on the $\Wcal^{(q)}$-factor of $\Wcal$ is given by $\Rcal_q(A_q)$ (see \textsection~\ref{subsec:horo:product}).
Let
$$
\Sigma:=\{\Rcal\circ\hat{\rho}_d(\alpha)(\nu), \; \alpha \in \PB_m\} \subset \Vcal_{m-1}, \text{ and } V:=\Span_\C(\Sigma) \subset \Vcal_{m-1}.
$$
By definition $V$  is $\Rcal\circ\hat{\rho}_d(\PB_m)$-invariant and contains the vector $\nu=(\nu_q)_{q\in \Ub_{d,\prim}}$ where $\nu_q\neq 0$ for all $q\in \Ub_{d,\prim}$. It follows from Proposition~\ref{prop:irred:rep:PB:m} that  $V=\Vcal_{m+1}$, that is $\Span_\C(\Sigma)=\Vcal_{m-1}$.
%Similarly, it follows from Proposition~\ref{prop:irred:rep:EB:m+1} that $W=\Wcal_{m+1}$.

Recall that for each $q\in \Ub_{d,\prim}$, we have specified a basis $\Gcal^{(q)}$ of $\Vbb^{(q)}$.
We then have a distinguished basis $\Hcal^{(q)}=(g^*_1\otimes w, \dots, g^*_{m-2}\otimes w, g^*_{m+2}\otimes w,\dots,g^*_{n-1}\otimes w)$ of $ \Hom(\Wbb^{(q)},\Lbb^{(q)})\simeq \Wcal^{(q)}$, where $(g^*_1,\dots,g^*_{m-2},g^*_{m+2},\dots,g^*_{n-1})$ is the basis of $\Wbb^{(q)*}$ dual to  $(g_1,\dots,g_{m-2},g_{m+2},\dots,g_{n-1})$.
Observe that the coordinates of $\nu_q$ in the basis $\Hcal^{(q)}$ are the complex conjugates of the coordinates of $\nu_{\bar{q}}$ in the basis $\Hcal^{(\bar{q})}$.
Moreover, for all  $\alpha \in \PB_m$ the   matrix of $\Rcal_q\circ\rho_q(\alpha)$  in the basis $\Hcal^{(q)}$ is the complex conjugate  of the matrix of and $\Rcal_{\bar{q}}\circ\rho_{\bar{q}}(\alpha)$ in the basis $\Hcal^{(\bar{q})}$.
As a consequence, we can identify $\Vcal^+_{m-1}$ and $\Vcal^-_{m-1}$ with $\C^{\ell\cdot(m-2)}$ in such a way that for all $\eta=(\eta^+,\eta^-) \in \Sigma$, we have $\eta^-=\bar{\eta}^+$.

Consider now a vector $\xi \in \Vcal^+_{m-1}$. Let $\hat{\xi}:=(\xi,\bar{\xi}) \in \Vcal_{m-1}$. Since $\Span_\C(\Sigma)=\Vcal_{m-1}$, there exist $\eta_1, \dots,\eta_s \in \Sigma$, and $a_1,\dots,a_s \in \C$ such that $\hat{\xi}=\sum_i a_i\cdot\eta_i$. We then have $\xi=\sum_ia_i\cdot \eta^+_i$, and $\bar{\xi}=\sum_i a_i\cdot\eta^-_i=\sum_i a_i\cdot \bar{\eta}^+_i$. Therefore,
$$
\xi=\sum_{i=1}^s a_i\cdot\eta^+_i=\sum_{i=1}^s \bar{a}_i\cdot\eta^+_i
$$
which implies that
$$
\xi=\sum_{i=1}^s \real(a_i)\cdot\eta^+_i.
$$
Since $\eta^+_i \in \Sigma^+$, we conclude that $\Vcal^+_{m-1}=\Span_\R(\Sigma^+)$. The fact that $\Wcal^+_{m+1}=\Span_\R(\Sigma'{}^+)$ follows from the same argument.
\end{proof}

Proposition~\ref{prop:orbits:basis:W} means that $\Sigma^{+}\cup\Sigma'{}^+$ contains an $\R$-basis $\Fcal^+$ of $\Wcal^+$. Let $\Lambda^+$ denote the Abelian subgroup generated by $\Fcal^+$ in $\Wcal^+$. Then $\Lambda^+$ is  a lattice in $\Wcal^+$, and we have
\begin{Corollary}\label{cor:Gam:inters:U:proj:lattice}
Under the assumption of Theorem~\ref{th:inters:w:horosp:lattice},  $\chi^+(\Gamma\cap U)$  contains the lattice $\Lambda^+\subset \Wcal^+\simeq U/N$.
\end{Corollary}

To prove that $\Gamma\cap U$ is a lattice of $U$, it remains to show that
$\Gamma\cap N$ is a lattice of $N$.
Recall that for each $q\in \Ub_{d,\prim}$, $U_q$ is isomorphic to $\C^{n-4}\times\R$ and $N_q$ is identified with $\{0\}\times \R \subset U_q$. The group law in $U_q$ is given by
$$
(x,s)\cdot (y,t)=(x+y, s+t+\mu\mathrm{Im}({}^tx\cdot H_q^{-1}\cdot \bar{y}))
$$
where $H_q$ is the matrix of the intersection form on $\Wbb^{(q)}$ in the basis $(g_1,\dots,g_{m-2},g_{m+2},\dots,g_{n-1})$.
In particular, we have $(x,t)^{-1}=(-x,-t)$, and
\begin{equation}\label{eq:commutator:in:U}
[(x,s), (y,t)]=(x,s)\cdot(y,t)\cdot(-x,-s)\cdot(-y,-t)=(0,2\mu\Im({}^tx\cdot H_q^{-1}\cdot\bar{y})) \in N_q.
\end{equation}
%Recall that $N_q$ is identified with $\{0\}\times\R$ in this description.
For all $q\in \Ub_{d,\, {\rm prim}}$, let us define
$$
\omega_q(x,y):=2\mu\cdot\Im({}^tx\cdot H_q^{-1}\cdot\bar{y}).
$$
Define
$$
\Pi^+:=\{(\omega_q(\eta^+_{q},{\eta}'{}^+_{q}))_{q\in \Ub^+_{d,\prim}} \in \R^\ell, \; \eta^+, \eta'{}^+ \in \Lambda^+ \} \subset \R^\ell.
$$
Corollary~\ref{cor:Gam:inters:U:proj:lattice} then implies the following
\begin{Lemma}\label{lm:Gam:inters:N}
The set $\Pi^+$ is contained in the intersection $\Gamma\cap N$, where $N=\prod_{q\in \Ub^+_{d, \prim}}N_q$ is identified with $\R^\ell$.
\end{Lemma}

Our goal now is to show
\begin{Proposition}\label{prop:Gam:inters:N:lattice}
Under the assumption of Theorem~\ref{th:inters:w:horosp:lattice}, the set $\Pi^+$ contains a lattice of $\R^\ell$.
\end{Proposition}
For the proof of Proposition~\ref{prop:Gam:inters:N:lattice}, we first need the following lemmas.
\begin{Lemma}\label{lm:coeff:H:q:in:K:d}
For all $q\in \Ub_{d,\, {\rm prim}}$, we have  $\imath\cdot H_q \in \Mb_{n-4}(K_d)$, where $K_d=\Q(\zeta_d)$.
\end{Lemma}
\begin{proof}
Since the coefficients of $H_q$ is given by $\langle g_i, g_j\rangle$, with $i, j \in \{1,\dots,m-2,m+2,\dots,n-1\}$, the lemma follows immediately from Theorem~\ref{th:Menet:generator:set}.
\end{proof}

Let $L_d:=K_d\cap\R$. Then $L_d=\Q(\cos(\frac{2\pi}{d}))$ is totally real number field of degree $\ell$ over $\Q$.
\begin{Lemma}\label{lm:values:inters:form:in:L}
For all $q\in \Ub_{d,\prim}$, $x, y \in (K_d)^{n-4}$, we have
$\omega_q(x,y)\in L_d$.
\end{Lemma}
\begin{proof}
Since $\omega_q(x,y)$ is a real number, we only need to show that $\omega_q(x,y)\in K_d$. We can write
\begin{align*}
\omega_q(x,y) &=(1-q)(1-\bar{q})\cdot\imath\cdot({}^ty\cdot H_q^{-1}\cdot\bar{x} - {}^tx\cdot H_q^{-1}\cdot\bar{y})\\
&= (1-q)(1-\bar{q})\cdot ({}^ty\cdot(\imath\cdot H_q^{-1})\cdot\bar{x}- {}^tx\cdot(\imath\cdot H_q^{-1})\cdot\bar{y}).
\end{align*}
Since $\imath\cdot H_q$ is a matrix with coefficients in $K_d$ and the complex conjugation preverses $K_d$ the lemma follows.
\end{proof}

\begin{Lemma}\label{lm:values:H:conjugate}
For all $q \in \Ub_{d,\prim}$, $x, y \in (K_d)^{n-4}$, and $\sigma\in\mathrm{Gal}(K_d/\Q)$, we have
$$
\omega_{\sigma(q)}(\sigma(x),\sigma(y))=\sigma(\omega_q(x,y)).
$$
\end{Lemma}
\begin{proof}
From Theorem~\ref{th:Menet:generator:set} we get that $\imath\cdot H_q\in \Mb_{n-4}(K_d)$, and $\imath\cdot H_{\sigma(q)}=\sigma(\imath \cdot H_q)$. The lemma follows from the fact that
$$
\omega_q(x,y)=(1-q)(1-\bar{q})\cdot ({}^ty\cdot(\imath\cdot H_q^{-1})\cdot\bar{x}- {}^tx\cdot(\imath\cdot H_q^{-1})\cdot\bar{y})
$$
and that complex conjugation comutes with $\sigma$.
\end{proof}

\subsubsection{Proof of Proposition~\ref{prop:Gam:inters:N:lattice}}\label{subsec:Gam:inters:N:lattice}
\begin{proof}
Recall that by assumption we have $n\geq 5$. 
Let $\{\nu_1,\dots,\nu_{2\ell(n-4)}\}$ be the vectors in the $\R$-basis $\Fcal^+$ of $\Wcal^+$. 
Let us  pick a primitive $d$-th root of unity $q\in \Ub^+_{d,\prim}$. Let $\Fcal_{q}^+:=\{\nu^+_{1,q},\dots,\nu^+_{2\ell(n-4),q}\}$, where $\nu_{i,q}$ is the projection of $\nu_i$ in the $\Wcal^{(q)}$ factor of $\Wcal^+$.
We use the basis $\Hcal^{(q)}=(g^*_1\otimes w, \dots, g^*_{m-2}\otimes w,g^*_{m+2}\otimes w, \dots,g^*_{n-1}\otimes w)$ of $\Wcal^{(q)}$ to identify $\nu^+_{i,q}$ with a vector in $(K_d)^{n-4}$.

\begin{Claim}\label{clm:Q:basis:K:d:n-4}
The family $\Fcal^+_q$ is a basis of $(K_d)^{n-4}$ over $\Q$.
\end{Claim}
\begin{proof}
We have $[K_q:\Q]=\varphi(d)=2\ell$. Therefore $(K_d)^{n-4}$ is a $\Q$-vector space of dimension $2\ell(n-4)$. All we need to show is that the family $\{\nu^+_{1,q},\dots,\nu^+_{2\ell(n-4),q}\}$ is independent over $\Q$. Assume that there is a non-trivial relation $\sum_{i=1}^{2\ell(n-4)} a_i\nu^+_{i,q}=0$, with $a_i\in \Q, \; i=1,\dots,2\ell(n-4)$. For all $q'\in \Ub^+_{d,\prim}$ there is a unique Galois automorphism $\sigma\in \mathrm{Gal}(K_d/\Q)$ such that $q'=\sigma(q)$. By construction, we also have $\nu^+_{i,q'}=\sigma(\nu^+_{i,q})$. Thus for all $q'\in \Ub^+_{d,\prim}$, we have
$$
\sum_{i=1}^{2\ell(n-4)}a_i\nu^+_{i,q'}=0,
$$
which means that $\sum_{i=1}^{2\ell(n-4)}a_i\nu^+_i=0$. Since $\{\nu^+_1,\dots,\nu^+_{2\ell(n-4)}\}$ is an $\R$-basis of $\Wcal^+$, this is impossible. The claim is then proved.
\end{proof}

\begin{Claim}\label{clm:values:Im:H:on:Lambda:q}
Let $\Lambda^+_q \subset (K_d)^{n-4}$ denote the subgroup  generated by $\Fcal^+_q$.
Then the set
$$
\Pi^+_q:=\{\omega_q(\nu^+_{q},\nu'{}^+_{q}), \nu^+_q, \nu'_q{}^+ \in \Lambda^+_q\} \subset L_d
$$
contains a  basis of $L_d$ over $\Q$.
\end{Claim}
\begin{proof}
We consider $K_d^{n-4}$ and an $L_d$-vector space and $\omega_q(.,.)$ as an $L_d$-bilinear form on $(K_d)^{n-4}$ taking values in $L_d$.
Since $H_q$ is non-degenerate, and $\Fcal_q^+$  contains an $L_d$-basis of $(K_d)^{n-4}$ by Claim~\ref{clm:Q:basis:K:d:n-4}, there exist $\nu^+_{i,q}, \nu^+_{j,q}$ in $\Fcal^+_q$ such that $a_q:=\omega_q(\nu^+_{i,q},\nu^+_{j,q})\in L_d^*$.

Let $(\lambda_{1},\dots,\lambda_{\ell})$ be a basis of $L_d$ over $\Q$.
Since $\{\nu^+_{1,q},\dots,\nu^+_{2\ell(n-4),q}\}$ is a basis of $K_d^{n-4}$ over $\Q$,
for all $s\in \{1,\dots,\ell\}$, we can write
$$
\lambda_s\cdot\nu^+_{i,q}=\sum_{k=1}^{2\ell(n-4)}r_{k,s}\cdot\nu^+_{k,q},
$$
with $(r_{1,s},\dots,r_{2\ell(n-4),s}) \in \Q^{2\ell(n-4)}$.
Thus by multiplying $\lambda_{s}$ by an integer, we can assume that $\lambda_{s}\cdot\nu^+_{i,q}$ is an element of $\Lambda^+_q$.
Note that if one multiplies each member of the family $(\lambda_1,\dots,\lambda_\ell)$ by an integer, this family remains a basis of $L_d$ over $\Q$.
We now have
$$
\omega_q(\lambda_s\cdot\nu^+_{i,q},\nu^+_{j,q})=\lambda_s\cdot\omega_q(\nu^+_{i,q},\nu^+_{j,q})=\lambda_{s}\cdot a_q.
$$
Since $(\lambda_{1},\dots,\lambda_{\ell})$ is a basis of $L_d$ over $\Q$, so is $(a_q\cdot\lambda_1,\dots,a_q\cdot\lambda_\ell)$.
\end{proof}

Note that every element of $\mathrm{Gal}(K_d/\Q)$ leaves invariant $L_d$ (since the complex conjugation preserves $K_d$ and commutes with all elements of $\mathrm{Gal}(K_d/\Q)$). Moreover, every element of $\mathrm{Gal}(L_d/\Q)$ is the restriction of some elements of $\mathrm{Gal}(K_d/\Q)$ to $L_d$.

Recall that
$$
\Pi^+:=\{(\omega_{q'}(\eta_{q'},\eta'_{q'}))_{q'\in \Ub^+(d, \prim)}, \; \text{ with } \; \eta=(\eta_{q'}) \in \Lambda^+, \, \eta'=(\eta'_{q'}) \in \Lambda^+\} \subset L_d^{\ell}.
$$
Let $a_q$ and $(\lambda_1,\dots,\lambda_\ell)$ be as in Claim~\ref{clm:values:Im:H:on:Lambda:q}.
For $s=1,\dots,\ell$, we define $\kappa_s:=(\kappa_{s,q'})_{q'\in \Ub^+_{d,\prim}} \in L_d^\ell$ as follows: for each $q'\in \Ub^+_{d,\prim}$, $\kappa_{s,q'}=\sigma(\lambda_{s}\cdot a_q)$ where $\sigma$ is the element of $\mathrm{Gal}(K_d/\Q)$ such that $q'=\sigma(q)$.
Since each member of the family $\{a_q\cdot\lambda_{1}, \dots,a_q\cdot\lambda_\ell\}$ is equal to $\omega_q(\eta^+_q,\eta'_{q}{}^+)$ for some $\eta^+_q$ and $\eta'_q{}^+$ in the lattice $\Lambda^+_q$, it follows from Lemma~\ref{lm:values:H:conjugate} that  the vectors $\kappa_1,\dots,\kappa_\ell$ are contained in $\Pi^+$.

\begin{Claim}\label{clm:lattice:in:N}
The family $(\kappa_1,\dots,\kappa_\ell)$ is a basis of $\R^\ell$.
\end{Claim}
\begin{proof}
It is enough to show that $(\kappa_1,\dots,\kappa_\ell)$ is independent over $\R$.
Assume that there is a non-trivial relation $\sum_i c_i\cdot\kappa_i=0$. We can assume that $c_1=1$, and  $c_i\in L_d$ for all $i\geq 2$ (since $\kappa_i \in L_d^{\ell}$). For each $q'\in\Ub^+_{d,\prim}$, we have
$$
\kappa_{1,q'}=-\sum_{i=2}^\ell c_i\cdot\kappa_{i,q'}
$$
Let $\sigma\in\mathrm{Gal}(K_d/\Q)$ be the automorphism of $K_d$ that sends $q'$ to $q$. Applying $\sigma$ to the equality above, we get
$$
\kappa_{1,q}=-\sum_{i=2}\sigma(c_i)\cdot \kappa_{i,q}.
$$
We now remark that the set of $\sigma \in \mathrm{Gal}(K_d/\Q)$ such that $\sigma^{-1}(q)\in\Ub^+_{d,\prim}$ is in bijection with $\mathrm{Gal}(L_d/\Q)$. Taking the sum over all such $\sigma$'s, we get
$$
\ell\cdot \kappa_{1,q}=-\sum_{i=2}^\ell \hat{c}_i\cdot \kappa_{i,q},
$$
where
$$
\hat{c}_i=\sum_{\theta\in \mathrm{Gal}(L_d/\Q)}\theta(c_i) \in \Q.
$$
Since $(\kappa_{1,q},\dots,\kappa_{\ell,q})$ is independent over $\Q$, this is impossible, and the claim follows.
\end{proof}
Claim~\ref{clm:lattice:in:N} implies that the group generated by $\{\kappa_1,\dots,\kappa_\ell\}$ is a lattice in $\R^\ell$. Since $\{\kappa_1,\dots,\kappa_\ell\}$ is contained in $\Pi^+$, the proposition follows.
\end{proof}

\subsubsection{Proof of Theorem~\ref{th:inters:w:horosp:lattice}}\label{subsec:prf:th:inters:w:horosp:lattice}
\begin{proof}
Remark that $N\simeq \R^\ell$ is the center of $U$. By Proposition~\ref{prop:Gam:inters:N:lattice}, $\Gamma\cap N$ contains a lattice in $N$, while by Corollary~\ref{cor:Gam:inters:U:proj:lattice}, $\chi^+(\Gamma\cap U)$ contains a lattice of $U/N$. Therefore, $\Gamma\cap U$ contains a lattice of $U$.
\end{proof}

\subsection{Proof of Theorem~\ref{th:main:arithm}}\label{subsec:prf:th:main:arith}
\begin{proof}
By Lemma~\ref{lm:reduction:d:divides:sum:ki}, we can assume that $d \, | \, (k_1+\dots+k_n)$, and $n\geq 5$.
Recall that $G=\prod_{q\in \Ub^+_{d, \prim}}\SU(\Vbb^{(q)})$ and $\Gamma=\rho_d(\PB_n)\cap G$.
Since for all $q \in \Ub_{d,\prim}$, we have $\det\circ\rho_q(\PB_n) \subset \Ub_d$,  the subgroup $\Gamma$ has finite index in $\rho_d(\PB_n)$.

We first observe that it is enough to show that $\Gamma$ is commensurable to  $G(\Z)$. Indeed, by the Borel and Harish-Chandra's theorem, $G(\Z)$ is a lattice in $G$. Since $\Sp(\hX,\R)^T_d \simeq \prod_{q\in \Ub^+_{d, \prim}}\U(\Vbb^{(q)})$, we have $\Sp(\hX,\R)^T_d/G \simeq (\S^1)^\ell$. Thus $G(\Z)$ is also a lattice in $\Sp(\hX,\R)^T_d$. By definition, $G(\Z)$ is contained in $\Sp(\hX,\Z)^T_d$. Since $G(\Z)$ is a lattice in $\Sp(\hX,\R)^T_d$, so is $\Sp(\hX,\Z)^T_d$ (because $\Sp(\hX,\Z)^T_d$ is a discrete subgroup of $\Sp(\hX,\R)^T_d$). In particular, we have $[\Sp(\hX,\Z)^T_d:G(\Z)]$ is finite. Therefore $\rho_d(\PB_n)$ is commensurable to $\Sp(\hX,\Z)^T_d$ if $\Gamma$ is commensurable to $G(\Z)$.

By Corollary~\ref{cor:Zar:dense:d:divides:partial:sum}, we know that $\Gamma$ is Zariski dense in $G$.
By Theorem~\ref{th:inters:w:horosp:lattice}, $\Gamma\cap U$ contains a lattice $\Delta$ in $U$. To apply Theorem~\ref{th:Margulis:crit}, we need to show that $\Delta$ is irreducible in $G$. By definition, $\Delta$ is irreducible if for any  algebraic normal subgroup $G'$ of $G$ of infinite index, $\Delta\cap G'$ is finite (see \cite[Def. 1.13]{Benoist:survey}).
Since $G'$ is normal in $G$, for each $q\in \Ub^+_d$ the projection $G'_q$ of $G'$ in $\SU(\Vbb^{(q)})$ is an algebraic normal subgroup of $\SU(\Vbb^{(q)})$. Since $\SU(\Vbb^{(q)})$ is quasisimple (that is, its Lie algebra is simple), either $G'_q=\SU(\Vbb^{(q)})$, or $G'_q$ is contained in the center  of $\SU(\Vbb^{(q)})$. If $G'_q=\SU(\Vbb^{(q)})$ for all $q\in \Ub^+_{d, \prim}$ then $G'=G$. Hence for some $q\in \Ub^+_{d, \prim}$, $G'_q$ is contained in the center $\Zcal_q$ of $\SU(\Vbb^{(q)})$.

Let $u=(u_{q'})_{q'\in \Ub^+_{d, \prim}} \in \Delta\cap G'$. We then have $u_q \in \Zcal_q$. Note that $\Zcal_q$ is a finite group, while $u_q \in  U_q$ is  unipotent. Therefore, we must have $u_q=\id_{\Vbb^{(q)}}$.

By construction, for all $q'\in \Ub^+_{d,\prim}$, we have $u_{q'}=\sigma(u_q)$, where $\sigma$ is the element of $\mathrm{Gal}(K_d/\Q)$ such that $\sigma(q)=q'$ (here we abusively denote by $u_{q'}$ the matrix of $u_{q'}$ in some appropriate basis of $\Vbb^{(q')}$).
This implies that $u=e_G$, which proves that $\Delta$ is irreducible.

\medskip

By definition we have $\rk_\R(G)=\sum_{q\in \Ub^+_{d,\prim}}\rk_\R(\SU(\Vbb^{(q)}))$. It follows from Lemma~\ref{lm:0:mod:d:good} and Theorem~\ref{th:Menet:dim:signature} that the signature $(r_q,s_q)$ of the intersection form on $\Vbb^{(q)}$ satisfies $r_q\geq 1$ and $s_q\geq 1$ for all $q \in \Ub_{d,\prim}$. Thus we have
$$
\rk_\R(G) \geq |\Ub^+_{d,\prim}|=\ell.
$$
If $\ell \geq 2$ then we immediately get $\rk_\R(G)\geq 2$.
Otherwise, that is when $\ell=1$, we must have $d\in\{3,4,6\}$. In those cases, the condition $2 < \left\{k_1/d\right\}+\dots+\left\{k_n/d\right\} < n-2$ implies that $r_q\geq 2$ and $s_q\geq 2$. Therefore $\rk_\R(G)=\min\{r_q,s_q\} \geq 2$ as well.
We can now conclude by Theorem~\ref{th:Margulis:crit}.
\end{proof}

%*************************************************************
%*************************************************************
%*************************************************************

\end{document}